\documentclass[reqno,11pt]{amsart}
\usepackage{amsmath,amsfonts,amssymb,amsxtra,latexsym,amscd,enumerate,enumitem,amsthm,verbatim,here}
\usepackage{mathtools}
\usepackage{array}
\usepackage{booktabs}
\usepackage[table,dvipsnames]{xcolor}
\usepackage{cancel}

\usepackage[margin=1.2in]{geometry}
\numberwithin{equation}{section}

\usepackage{hyperref}
\definecolor{myrefred}{RGB}{180,0,0}
\definecolor{myciteblue}{RGB}{0,70,140} 
\hypersetup{
    colorlinks=true,
    pdfstartview=FitV,
    linkcolor=myrefred,      
    citecolor=myciteblue,    
    urlcolor=black
}
\definecolor{labelkey}{rgb}{0.6,0,0}

\newcommand{\pr}{\partial}

\providecommand{\abs}[1]{\ensuremath{\left\lvert #1 \right\rvert}}
\providecommand{\norm}[1]{\ensuremath{\left\lVert #1 \right\rVert}}
\newcommand{\of}[1]{\ensuremath{\left(#1\right)}}

\newcommand{\offf}[1]{\ensuremath{\left\{#1\right\}}}

\newcommand{\R}{\mathbb{R}}

\newcommand{\la}{\langle}
\newcommand{\ra}{\rangle}
\newcommand{\ncal}{\big\|\la x,v\ra^4 \gamma(t)\big\|_{L^\infty_{x,v}}^2}
\newcommand{\Ri}{{\mathcal{R}} }
\newcommand{\lo}{{\mathrm{low}} }
\newcommand{\hi}{{\mathrm{high}} }
\newcommand{\h}{\mathfrak{h} }
\newcommand{\tail}{\mathfrak{E}}

\def\ga{\gamma}

\def\pr{{\partial}}

\def\nab{\nabla}
\def\les{\lesssim}

\numberwithin{equation}{section} 

\newtheorem{theorem}{Theorem}

\newtheorem{lemma}{Lemma}[section]
\newtheorem{corollary}[lemma]{Corollary}

\newtheorem{proposition}[lemma]{Proposition}

\theoremstyle{remark}
\newtheorem{remark}[lemma]{Remark}
\def\oo{\infty}

\begin{document}
\title[Long-time dynamics of Vlasov--Hartree systems across the Coulomb threshold]{Long-time dynamics of Vlasov--Hartree systems across the Coulomb threshold}
\author{Wenrui Huang}
\address{Brown University}
\email{wenrui\_huang@brown.edu}
\author{Mengyi Xie}
\address{Yale University}
\email{mengyi.xie@yale.edu}

\thanks{We thank Nata\v{s}a Pavlovi\'{c} for emphasizing this question and motivating its study during her talk at the ICM Satellite Conference {\it New Methods in Evolution Partial Differential Equations}, Princeton University, July 2026.}

\begin{abstract}
We study the three-dimensional Vlasov–Hartree system with Coulomb and inverse-power interactions. For small, regular, localized data, we prove global well-posedness and determine long-time dynamics. The main result concerns the Coulomb case, where the coupling persists at leading order through reciprocal asymptotic corrections: the fermionic distribution scatters along logarithmically corrected free characteristics set by the asymptotic bosonic profile, while the bosonic wave gains a logarithmic phase set by the asymptotic fermionic density. We also treat longer- and shorter-range interactions. Our proof combines Hamiltonian methods for the Vlasov equation with  purely physical-space vector-field methods for dispersive equations. For stronger long-range interactions, we introduce a symplectic transformation and a regularized effective-phase cancellation, which remove leading nonintegrable interactions while preserving Hamiltonian structure and avoiding derivative loss.
\end{abstract}

\maketitle

\setcounter{tocdepth}{1}
\tableofcontents

\section{Introduction}
We study the long-time dynamics of the Vlasov--Hartree system in three spatial dimensions,
\begin{subequations}\label{main}
\begin{align}
    &\partial_t f+v\cdot\nabla_x f+F\cdot\nabla_v f=0,
    \qquad F(t,x)=-\nabla_x\Psi(t,x), \qquad \Psi=V*|u|^2,
    \label{eq:vlasov}\\
    &i\partial_tu-\Delta_xu=\phi u, \qquad\phi=V*\rho, \qquad \rho(t,x)=\int_{\mathbb{R}^3}f^2(t,x,v)\,dv, \label{eq:hartree}
\end{align}
\end{subequations}
where 
\begin{equation}
    V_\alpha(x)=\lambda |x|^{-\alpha}, \qquad \lambda\in \R\setminus\{0\}.
\end{equation}
Throughout the paper, $V$ without a subscript denotes $V_1$, the Coulomb potential. Here $f=f(t,x,v)$ is real-valued and $u=u(t,x)$ is complex-valued, with $(x,v)\in\mathbb R^3\times\mathbb R^3$. We work with the square root $f$ of the fermionic phase-space density, so that $f^2\geq 0$ is the physical density distribution function. This convention is convenient for the $L^2$-based analysis and is compatible with the transport structure, since $f$ and $f^2$ satisfy the same Vlasov equation. The function $u$ is the wave function of the bosonic component. 

The coupling in \eqref{main} is entirely through the interaction between the two species. The bosonic density $|u|^2$ generates the force field $F$ that transports the fermions, whereas the fermionic spatial density $\rho$ generates the potential $\phi$ governing the Schr\"odinger evolution of the bosons. Thus, although boson--boson and fermion--fermion interactions are absent, the two components remain nonlinearly coupled through their mutual mean field. The analytic study of this system \eqref{main} was recently highlighted by Nata\v{s}a Pavlovi\'{c} as a natural problem for PDE analysis.

The system has a Hamiltonian structure and conserves the fermionic and bosonic masses separately:
\begin{equation}\label{eq:mass_fermion}
    \mathcal{M}_F(f(t)):=\iint f^2(t,x,v)dxdv=\mathcal{M}_F(f(t=0)),
\end{equation}
and 
\begin{equation}\label{eq:mass_boson}
    \mathcal{M}_B(u(t)):=\int |u(t,x)|^2 dx= \mathcal{M}_B(u(t=0)).
\end{equation}
Moreover, the divergence-free structure of the Vlasov flow also preserves every integrable Casimir $\iint \beta(f)dxdv$, while the total momentum 
\begin{equation}
    \mathcal{P}(f,u):=\iint v f^2(t,x,v)dxdv+\Im \int \overline{u(t,x)}\nabla u(t,x)dx
\end{equation}
and the total energy
\begin{equation}
    \mathcal{E}(f,u):=\frac{1}{2}\iint  |v|^2 f^2(t,x,v)dxdv-\int |\nabla u(t,x)|^2 dx+\int \rho(x)\left(V_\alpha* |u|^2 \right)(t,x)dx
\end{equation}
are also conserved for sufficiently regular and decaying solutions. These identities reflect the absence of conversion between the two species: the mass of each component is conserved separately, whereas momentum and energy may be exchanged through the coupling, with the corresponding total quantities remaining conserved.


\subsection{Physical origin}
The Vlasov--Hartree system describes a zero-temperature mixture of a Bose--Einstein condensate and a degenerate Fermi gas in a regime where the two components are governed by different effective theories. The distinction originates in their quantum statistics. Fermions obey the Pauli exclusion principle, which prevents many particles from occupying the same one-particle state and leads, in an appropriate semiclassical regime, to a phase-space description. Bosons may instead macroscopically occupy a common state, making a wave-function description appropriate at the scale of the effective dynamics. The Vlasov--Hartree system is therefore a genuinely mixed classical--quantum model: the fermionic component is described by a kinetic distribution, while the bosonic component retains a dispersive quantum evolution. 

The system was rigorously derived by C\'ardenas, Miller, and Pavlovi\'c \cite{CKP25} from the many-body Schr\"odinger dynamics of a Bose--Fermi mixture with only interspecies interactions. Their argument proceeds through two effective limits. First, in a suitable mean-field regime, they prove that the many-body evolution is approximated by a coupled Hartree-type system for the fermionic one-particle density matrix and the bosonic condensate wave function. Second, a semiclassical limit is taken only in the fermionic component. Its Wigner transform converges to a Vlasov distribution, while the bosonic wave function remains governed by a Schr\"odinger equation. In particular, their work identifies and justifies a regime in which classical and quantum effective dynamics coexist and remain coupled at leading order.

More precisely, if $N$ and $M$ denote the numbers of bosons and fermions, respectively, and $m_B$ and $m_F$ their masses, the relevant scaling is
\[
    \frac{M}{N}=\frac{m_B}{m_F} =\hbar.
\]
As $\hbar\to0$, the bosons are therefore lighter and more numerous, whereas the fermions are heavier and fewer. This separation of the mass and particle-number scales allows the bosonic component to retain macroscopic quantum effects while the fermionic component approaches a classical kinetic regime. In the simultaneous mean-field and semiclassical limit 
\[
    \hbar\longrightarrow0,\qquad N,M\longrightarrow\infty,
\]
C\'ardenas, Miller, and Pavlovi\'c prove convergence to the coupled Vlasov–-Hartree dynamics, with quantitative error bounds in a nontrivial scaling window.

This derivation belongs to a broader program concerning the effective mathematics of Bose– Fermi mixtures. In related work on ultracold atomic mixtures, Cárdenas, Miller, Mitrouskas, and Pavlović \cite{CMMP25} rigorously established fermion-mediated interactions between bosons and an associated stability–instability transition. Although their work is spectral rather than dynamical, it similarly demonstrates that interspecies coupling can produce macroscopic effects even without direct intraspecies interactions.

To the best of our knowledge, the only previous work devoted to the Vlasov–Hartree PDE is the recent work of Cavanagh \cite{Cav26}, who established global existence for large data in a class of Lagrangian weak–mild solutions and obtained quantitative decay estimates in the Coulomb case. The present work addresses a different question: near vacuum, we construct global strong solutions and determine their precise asymptotic dynamics, including the leading-order corrections generated by the coupling.


\subsection{Relation with Hartree and Vlasov–Poisson; asymptotic regimes} The Vlasov--Hartree system couples, in a crossed form, the two basic self-consistent mean-field equations associated with classical transport and quantum dispersion. For comparison, the Vlasov-Riesz equation is 
\[
\partial_t f+v\cdot\nabla_x f +F_f\cdot\nabla_v f=0,\qquad F_f=-\nabla\left(V_\alpha *\rho_f\right),\qquad \rho_f=\int _{\R^3}f^2dv,
\]
and reduces to Vlasov-Poisson when $\alpha=1$. The corresponding Hartree equation is 
\[
i\partial_t u-\Delta u=\left(V_\alpha * |u|^2\right)u.
\]
In each model, the evolving component produces the field that acts back on itself. The Vlasov--Hartree system retains the two underlying coupling mechanisms -- a scalar potential acts on the Schr\"odinger wave, while its gradient drives the Vlasov transport—but exchanges their sources: the fermionic density $\rho_f$ generates the potential acting on $u$, whereas the bosonic density $|u|^2$ generates the force acting on $f$. 

This correspondence extends to the basic geometric symmetries of the two equations. Both Vlasov--Poisson and Hartree are time-reversible and invariant under spatial translations and rotations, and, more importantly for the present analysis, both possess a {\it Galilean symmetry}. On the kinetic side, Galilean boosts preserve the free-transport coordinate $x-tv$, while on the Schr\"odinger side the corresponding symmetry is encoded by its quantization, Galilean vector field $G=ix+2t\nabla$. Thus, the free-transport coordinates of Vlasov theory and the Galilean vector fields of Schr\"odinger theory are classical and quantum manifestations of the same underlying symmetry. This symmetry correspondence further explains why velocity space in the kinetic problem and frequency, or equivalently group velocity, in the dispersive problem play parallel roles in the asymptotic analysis. The correspondence extends further to spacetime geometry: both equations possess pseudo-conformal transformations \cite{FOPW23,LMR05,Tao09} and the closely related lens transforms \cite{HK26,Tao09}.

The analogy between Valsov--Poisson and Hartree has long been analytically productive. As emphasized in Pausader's ICM survey \cite{Pau26} (Section 1.4.2), it is particularly transparent near vacuum and helped guide the first precise description of modified scattering for the three-dimensional Vlasov-Poisson equation. The two forms of dispersion are microscopically different: free transport disperses by mixing particles with different velocities, whereas the Schr\"odinger flow disperses through oscillation and separation of group velocities. Nevertheless, in three dimensions they produce the same leading dilution of spatial density.

This common scaling can be seen directly. Introduce the free-transport profile
\begin{equation}\label{eq:intro_gamma}
   \gamma(t,q,v):=f(t,q+tv,v).
\end{equation}
At a point $x=ty$, a change of variables gives the exact identity
\[
t^3\rho(t,ty)=\int_{\R^3} \gamma^2\left(t,z,y-\frac{z}{t}\right)dz.
\]
Thus, if $\gamma(t)\to f_\infty$, then
\[
t^3\rho(t,ty)\,\to\, \rho_\infty(y):=\int_{\R^3} f_\infty^2 (z,y)dz.
\]
On the quantum side, define the self-similar Schr\"odinger profile
\[
g(t,y):=t^{\frac{3}{2}} e^{i\frac{|y|^2t}{4}} u(t,ty).
\]
If a suitable renormalization of $g$ converges to $h_\infty$, then 
\begin{equation}\label{eq:intro_u}
    t^3 |u(t,ty)|^2 \,\to \, |h_\infty(y)|^2.
\end{equation}
This variable $v$ in \eqref{eq:intro_gamma} labels the asymptotic rays of free transport, while $y=x/t$ in \eqref{eq:intro_u} labels the group velocities selected by the Schr\"odinger stationary phase. This is the precise sense in which velocity space on the kinetic side and the frequency variable on the dispersive side play parallel roles. 

\medskip

The homogeneity of $V_\alpha$ now gives, along $x=ty$, 
\begin{equation}\label{eq:intro:order}
    \phi(t,ty)\sim t^{-\alpha} \phi_\infty(y), \qquad F(t,ty)\sim t^{-1-\alpha} F_\infty(y),
\end{equation}
where 
\[
\phi_\infty=V_\alpha *\rho_\infty, \qquad F_\infty =-\nabla \left(V_\alpha* |h_\infty|^2\right).
\]
The extra factor $t^{-1}$ in the force does not make the kinetic interaction shorter range. Indeed, the profile $\gamma$ satisfies
\[
\partial_t \gamma=-F(t,q+tv)\cdot \nabla_v\gamma+tF(t,q+tv) \nabla_q \gamma.
\]
The leading term is therefore $tF(t,q+tv)$, which has size $t^{-\alpha}$, exactly the size of the potential in the equation for the Schr\"odinger profile as seen in \eqref{eq:intro:order}. This compensation is the basic reason that the two components cross the long-range threshold at the same exponent. 

The preceding calculation also identifies the leading effect of the coupling on the two free profiles. Formally, their dominant evolutions have the form
\[
i\partial_t g(t,y)\approx t^{-\alpha}\phi_\infty(y)g(t,y),\qquad \partial_t\gamma(t,q,v) \approx t^{-\alpha} F_\infty(v) \cdot\nabla_q \gamma(t,q,v).
\]
Thus the limiting fermionic density produces a phase correction for the Schr\"odinger profile, whereas the limiting bosonic density produces a shift of the free-transport coordinate $q$. Although these corrections have different forms, they have the same cumulative strength: both are obtained by integrating $t^{-\alpha}$ in time. This leads naturally to the common interaction factor
\begin{align*}
    \Lambda_\alpha(t):=\int_1^t s^{-\alpha} ds\,=\,
    \begin{cases}
        \frac{t^{1-\alpha}-1}{1-\alpha}, &\,\alpha\neq 1, \\
        \log(t), &\,\alpha=1.
    \end{cases}
\end{align*}
When $\alpha>1$, this factor converges and the nonlinear effects can be absorbed into ordinary scattering states. At $\alpha=1$, it grows logarithmically, and both components exhibit logarithmically modified scattering. When $\frac{1}{2}<\alpha<1$, it grows like $t^{1-\alpha}$, and the leading corrections are polynomial. The lower bound $\alpha>\frac{1}{2}$ has a different origin: after removing the leading $t^{-\alpha}$ interaction, the transformed equations contain quadratic remainders of size $t^{-2\alpha}$, which are integrable precisely when $\alpha>1/2$. The upper restriction $\alpha<2$ has a different, spatial origin: the force kernel $|\nabla V_\alpha(x)|\sim |x|^{-1-\alpha}$ is locally integrable precisely in this range. Notably, the same thresholds $\alpha=\frac{1}{2}$ and $\alpha=2$ appear in the corresponding theories for Vlasov equations with Riesz interactions \cite{HK24,HP26} and for the Hartree equation \cite{GV00I,GV00II,NO92,HN01}, further reflecting the structural analogy between the classical Vlasov–Poisson and quantum Hartree dynamics.

This comparison explains both the unity and the difficulty of the problem. In the scalar Hartree and Vlasov--Riesz equations, the asymptotic field is determined by the scattering state of the same component. Here, by contrast, neither correction closes on its own. The Vlasov component is governed by transport and phase-space mixing, with its long-time behavior extracted through the characteristic flow and phase-space vector fields, whereas the Schr\"odinger component is controlled by dispersive and vector-field estimates. These two theories therefore cannot simply be applied independently: the asymptotic fields, the derivative bounds, and the corresponding corrections must be constructed in a coupled manner. This two-way closure is the main structural feature of the Vlasov--Hartree problem.


\subsection{Main results} We now state our main conclusions. The Coulomb interaction is the central case of the paper, and we first give its precise statement. 
\begin{theorem}[Coulomb modified scattering]\label{thm:main}
    Let $V(x)=V_1(x)=\lambda |x|^{-1}$, where $\lambda\in \R\setminus\{0\}$. There exist constants $\varepsilon_0>0$ and $\delta_0>0$ such that the following holds. Suppose that $f_0\in C^1_{x,v}$ and
\begin{equation}\label{eq:initial_smallness}
    \|u_0\|_{H^1}+\|\langle x\rangle^2u_0\|_{L^2}
    +\|\langle x\rangle^4\langle v\rangle^4f_0\|_{L^\infty_{x,v}} +\|\nabla_{x,v}f_0\|_{L^2_{x,v}\cap L^\infty_{x,v}}+\big\||x|\nabla_{x,v}f_0\big\|_{L^\infty_{x,v}} \leq \varepsilon_0.
\end{equation}
Then the Vlasov-Hartree system \eqref{main} admits a unique global strong solution 
\[
f\in C_{t,x,v}^1\big([0,\infty)\times\mathbb R^3_x\times\mathbb R^3_v\big), \qquad u\in C\big([0,\infty);H^1(\mathbb R^3)\big).
\]
Moreover, there exist asymptotic profiles
\[
f_\infty\in L^\infty(\mathbb R^3_x\times\mathbb R^3_v),
\qquad h_\infty\in L^\infty(\mathbb R^3),
\]
such that, upon defining
\begin{equation}\label{eq:asymptotic-fields}
\rho_\infty(v):=\int_{\mathbb R^3}f_\infty^2(x,v)dx,
\qquad  \phi_\infty := V*\rho_\infty, \qquad
F_\infty:=-\nabla\left( V*|h_\infty|^2\right),
\end{equation}
we have, for $t\to \infty$,
\begin{equation}\label{eq:asy_f}
    f\big(t,x+tv-\log (t)F_\infty(v),v\big)=f_\infty(x,v)+\mathcal{O}_{L^\infty_{x,v}} \left(t^{-\delta_0}\right),
\end{equation}
and 
\begin{equation}\label{eq:asy_u}
    u(t,x)=t^{-3/2}e^{-i|x|^2/(4t)}
e^{-i\log (t)\phi_\infty\left(\frac{x}{t}\right)}
h_\infty\left(\frac{x}{t}\right) +\mathcal{O}_{L^\infty_x}\left(t^{-\frac{3}{2}-\delta_0}\right).
\end{equation}
\end{theorem}
Theorem \ref{thm:main} gives a joint asymptotic description of the classical and quantum components. In particular, the limiting fermionic state determines the logarithmic phase of the bosonic wave, while the limiting bosonic density determines the logarithmic correction to the fermionic trajectories. 

\medskip

The stronger long-range regime has the same two-way structure. Here the corrections grow polynomially and the proof requires higher regularity. 
\begin{theorem}[Long-range modified scattering]\label{thm:long_range_alpha}
Let
\[
V(x)=\lambda |x|^{-\alpha}, \qquad \lambda\in \R\setminus\{0\}, \qquad \frac{1}{2}<\alpha<1.
\]
There exists a constant $\varepsilon_0(\alpha,\lambda)$ such that the following holds. Suppose that  $f_0\in C^1_{x,v}$ and   \begin{equation}\label{eq:initial_smallness_small} 
    \begin{split}
 \|u_0\|_{H^1}+\|\la x\ra^3u_0\|_{L^2}+\|\la x\ra^4\la v\ra^4f_0\|_{L^{\oo}_{x,v}}+\|\la x\ra^3 \nabla_{x,v} f_0\|_{L^{2}_{x,v}\cap L^{\oo}_{x,v}}+\|\la x\ra\nabla^2_{x,v} f_0\|_{L^2_{x,v}}       \leq \varepsilon_0.
      \end{split}
    \end{equation}
Then the Vlasov-Hartree system \eqref{main} admits a unique global strong solution 
\[
f\in C_{t,x,v}^1\big([0,\infty)\times\mathbb R^3_x\times\mathbb R^3_v\big), \qquad u\in C\big([0,\infty);H^1(\mathbb R^3)\big).
\]
Moreover, there exist asymptotic profiles
\[
f_\infty\in L^\infty(\mathbb R^3_x\times\mathbb R^3_v),
\qquad h_\infty\in L^\infty(\mathbb R^3),
\]
such that, upon defining
\begin{equation}\label{eq:asymptotic-fields-long}
\rho_\infty(v):=\int_{\mathbb R^3}f_\infty^2(x,v)dx,
\qquad  \phi_\infty := V*\rho_\infty, \qquad
F_\infty:=-\nabla\left( V*|h_\infty|^2\right),
\end{equation}
we have, for $t\to \infty$,
\begin{equation}\label{eq:asy_f-long}
    f\big(t,x+tv- \frac{t^{1-\alpha}}{1-\alpha}F_\infty(v),v\big)=f_\infty(x,v)+\mathcal{O}_{L^\infty_{x,v}} \left(t^{1-2\alpha}\right),
\end{equation}
and 
\begin{equation}\label{eq:asy_u-long}
    u(t,x)=t^{-3/2}e^{-i|x|^2/(4t)}
e^{-i \frac{t^{1-\alpha}}{1-\alpha} \phi_\infty\left(\frac{x}{t}\right)}
h_\infty\left(\frac{x}{t}\right) +\mathcal{O}_{L^\infty_x}\left(t^{-\frac{3}{2}+\frac{3}{4}(1-2\alpha)}\right).
\end{equation}
\end{theorem}

\medskip

For $\alpha>1$, the interaction factor is finite and no long-range correction is needed.   Let\[
    n_\alpha=
    \begin{cases}
        1,&1<\alpha<\frac{3}{2},\\
        2,&\frac{3}{2}\leq\alpha<2.
    \end{cases}
\] 
\begin{theorem}[Short-range scattering]\label{thm:short_range}
Let
\[
V(x)=\lambda |x|^{-\alpha},
\qquad
\lambda\in\mathbb R\setminus\{0\},
\qquad 1<\alpha<2.
\]
There exists  a constant $\varepsilon_0(\alpha,\lambda)>0$ such that the following holds. Suppose that $f_0\in C^1_{x,v}$ and
\begin{equation}\label{eq:initial-short-range}
    \begin{split}
         \|u_0\|_{H^{n_{\alpha}}}+\|\langle x\rangle^2u_0\|_{L^2}
    +\|\langle x\rangle^4\langle v\rangle^4f_0\|_{L^\infty_{x,v}} +\|\nabla_{x,v}f_0\|_{L^2_{x,v}\cap L^\infty_{x,v}}\leq \varepsilon_0.
    \end{split}
\end{equation}    
Then \eqref{main} admits a unique global strong solution 
\begin{equation*}
    \begin{split}
        f\in C_{t,x,v}^1([0,\oo)\times \R^3_x\times\R^3_v), \qquad u\in C([0,\oo);H^{n_{\alpha}}).
    \end{split}
\end{equation*}
Moreover, there exist
\[
f_+\in L^2(\mathbb R^3_x\times\mathbb R^3_v)\cap
L^\infty(\mathbb R^3_x\times\mathbb R^3_v),
\qquad
u_+\in H^{n_{\alpha}}(\mathbb R^3),
\]
such that, for every $t\geq 1$,
\begin{equation}\label{eq:short_range_scattering_f}
\big\|f(t,x+tv,v)-f_+(x,v)\big\|_{L^2_{x,v}\cap L^\infty_{x,v}}
\lesssim_{\alpha,\lambda}
\varepsilon_0^3 t^{1-\alpha},
\end{equation}
and
\begin{equation}\label{eq:short_range_scattering_u}
\big\|u(t)-e^{-it\Delta}u_+\big\|_{H^{n_{\alpha}}}
\lesssim_{\alpha,\lambda}
\varepsilon_0^3 t^{1-\alpha}.
\end{equation}
\end{theorem}
Taken together, the three theorems describe the Vlasov-Hartree dynamics across the Coulomb threshold. Their unity lies in the common interaction factor $\Lambda_\alpha$ and in the two-way generation of the limiting fields. Their proofs are necessarily regime-dependent: the Coulomb case is borderline, and the method can be generalized to handle the short-range case, while the stronger long-range case requires nonlinear corrections of both components. 


\subsection{Previous work on Vlasov--Poisson and Hartree} 
The global Cauchy theory for the three-dimensional Vlasov--Poisson equation is classical; see, among many works, \cite{BardosDegond1985,LionsPerthame1991,Pfaffelmoser1992}. Its precise large-time behavior is substantially more delicate. \cite{ChoiKwon2016} established that linear scattering is not the correct asymptotic description. Following this, Ionescu, Pausader, Wang, and Widmayer \cite{IPWW22} recently identified an explicit and robust asymptotic dynamics in terms of the limiting field profile. The corresponding modified scattering operators were constructed by \cite{FOPW23}. For inverse-power interactions, the Vlasov-Riesz problem exhibits the same threshold at $\alpha=1$. \cite{HK24} proved (modified) scattering in the regime $1-\delta<\alpha<1$ and $1<\alpha<2$, while \cite{HP26} established modified scattering for $\frac{1}{2}<\alpha<1$ through a Lagrangian construction of modified characteristic flows. We further refer to \cite{PW21,PWY22,HPS24,BV26,HK26,KW25,IacobelliRossiWidmayer2026,SchlueTaylor2025} for related advances and extensions in the asymptotic analysis of Vlasov-type systems. These developments show that even near vacuum, obtaining a complete asymptotic description of a collisionless Coulomb system requires substantially more than global existence and decay.

The Coulomb Hartree equation is likewise long-range. Modified wave operators and large-time asymptotics were developed in the classical works of Ginibre and Ozawa \cite{GO93} and Hayashi and Naumkin \cite{HN98}. Kato and Pusateri \cite{KP11} later gave a Fourier-space proof based on stationary phase and space-time resonance ideas. More recently, \cite{VanHoose2024} provided a new way of establishing modified scattering using testing by wave packets. For general inverse-power potentials, the Hartree scattering theory exhibits the same transition at $\alpha=1$. In the short-range regime $1<\alpha<2$, solutions scatter to free Schr\"odinger evolutions \cite{NO92}, whereas for $\frac{1}{2}<\alpha<1$, modified scattering requires a corresponding polynomial phase correction \cite{GV00I,GV00II,HN01}. 


\subsection{Ideas of the proof} The proofs exploit a common physical-space description of Schr\"odinger waves and Vlasov particles along free rays, with the required corrections determined by the range of $\alpha$. 

\medskip

\paragraph{\bf The Coulomb case} A \emph{physical-space approach} to Hartree asymptotics is developed by Pausader and the second author in the forthcoming work \cite{PX}, which is the main ingredient of Theorem \ref{thm:main}. This argument is organized around the Galilean vector field, which captures the propagation of Schr\"odinger waves along free rays. In particular, \cite{PX} proves the relevant Hartree asymptotics without passing to a Fourier-space evolution equation. The physical-space estimates developed there provide an important dispersive input for our analysis.

The classical arguments of \cite{HN98,KP11} are not directly available for \eqref{main}. In the scalar Hartree equation, the leading evolution of the Schr\"odinger profile closes on that profile itself. In \eqref{main}, by contrast, the potential acting on the Schr\"odinger component is generated by the kinetic density. A direct Fourier-space adaptation would therefore require a simultaneous analysis of the Schr\"odinger oscillations and the phase-space transport profile. A central input that allows us to avoid this difficulty is the physical-space approach developed by Pausader and the second author in the forthcoming work \cite{PX}. Its formulation is particularly well suited to the Vlasov characteristic analysis, since both sides are organized along free rays. 

\medskip

\paragraph{\bf The long-range regime} When $\frac{1}{2}<\alpha<1$, the interaction along free rays has size $t^{-\alpha}$, and its accumulation produces terms of order $t^{1-\alpha}$. Our analysis is organized around a pair of reciprocal, ray-adapted corrections: a symplectic action correction for the Vlasov component generated by the bosonic potential, and a frequency-balanced phase correction for the Hartree component generated by the fermionic potential. Both are constructed from the evolving solution and they remove the nonintegrable terms already in the global estimates, while the same transformed variables subsequently identify the modified scattering dynamics. 

On the Vlasov side, we integrate the bosonic potential along each free ray and incorporate the resulting action into a time-dependent generating function. The induced symplectic transformation converts the Hamiltonian exactly into the difference 
\[
\mathcal K(t,x,v)=\Psi\big(t,x+tv-\Gamma(t,v)\big)-\Psi(t,tv).
\]
This difference gains one spatial derivative of the force together with the displacement from the reference position, replacing the $t^{-\alpha}$ coefficient by terms with the integrable weights. Because the cancellation is built into a symplectic change of variables, the Hamiltonian transport structure is preserved and can be used directly to propagate weighted phase-space derivatives. The same coordinates yield the optimal density decay and the convergence of the transformed distribution, without requiring a separate construction of finite-and infinite-time modified wave operators as in \cite{HP26}. 

\smallskip

On the Hartree side, directly integrating the full rescaled fermionic potential into the phase would lose a derivative, as noticed in \cite{HN01}. We instead introduce the scale-dependent multiplier
\[
J(t):=\left(1-t^{-1}\Delta\right)^{-\frac{\alpha}{2}}
\]
and define the corrected phase via $J(t)\phi(t)$. Above the transition frequency, $J(t)$ supplies the missing derivatives, whereas below the transition frequency, $\operatorname{Id}-J(t)$ produces a faster-decaying remainder. The transition scale is chosen so that these two effects balance at the integrable weight $t^{-1-\frac{\alpha}{2}}$. In contrast with the auxiliary-evolution construction of Hayashi--Naumkin \cite{HN01}, the phase here is a single explicit functional of the evolving potential, and no auxiliary equation or separate lower-regularity propagation is required. 

Crucially, the renormalized profile $\h$ is not merely a device for closing the bootstrap: the evolution of $|\h|^2$ has an extra cancellation, hence yields the sharper convergence needed to identify the limiting force correcting the Vlasov trajectories. Conversely, since $J(t)$ approaches the identity sufficiently rapidly, the regularization changes only the bounded part of the Hartree phase, whose leading polynomial term is determined by the limiting fermionic density. Thus the two constructions close the reciprocal asymptotic coupling as well as the global argument.

\medskip

\paragraph{\bf The short-range regime} For $1<\alpha<\frac{3}{2}$, the interaction is integrable in time, and the physical-space and characteristic estimates used in the Coulomb case extend with only minor modifications, yielding ordinary scattering for both components. Extending the argument to the full range $\frac{3}{2}\leq \alpha<2$ requires the additional assumption that $u_0\in H^2$.


\subsection{Organization of the paper} Section \ref{sec:preliminary} collects the analytic tools, including estimates for Riesz potentials, Schr\"odinger vector fields and Hamiltonian transport. 

Sections \ref{sec:GWP} and \ref{sec:asymptotic} are devoted to the Coulomb interaction. Section \ref{sec:GWP} proves the local and global well-posedness and global a priori estimates. Section \ref{sec:asymptotic} constructs the asymptotic bosonic force and fermionic potential and proves the logarithmically modified scattering in Theorem \ref{thm:main}.

The stronger long-range regime $\frac{1}{2}<\alpha<1$ is considered in Sections \ref{sec:GWP_small} and \ref{sec:asymptotic_small}. Section \ref{sec:GWP_small} establishes global control of the coupled evolution, while Section \ref{sec:asymptotic_small} constructs the asymptotic fields and proves the polynomially modified scattering in Theorem \ref{thm:long_range_alpha}.

Section \ref{sec:short_range} treats $1<\alpha<2$ and proves Theorem \ref{thm:short_range}. 


\subsection*{Acknowledgment}
W. Huang was partially supported by NSF grant DMS-2452275.


\section{Preliminaries}\label{sec:preliminary}
In this section, we collect the analytic tools used throughout the paper. We first fix the notation and record the Lorentz-space estimates needed for the Coulomb interaction. We then recall the basic estimates for the Schr\"odinger and Hamiltonian transport flows.
\subsection{Notation}\label{subsec:notation}
We begin by fixing some notation. All function spaces without an explicit underlying domain are understood to be over $\mathbb{R}^3$. We write $p'$ for the H\"older conjugate of $p$, with the convention $1/\infty=0$. Let $I\subseteq\mathbb{R}$ be a time interval. We use the standard mixed-norm notation, with norms taken successively from right to left. For example,
\begin{equation}\label{eq:mixed-norm-notation}
\|h\|_{L_t^pL_x^qL_v^r(I\times\mathbb{R}^3\times\mathbb{R}^3)}
    :=\left(  \int_I\left[\int_{\mathbb{R}^3} \left(  \int_{\mathbb{R}^3}|h(t,x,v)|^r\,dv\right)^{q/r}dx
    \right]^{p/q} dt  \right)^{1/p},
\end{equation}
with the usual modifications when one of the exponents is infinity. We use analogous notation for norms involving fewer variables, such as $L_t^pL_x^q$ and $L_x^qL_v^r$. When the underlying domains are clear, we omit them from the notation. Variables not appearing in the norm are regarded as parameters. For functions on phase space, we use the Poisson bracket
\[
\{f,g\}=\nabla_x f\cdot \nabla_vg-\nab_v f\cdot \nab_x g.
\]

We write
\[
\langle x\rangle=(1+|x|^2)^{1/2},\qquad \langle x,v\rangle=(1+|x|^2+|v|^2)^{1/2},
\]
and use 
\[
\nabla_{x,v}=(\nabla_x,\nabla_v), \quad \nabla^2_{x,v}=(\nabla_x^2,\nabla_x\nabla_v,\nabla_v^2).
\]
  For a multi-index $\nu$, $\partial^\nu$ denotes the corresponding spatial derivative. For multi-indices $\mu\leq\nu$, where the inequality is understood componentwise, we set
\[
\binom{\nu}{\mu}:=\prod_{j=1}^d\binom{\nu_j}{\mu_j}\,.
\]
The notation $A\lesssim B$ means that $A\leq CB$ for an absolute constant $C$, which may change from line to line. Dependence on additional parameters is indicated by a subscript. We write $A\lesssim_\delta B$, for example, when the constant may depend on $\delta$. We write $A\simeq B$ if $A\lesssim B$ and $B\lesssim A$. The notation $a+$ denotes $a+\delta$ for an arbitrarily small fixed $\delta>0$. 

If $X$ is a Banach space and $A=A(t)\geq 0$, we write, 
\[
R(t)=\mathcal{O}_X(A(t)), \quad \text{if } \|R(t)\|_X\lesssim A(t),
\]
uniformly over the relevant range of $t$. 

We use the notation $D=-i\nabla$. More generally, for a symbol $m:\mathbb R^3\to\mathbb C$, we denote by $m(D)$ the corresponding Fourier multiplier, defined by
\[
\widehat{m(D)f}(\xi)=m(\xi)\hat f(\xi).
\]
We note that if $m(-\xi)=\overline{m(\xi)}$, then $m(D)$ preserves real-valued functions; in particular, this holds whenever $m$ is real-valued and even.


\subsection{Lorentz-spaces and the Coulomb Kernel}
For $1\leq p,q\leq \infty$, we denote by $L^{p,q}$ the usual Lorentz space. Given a measurable function $h$ on $\mathbb{R}^3$, let $h^*$ denote its decreasing rearrangement. For $1\leq p<\infty$ and
$1\leq q<\infty$, we define
\[
    \|h\|_{L^{p,q}}:=\left( \int_0^\infty\big[s^{\frac{1}{p}}h^*(s)\big]^q\,\frac{ds}{s}\right)^{\frac{1}{q}},
\]
while
\[
    \|h\|_{L^{p,\infty}}:=\sup_{s>0}\, s^{\frac{1}{p}}h^*(s).
\]
In particular,
\begin{equation}\label{eq:lorentz_comparison}
        L^{p,p}=L^p,
    \qquad
    L^{p,q_1}\hookrightarrow L^{p,q_2}
    \quad\text{whenever }q_1\leq q_2.
\end{equation}
We use the following standard estimates. See O’Neil \cite{ONeil} for H\"older and convolution inequalities in Lorentz spaces, Peetre \cite{Peetre} for the Sobolev–Lorentz embedding, and Stein \cite{Stein} for Riesz potentials and singular integrals.
\begin{lemma}\label{lem:lorentz_estimates}
The following statements hold.
\begin{enumerate}
    \item For any  $p,q\in [1,\infty]$,   $\frac{1}{p}=\frac{1}{p_1}+\frac{1}{p_2}$ and $\frac{1}{q}=\frac{1}{q_1}+\frac{1}{q_2}$, we have  $$ \|f g\|_{L^{p,q}}\les \|f\|_{L^{p_1,q_1}}\|g\|_{L^{p_2,q_2}}. $$
    \item If $\frac{1}{p}=\frac{1}{q}-\frac{1}{3}$, with $1\leq q< 3$, then
    \[
    \|f\|_{L^{p,q}}\les \|\nab f\|_{L^q}.
    \]
    In particular, 
    \begin{equation}\label{eq:6_2_sobolev}
        \|f\|_{L^{6,2}}\lesssim \|\nabla f\|_{L^2}. 
    \end{equation}
    \item Suppose $1\leq p_1<p_2\leq \infty$, and $1\leq q_1,q_2\leq\infty$. Then for every $p_1<p<p_2$ and $1\leq q\leq \infty$,
    \begin{equation*}
    \begin{aligned}
        \|f\|_{L^{p,q}}\leq C\|f\|_{L^{p_1,q_1}}^{\theta}\| f\|^{1-\theta}_{L^{p_2,q_2}}
    \end{aligned}
    \end{equation*}
    where \[ \theta=\frac{1/p-1/p_2}{1/p_1-1/p_2}\in(0,1). \]
    \item Suppose $1<p_0,p_1,p<\infty$, $1\leq q,q_0,q_1\leq \infty$ and 
    \[
    \frac{1}{p_0}+\frac{1}{p_1}=1+\frac{1}{p}, \quad \frac{1}{q_0}+\frac{1}{q_1}\geq \frac{1}{q}.
    \]
    We have
    \begin{equation*}
    \begin{aligned}
          \| f*g  \|_{L^{p,q}} &\leq C\| f\|_{L^{p_0,q_0}}\| g\|_{L^{p_1,q_1}} .
    \end{aligned}
\end{equation*}
    \item The endpoint Riesz-potential estimates
    \[
    \left\|(-\Delta)^{-1}f\right\|_{L^{3,\infty}}\lesssim \|f\|_{L^1},\qquad \left\|(-\Delta)^{-1}f\right\|_{L^{\infty}}\lesssim \|f\|_{L^{\frac{3}{2},1}},
    \]
    and
    \[
     \left\||\nabla|^{-1}f\right\|_{L^{\infty}}\lesssim \|f\|_{L^{3,1}}
    \]
    hold whenever the right-hand sides are finite.
\end{enumerate}  
\end{lemma}
For later use, we record the direct estimates for density functions. If
\[
\rho_f(x):=\int_{\mathbb R^3}f^2(x,v)dv,
\]
then H\"older's inequality gives
\begin{equation}\label{eq:rho_est}
    \|\rho_f\|_{L^1_x}=\|f\|_{L^2_{x,v}}^2,\qquad\|\rho_f\|_{L^\infty_x}\lesssim\left\|\langle v\rangle^{\frac{3}{2}+}f\right\|_{L^\infty_{x,v}}^2.
\end{equation}
Consequently, for every $1<p<\infty$, Lemma \ref{lem:lorentz_estimates} gives
\begin{equation}
    \|\rho_f\|_{L^{p,1}_x}\lesssim_p\|f\|_{L^2_{x,v}}^{\frac{2}{p}}\left\|\langle v\rangle^{\frac{3}{2}+}f\right\|_{L^\infty_{x,v}}^{2-\frac{2}{p}}.
\end{equation}

\medskip

We next record the singular-integral estimates used below. Let
\[
\Ri_j:=\partial_j |\nabla|^{-1}, \qquad 1\leq j\leq3,
\]
denote the Riesz transforms.
\begin{lemma}\label{lem:Riedz_est}
    Let $1<p<\infty$, $1\leq q\leq \infty$. Then
    \begin{equation}\label{eq:bdd_riesz}
        \|\Ri_jf\|_{L^{p,q}}+\|\Ri_j\Ri_k f\|_{L^{p,q}}\lesssim \|f\|_{L^{p,q}}.
    \end{equation}
    If $p>3$,
    \begin{equation}\label{eq:gagliardo_riesz}
        \|\Ri_jf\|_{L^\infty}+\|\Ri_j\Ri_kf\|_{L^\infty}\lesssim \|f\|_{L^p}^{1-\frac{3}{p}}\|\nabla f\|_{L^p}^{\frac{3}{p}}.
    \end{equation}
    Finally, 
    \begin{equation}\label{eq:endpt_riesz}
        \big\|\nabla^2(-\Delta)^{-1}f\big\|_{L^\infty}\lesssim\|\nabla f\|_{L^{3,1}}.
    \end{equation}
\end{lemma}
\begin{proof}
\eqref{eq:bdd_riesz} and \eqref{eq:endpt_riesz} are standard estimates and can be found in \cite{Grafakos} and \cite{ONeil}. \eqref{eq:gagliardo_riesz} follows from \eqref{eq:bdd_riesz} and the Gagliardo–Nirenberg inequality. 
\end{proof}


\subsection{Schr\"odinger estimates and Galilean vector fields}
A pair $(q,r)$ is Schr\"odinger admissible in three dimensions if 
\[
2\leq q,r\leq\infty,\qquad\frac{2}{q}+\frac{3}{r}=\frac{3}{2}.
\]
We use the standard homogeneous and inhomogeneous Strichartz estimates; see Keel and Tao \cite{KT}. In particular, if $I$ is a time interval, $t_0\in I$, and 
\begin{equation}\label{eq:schrodinger}
    i\partial_t u-\Delta u=G,
\end{equation}
then
\begin{equation}\label{eq:strichartz}
    \|u\|_{L^\infty_t H^1_x(I\times \R^3)}+\|u\|_{L^\frac{8}{3}W^{1,4}_x(I\times \R^3)}\lesssim \|u(t_0)\|_{H^1}+\|G\|_{L^1_t H^1_x(I\times \R^3)}.
\end{equation}

The Galilean vector fields associated with \eqref{eq:schrodinger} are
\begin{equation}\label{eq:G}
    G_j:=ix_j+2t\partial_j,\qquad G=(G_1,G_2,G_3).
\end{equation}
They commute with $i\partial_t -\Delta$. For $t>0$, $G_ju$ admits a useful expression
\begin{equation}\label{eq:phase_G}
    Gu(t,x)=2t e^{-i\varsigma} \nabla \left(e^{i\varsigma} u\right)(t,x), \qquad  \text{where }\varsigma=\varsigma(t,x)=\frac{|x|^2}{4t}.
\end{equation}
Combining \eqref{eq:phase_G} with \eqref{eq:6_2_sobolev} gives
\begin{equation}\label{eq:L6_Gu}
    \|u(t)\|_{L^{6,2}}\lesssim t^{-1}\|Gu(t)\|_{L^2}, \qquad \left\|\nabla\left(e^{i\varsigma}u\right)(t)\right\|_{L^{6,2}}\lesssim t^{-2}\|G^2u(t)\|_{L^2}.
\end{equation}


\subsection{Hamiltonian transport} 
Let $\mathcal{H}=\mathcal{H}(t,x,v)$  be a real-valued Hamiltonian and set
\[
b_{\mathcal{H}}:=\left(\nabla_v \mathcal{H},-\nabla_x \mathcal{H}\right).
\]
The vector field $b_\mathcal{H}$ is divergence-free on phase space. Its characteristic flow therefore preserves Lebesgue measure. 
\begin{lemma}\label{lem:hamiltonian_transport}
    Let $1\leq r\leq \infty$, and suppose that
    \begin{equation}\label{eq:hamiltonian_transport_eqn}
        \partial_t f +\{ f,\mathcal{H}\}=g,\quad f(t=0)=f_0, 
    \end{equation}
    on $[0,T]$. Assume that the characteristic flow of $b_\mathcal{H}$ is well-defined and that $g\in L^1_t L^r_{x,v}([0,T]\times \R^3)$. Then 
    \begin{equation*}
    \begin{split}
        \|f(t)\|_{L^r_{x,v}}\leq \|f_0\|_{L^r_{x,v}}+\int_0^t \|g(s)\|_{L^r_{x,v}}ds, \quad t\in [0,T].
    \end{split}
\end{equation*}
\end{lemma}
\noindent The proof follows by integrating \eqref{eq:hamiltonian_transport_eqn} along the Hamiltonian flow and using Liouville’s theorem. We refer to Glassey \cite{Glassey} for the standard characteristic theory for Vlasov equations. 

We will also frequently use the Leibniz rule for the Poisson bracket and its commutation with differentiation:
\begin{equation*}
    \{fg,h\}=f\{g,h\}+g\{f,h\},
    \qquad
    \nabla\{f,g\}=\{\nabla f,g\}+\{f,\nabla g\}.
\end{equation*}


\section{Global well-posedness and decay estimates in the Coulomb case}\label{sec:GWP}
In this section, we prove global well-posedness and establish the estimates needed for the asymptotic analysis. We begin with the local theory and then introduce the free-transport profile and close a bootstrap argument coupling the decay of the two force fields with the slow growth of the relevant moments and Galilean norms. Throughout this section, implicit constants may depend on $\lambda$.

\subsection{Local well-posedness} \label{subsec:LWP} We begin by establishing the local well-posedness theory and the continuation criterion used in the global argument.
\begin{proposition}\label{prop:LWP} 
    Let $u_0\in H^1(\R^3)$ and let $f_0\in C^1(\R^3_x\times \R^3_v)$ with 
    \[
    A_0:=\|u_0\|_{H^1}+\|\la x,v\ra^4f_0\|_{ L^\infty_{x,v}}+\|\nabla_{x,v}f_0\|_{L^2_{x,v}\cap L^\infty_{x,v}}<\infty.
    \]
    Then there exists $T=T(A_0)>0$ such that the system admits a unique strong solution on $[0,T]$. Moreover,
    \[
    u\in C([0,T];H^1)\,\cap\, L^\frac{8}{3}([0,T]; W^{1,4}),\qquad f\in C^1([0,T]\times \R^3_x\times \R^3_v). 
    \]
    Writing 
    \begin{equation}\label{eq:gamma}
        \gamma(t,x,v):=f(t,x+tv,v),
    \end{equation}
  we have   
    \[
    \|u(t)\|_{L^\frac{8}{3}_tW^{1,4}_x([0,T]\times \R^3)}+\sup_{0\leq t\leq T} \left( \|u(t)\|_{H^1}+\|\la x,v\ra^4\gamma(t)\|_{ L^\infty_{x,v}}+\|\nabla_{x,v}\gamma(t)\|_{L^2_{x,v}\cap L^\infty_{x,v}}\right)\leq C(A_0).
    \]
\end{proposition}
\begin{proof}
    The local well-posedness follows from a standard Picard iteration. We first give the estimates underlying the standard iteration argument. For a phase-space function, define
    \[
    \rho_g(x):=\int g^2(x,v)dv,\qquad \phi_g=V*\rho_g.
    \]
    For $w\in H^1$, define
    \[
    \Psi_w:=V*|w|^2,\qquad F_w:=-\nabla \Psi_w.
    \]
    Lemma \ref{lem:lorentz_estimates} gives
    \[
    \|\phi_g\|_{W^{1,\infty}}\lesssim \|\rho_g\|_{L^{\frac{3}{2},1}}+\|\rho_g\|_{L^{3,1}}. 
    \]
    If that $\|\la x,v\ra ^4g\|_{L^\infty_{x,v}}\lesssim A$, then \eqref{eq:rho_est} implies
    \[
    \|\rho_g\|_{L^1_x}+\|\rho_g\|_{L^\infty_x}\lesssim A^2.
    \]
    Interpolation then gives
    \begin{equation}\label{eq:prop_phi_W1}
        \|\phi_g\|_{W^{1,\infty}}\lesssim A^2.
    \end{equation}
    The same conclusion holds uniformly under translation
    \[
    g_t(x,v)=g(x-tv,v),
    \]
    as the $L^1_{x,v}$ and $L^2_{x,v}$ norms are unchanged. Indeed, 
    \[
    \rho_{g_t}(x)=\int |g(x-tv,v)|^2 dv\lesssim A^2;
    \]
    hence
    \[
    \sup_{t\geq 0} \left(\|\rho_{g_t}\|_{L^1_x}+\|\rho_{g_t}\|_{L^\infty_x}\right)\lesssim A^2,
    \]
    and interpolation gives the uniform estimate for $\phi_{g_t}$ in $W^{1,\infty}$. \eqref{eq:gagliardo_riesz} with $p=4$ and Sobolev embedding $W^{1,4}(\R^3)\hookrightarrow L^4(\R^3)\cap L^\infty(\R^3)$ imply
    \[
    \|\nabla F_w(t)\|_{L^\infty}\lesssim \|w(t)\|_{W^{1,4}}^2.
    \]
    Consequently, 
    \begin{equation}\label{eq:prop_L1_partial_F}
        \int_0^T \|\nabla F_w(t)\|_{L^\infty}dt\lesssim T^\frac{1}{4} \|w\|_{L^\frac{8}{3}_t W_x^{1,4}([0,T])}^2\,.
    \end{equation}
    Finally, Lemma \ref{lem:lorentz_estimates} and Sobolev embedding give
    \begin{equation}\label{eq:prop_F}
        \|\Psi_w\|_{L^\infty}+\|F_w\|_{L^\infty}\lesssim \|w\|_{H^1}^2.
    \end{equation}
    We iterate simultaneously in $u$ and in the free-transport profile $\gamma$. Given $(u_n,\gamma_n)$, define
    \[
    f_n(t,x,v):=\gamma_n(t,x-tv,v),\qquad \phi_n=V* \rho_{f_n},\qquad \Psi_n:=V* |u_n|^2,\qquad F_n:=-\nabla \Psi_n.
    \]
    We then define the next step by
    \[
    u_{n+1}(t)=e^{-it\Delta} u_0-i\int_0^t e^{-i(t-s)\Delta} \left(\phi_n(s)u_n(s)\right) ds,
    \]
    and $\gamma_{n+1}$ as the solution to
    \[
    \partial_t \gamma_{n+1}+\{\gamma_{n+1}, \mathcal{H}_{n}\}=0,\qquad \mathcal{H}_n:=\Psi_n(t,x+tv),\qquad \gamma_{n+1}(t=0)=f_0.
    \]
    The Strichartz estimates \eqref{eq:strichartz} give
    \begin{align*}
        \|u_{n+1}\|_{L^\infty_tH^1_x \cap L^\frac{8}{3}_t W^{1,4}_x([0,T]\times \R^3)}&\lesssim \|u_0\|_{H^1}+\|\phi_n u_n\|_{L^1_t H^1_x([0,T]\times \R^3)}\\
        &\qquad \lesssim A_0+T\|\phi_n\|_{L^\infty_t W^{1,\infty}_x}\|u_n\|_{L^\infty_t H^1_x}\,.
    \end{align*}
    The differentiated Hamiltonian transport equation and Lemma \ref{lem:hamiltonian_transport} give, for $p=2$ or $\infty$:
    \begin{align*}
        \|\nabla_{x,v}\gamma_{n+1}(t)\|_{L^p_{x,v}}&\leq \|\nabla_{x,v}f_0\|_{L^p_{x,v}}+\int_0^t (1+s^2)\|\nabla_{x,v}\gamma_{n+1}(s)\|_{L^p_{x,v}}\|\nabla F_n(s)\|_{L^\infty}ds\\
        &\leq A_0+\int_0^t (1+s^2)\|\nabla_{x,v}\gamma_{n+1}(s)\|_{L^p_{x,v}}\|\nabla F_n(s)\|_{L^\infty}ds.
    \end{align*}
    Gr\"onwall inequality then gives
    \begin{align*}
        \|\nabla_{x,v}\gamma_{n+1}(t)\|_{L^2_{x,v}\cap L^\infty_{x,v}}&\leq A_0 \,\exp\left((1+t^2)\|\nabla F_n\|_{L^1_sL^\infty_x([0,t]\times \R^3)} \right).
    \end{align*}
    Similarly, the Leibniz rule and Lemma \ref{lem:hamiltonian_transport} bound
    \[
    \|\la x,v\ra^4 \gamma_{n+1}(t)\|_{L^\infty_{x,v}}\leq \|\la x,v\ra^4 f_0\|_{L^\infty_{x,v}}+\int_0^t (1+s)\|\la x,v\ra^4 \gamma_{n+1}(s)\|_{L^\infty_{x,v}} \|F_n(s)\|_{L^\infty}ds,
    \]
    hence 
    \[
    \|\la x,v\ra^4 \gamma_{n+1}(t)\|_{L^\infty_{x,v}}\leq A_0\exp\left((t+\frac{t^2}{2})\|F_n\|_{L^\infty_sL^\infty_x([0,t]\times \R^3)} \right).
    \]
    These iterated bounds together with \eqref{eq:prop_phi_W1}-\eqref{eq:prop_F} prove uniform boundedness of the norm
    \[
    A_n:=\|u_n(t)\|_{L^\frac{8}{3}_tW^{1,4}_x([0,T]\times \R^3)}+\sup_{0\leq t\leq T} \left( \|u_n(t)\|_{H^1}+\|\la x,v\ra^4\gamma_n(t)\|_{L^\infty_{x,v}}+\|\nabla_{x,v}\gamma_n(t)\|_{L^2_{x,v}\cap L^\infty_{x,v}}\right)
    \]
    by $C(A_0)$, with $T$ chosen sufficiently small depending on $A_0$. 

    \medskip

    We now prove contraction of $(u_n,\gamma_n)$ in a weaker norm
    \[
    \|(w,g)\|_{X_t}:=\|w\|_{L^\infty_t H^1_x([0,t]\times \R^3)}+\|g\|_{L^\infty_t ([0,t];L^2_{x,v}\cap L^\infty_{x,v})}.
    \]
    Since $(x,v)\to (x-tv,v)$ preserves Lebesgue measure, 
    \begin{align*}
        \big\|\rho_{f_n}(t)-\rho_{f_{n-1}}(t)\|_{L^1_x}\lesssim \|\gamma_n(t)-\gamma_{n-1}(t)\|_{L^2_{x,v}}\left(\|\gamma_n(t)\|_{L^2_{x,v}}+\|\gamma_{n-1}\|_{L^2_{x,v}} \right).
    \end{align*}
    The conservation of mass of Hamiltonian transport flow \eqref{eq:mass_fermion} bounds
    \[
    \big\|\rho_{f_n}(t)-\rho_{f_{n-1}}(t)\|_{L^1_x}\lesssim A_0\|\gamma_n(t)-\gamma_{n-1}(t)\|_{L^2_{x,v}}.
    \]
    Since $\la v\ra ^{-4}$ is integrable in 3 dimensions, we also bound
    \begin{align*}
         \big\|\rho_{f_n}(t)-\rho_{f_{n-1}}(t)\|_{L^\infty_x}&\lesssim \|\gamma_n(t)-\gamma_{n-1}(t)\|_{L^\infty_{x,v}}\left(\|\gamma_n(t,x-tv,v)\|_{L^1_v L^\infty_x}+\|\gamma_{n-1}(t,x-tv,v)\|_{L^1_v L^\infty_x}\right)\\
         &\lesssim C(A_0) \|\gamma_n(t)-\gamma_{n-1}(t)\|_{L^\infty_{x,v}}.
    \end{align*}
    Interpolation and Lemma \ref{lem:lorentz_estimates} imply
    \[
    \big\|\phi_n(t)-\phi_{n-1}(t)\big\|_{W^{1,\infty}_x}\lesssim C(A_0)\|\gamma_n(t)-\gamma_{n-1}(t)\|_{L^2_{x,v}\cap L^\infty_{x,v}}.
    \]
    We also retain the uniform estimate from \eqref{eq:prop_phi_W1}
    \[
    \|\phi_n\|_{W^{1,\infty}}\lesssim C(A_0)^2.
    \]
    Next, we estimate the difference of the potentials generated by the Schr\"odinger component. Lemma \ref{lem:lorentz_estimates} bounds
    \begin{align*}
        \big\|\nabla\Psi_{n}(t)-\nabla\Psi_{n-1}(t)\big\|_{L^\infty}&\lesssim \|u_n(t)-u_{n-1}(t)\|_{L^{6,2}}\left( \|u_n(t)\|_{L^{6,2}}+\|u_{n-1}(t)\|_{L^{6,2}}\right)\\
        &\lesssim C(A_0) \|u_n(t)-u_{n-1}(t)\|_{H^1}.
    \end{align*}
    We now estimate $u_{n+1}-u_{n}$. Subtracting the two Duhamel formulas yields
    \[
    u_{n+1}(t)-u_{n}(t)=-i\int_0^t e^{-i(t-s)\Delta} \left[(\phi_n(s)-\phi_{n-1}(s))u_n(s)+\phi_{n-1}(s)(u_n(s)-u_{n-1}(s)) \right]ds;
    \]
    hence we bound
    \begin{equation}\label{eq:lwp_contract_u}
        \begin{split}
            \big\|&u_{n+1}(t)-u_{n}(t)\big\|_{H^1}\\
        &\lesssim \int_0^t \big(\|\phi_n(s)-\phi_{n-1}(s)\|_{W^{1,\infty}}\|u_n(s)\|_{H^1}+\|\phi_{n-1}(s)\|_{W^{1,\infty}}\|u_n(s)-u_{n-1}(s)\|_{H^1}\big)ds\\
        &\lesssim C(A_0)^2 \int_0^t \big(\|\gamma_n(s)-\gamma_{n-1}(s)\|_{L^2_{x,v}\cap L^\infty_{x,v}}+ \|u_n(s)-u_{n-1}(s)\|_{H^1} \big)ds.
        \end{split}
    \end{equation}
    It remains to estimate $\gamma_{n+1}-\gamma_n$. Subtracting the iterating equations for $\gamma_{n+1}$ and $\gamma_n$:
    \[
    \partial_t\left(\gamma_{n+1}-\gamma_{n}\right)+\left\{\gamma_{n+1}-\gamma_{n}, \mathcal{H}_n \right\}=-\left\{\gamma_n. \mathcal{H}_n-\mathcal{H}_{n-1} \right\},\qquad \gamma_{n+1}(0)-\gamma_{n}(0)=0,
    \]
    Lemma \ref{lem:hamiltonian_transport} gives, for $p=2$ and $p=\infty$,
    \begin{equation}\label{eq:lwp_contract_gamma}
        \begin{split}
            \left\|\gamma_{n+1}(t)-\gamma_{n}(t)\right\|_{L^p_{x,v}}&\lesssim \int_0^t (1+s)\|\nabla_{x,v}\gamma_n(s)\|_{L^p_{x,v}}\big\|\nabla \Psi_n(s)-\nabla \Psi_{n-1}(s)\big\|_{L^\infty}ds\\
        &\lesssim C(A_0)^2 (1+t) \int_0^t \|u_n(s)-u_{n-1}(s)\|_{H^1}ds.
        \end{split}
    \end{equation}
    Summing up \eqref{eq:lwp_contract_u} and \eqref{eq:lwp_contract_gamma}, we have
    \begin{align}\label{eq:contract_final}
        \left\|\left(u_{n+1}-u_n, \gamma_{n+1}-\gamma_n\right) \right\|_{X_t}\lesssim C(A_0)^2t (1+t)\left\|\left(u_{n}-u_{n-1}, \gamma_{n}-\gamma_{n-1}\right) \right\|_{X_t}.
    \end{align}
    After decreasing $T$, if necessary, \eqref{eq:contract_final} shows that the iteration is a contraction. The asserted regularity follows from the Strichartz estimates and the characteristic equations. The same estimates give uniqueness.   
\end{proof}
\noindent The proof also yields the following continuation criterion: the solution can be continued beyond $T$ as long as
\[
\sup_{0\leq t<T}\left(\|u(t)\|_{H^1}+\|\la x,v\ra^4 \gamma(t)\|_{L^\infty_{x,v}}+\|\nabla_{x,v}\gamma(t)\|_{L^2_{x,v}\cap L^\infty_{x,v}}\right)<\infty.
\]

\begin{remark}\label{rmk:long_range}
The proof of Proposition \ref{prop:LWP} extends, with only minor modifications, to
\[
V_\alpha(x)=|x|^{-\alpha},\qquad \frac12<\alpha<1.
\]
Indeed, since $\nabla^kV_\alpha$ is locally integrable for $0\leq k\leq2$, the same kernel decomposition gives
\[
\|\phi(t)\|_{W^{2,\infty}}\lesssim_\alpha\|\rho(t)\|_{L^1}+\|\rho(t)\|_{L^\infty}.
\]
Moreover,
\[
\|\nabla^3\phi(t)\|_{L^2}\lesssim_\alpha\big\||\nabla|^\alpha\rho(t)\big\|_{L^2}\lesssim\|\rho(t)\|_{L^2}^{1-\alpha}\|\nabla\rho(t)\|_{L^2}^{\alpha},
\]
which is bounded on every fixed short time interval by the phase-space norms propagated in Proposition \ref{prop:LWP}.

We also record that the third-order Galilean norm can be propagated on
$[0,1]$. Indeed, since
$G_j$ commutes with $i\partial_t-\Delta$, for every
multi-index $\beta$ with $|\beta|\leq3$,
\[
(i\partial_t-\Delta)G^\beta u=\phi G^\beta u+\sum_{0<\mu\leq\beta}c_{\beta,\mu}\,t^{|\mu|}(\partial^\mu\phi)G^{\beta-\mu}u.
\]
The terms with $|\mu|\leq2$ are controlled by
$\|\phi\|_{W^{2,\infty}}$, while the only top-order term satisfies
\[
\|(\nabla^3\phi)u\|_{L^2}\leq\|\nabla^3\phi\|_{L^2}\|u\|_{L^\infty}.
\]
The local Strichartz bound
\[
u\in L^{\frac83}\big([0,1];W^{1,4}(\mathbb R^3)\big)\hookrightarrow L^{\frac83}\big([0,1];L^\infty(\mathbb R^3)\big)
\]
therefore makes this contribution integrable in time. An induction on
$|\beta|$ and the standard energy estimate yield
\[
\sup_{0\leq t\leq1}\sum_{|\beta|\leq3}\|G^\beta u(t)\|_{L^2}\lesssim_\alpha \varepsilon_0.
\]
\end{remark}

\subsection{Slow growth of Galilean vector field} 
We next establish the slow growth of $Gu$. Since the Galilean vector field $G$ commutes with the free Schr\"odinger evolution and corresponds to differentiation of the Fourier profile, control of $Gu$ captures the localization of $u$ along the linear flow. This weighted information is essential for deriving sharp decay estimates for the bosonic density and the force it generates. We prove the required bound by a bootstrap argument.

Let $(u,f)$ be a solution to \eqref{main} with initial data $(u_0,f_0)$, and define
\begin{equation}\label{eq:E}
    E(t):=\nabla\phi(t).
\end{equation}
Fix $0<\delta<1$, say $\delta=\frac{1}{10}$. On an interval $[0,T]$, we impose the bootstrap assumptions
\begin{subequations}\label{boot}
    \begin{align}
        &\sup_{0\leq t\leq T}\left(\log^{-3}(2+t)\|Gu(t)\|_{L^2}+\log^{-25}(2+t)\|\la x,v\ra ^4 \gamma(t)\|_{L^\infty_{x,v}}\right)\leq \varepsilon,\label{boot_1}\\
        &\qquad \qquad\quad \sup_{0\leq t\leq T} \left(\|E(t)\|_{L^\infty}+\|F(t)\|_{L^\infty}\right)\leq \varepsilon^2 \la t\ra ^{-2+\delta}.\label{boot_field}
    \end{align}
\end{subequations}
We will show these bounds improve, provided the initial data satisfy \eqref{eq:initial_smallness} and $\varepsilon_0\ll\varepsilon\ll 1$.

\medskip

\paragraph{\bf The field generated by the Vlasov component} The direct estimate 
\[
\|E(t)\|_{L^\infty}\lesssim \|\rho(t)\|_{L^{3,1}}
\]
loses the logarithmic growth of the weighted profile. To recover the nearly optimal $t^{-2}$ decay, we compare $E$ with an effective field depending only on the velocity distribution of $\gamma$. 

Choose a radial function $\chi$ that is supported in the annulus $\frac{1}{2}<|z|<2$ and normalized by $ \int \chi dx=1$. Define, for $R>0$,
\begin{align*}
    \phi_R(t,x):=\iint \chi \left(\frac{x-y}{R}\right) f^2(t,y,v)dydv =\iint \chi\left(\frac{x-y-tv}{R}\right)\gamma^2(t,y,v)dydv,
\end{align*}
and 
\begin{align*}
    E_R(t,x):=\iint \nabla\chi\left(\frac{x-y}{R}\right) f^2(t,y,v) dydv=\iint \nabla \chi\left(\frac{x-y-tv}{R}\right) \gamma^2(t,y,v) dydv.
\end{align*}
Using
\begin{align*}
    \frac{1}{C_0|z|}=\int_0^\infty \chi\left(\frac{z}{R}\right) \frac{dR}{R^2},
\end{align*}
the Coulomb potential and its gradient admit the exact representations
\begin{align}
    \phi(t,x)=C_0\lambda \int_0^\infty \phi_R(t,x) \frac{dR}{R^2},
\end{align}
and 
\begin{align}
    E(t,x)=C_0\lambda \int_0^\infty E_R(t,x)\frac{dR}{R^3}.
\end{align}
For $r>0$ and $a\in \R^3$, we define the effective shells
\begin{align}
    \phi_r^0(t,a):=\iint \chi\left(\frac{a-v}{r}\right)\gamma^2(t,y,v)dydv
\end{align}
and 
\begin{align}
    E_r^0(t,a):=\iint \nabla \chi\left(\frac{a-v}{r}\right)\gamma^2 (t,y,v)dydv.
\end{align}
For $t>0$, the corresponding effective potential and field are 
\begin{align}
    \phi^0(t,a):=C_0\frac{\lambda}{t}\int_0^\infty \phi_r^0(t,a)\frac{dr}{r^2},
\end{align}
and 
\begin{equation}
    E^0(t,a):=C_0\frac{\lambda}{t^2}\int_0^\infty E_r^0(t,a)\frac{dr}{r^3}
\end{equation}
\begin{lemma}\label{lem:effective_field}
    We have
    \begin{equation}\label{eq:bounded_phi_R}
        \big| \phi_R(t,x)\big|+|E_R(t,x)|\lesssim \min\left\{1, \frac{R^3}{\la t\ra ^3} \right\}\big\|\la x,v\ra^4 \gamma(t)\big\|_{L^\infty_{x,v}}^2,
    \end{equation}
    and 
    \begin{equation}\label{eq:bdd_effective_fields}
        \big|\phi^0_r(t,x)\big|+\big|E^0_r(t,x)\big|\lesssim \min\left\{1, r^3\right\}\|\la x,v\ra^4 \gamma(t)\big\|_{L^\infty_{x,v}}^2.
    \end{equation}
    Moreover, for $1\leq t\leq T$, 
    \begin{equation}\label{eq:difference_effective_potential}
        t\left|\phi(t,x)-\phi^0(t, \frac{x}{t}) \right|\lesssim t^{-\frac{1}{20}}\ncal,
    \end{equation}
    and 
    \begin{equation}\label{eq:difference_effective_field}
        t^2\left| E(t,x)-E^0\left(t,\frac{x}{t}\right)\right|\lesssim t^{-\frac{1}{20}}\big\|\la x,v\ra^4 \gamma(t)\big\|_{L^\infty_{x,v}}^2.
    \end{equation}
\end{lemma}
\begin{proof}
We prove the estimates for $E$; the argument for $\phi$ is identical. By the compact support of $\nabla\chi$, we bound
\begin{align*}
    |E_R(t,x)|\lesssim \iint_{|x-y-tv|\leq 2R}\gamma^2 (t,y,v)dydv\lesssim \|\gamma(t)\|_{L^2_{x,v}}^2 \lesssim \ncal.
\end{align*}
If $t\geq 1$, we integrate first in $v$. For every $y$, the set of admissible velocities has volume $\mathcal{O}(\left(\frac{R}{t}\right)^3)$; hence
\begin{align*}
    |E_R(t,x)|\lesssim \left(\frac{R}{t}\right)^3 \ncal \int _{\R^3}\la y\ra^{-8}dy\lesssim \left(\frac{R}{t}\right)^3 \ncal.
\end{align*}
If $0\leq t\leq 1$, we instead integrate first in $y$ and a similar argument gives
\[
|E_R(t,x)|\lesssim R^3 \ncal. 
\]
Combining these estimates proves \eqref{eq:bounded_phi_R}. \eqref{eq:bdd_effective_fields} follows in the same way by integrating over $|a-v|\leq 2r$. It remains to compare the physical and effective fields. Fix $t\geq 1$, set $x=at$, and write for simplicity
\[
\mathfrak{d}_r(t,x):=E_{tr}(t,x)-E^0_r(t,a)=\iint \left[ \nabla\chi\left(\frac{a-v-\frac{y}{t}}{r}\right)-\nabla\chi\left(\frac{a-v}{r}\right) \right]\gamma^2(t,y,v) dydv.
\]
The estimates already proved give
\[
|\mathfrak{d}_r(t,x)|\lesssim \min\{1,r^3\}\ncal.
\]
On the other hand, the mean value theorem yields
\[
  \left| \nabla\chi\left(\frac{a-v-\frac{y}{t}}{r}\right)-\nabla\chi\left(\frac{a-v}{r}\right) \right|\lesssim_\chi \frac{|y|}{tr}.
\]
Consequently, 
\[
|\mathfrak{d}_r(t,x)|\lesssim \min\{1, r^3, \frac{1}{tr}\} \ncal.
\]
Plugging back into the shell representation, we split the integral at $t^{-\frac{1}{20}}$ and bound by
\begin{align*}
    t^2 |E(t,x)-E^0(t,a)|\lesssim\int_0^\infty |\mathfrak{d}_r(t,x)|\frac{dr}{r^3}\lesssim \ncal\left(\int_0^{t^{-\frac{1}{20}}}dr+\int_{t^{-\frac{1}{20}}}^\infty \frac{1}{t r^4} dr \right).
\end{align*}
This proves \eqref{eq:difference_effective_field}. For \eqref{eq:difference_effective_potential}, the same argument, now with the measure $r^{-2}dr$, gives
\[
t |\phi(t,x)-\phi^0(t,a)|\lesssim \ncal\left(\int_0^{t^{-\frac{1}{10}}}rdr+\int_{t^{-\frac{1}{10}}}^\infty \frac{1}{t r^3} dr \right)\lesssim t^{-\frac{1}{20}}\ncal.
\]
\end{proof}

We next exploit the evolution equation for $\gamma$ to obtain bounds for the effective shells that are uniform in time. 
\begin{proposition}\label{prop:uniform_effective}
    Under the bootstrap assumption \eqref{boot} on $[0,T]$ and the assumption on initial data \eqref{eq:initial_smallness}, for every $t\in[0,T]$, $a\in \R^3$ and $r>0$, we have
    \begin{align}
        |\partial_t \phi_r^0(t,a)|+|\partial_t E_r^0(t,a)|\lesssim \min\{r^{-1}, r^2\}\|F(t)\|_{L^\infty}\ncal.
    \end{align}
    Consequently, 
    \begin{equation}\label{eq:diff_bound}
        \big|\phi_r^0(t,a)-\phi_r^0(0,a)\big|+\big|E_r^0(t,a)-E_r^0(0,a)\big|\lesssim \varepsilon^4 \min\{r^{-1}, r^2\},
    \end{equation}
    and therefore
    \begin{equation}
        \big|\phi_r^0(t,a)\big|+\big|E_r^0(t,a)\big|\lesssim \varepsilon_0^2 \min\{1,r^3\}+\varepsilon^4 \min\{r^{-1}, r^2\}.
    \end{equation}
\end{proposition}
\begin{proof}
 The   Leibniz rule and \eqref{eq:vlasov} give
    \begin{equation}\label{eq:vlasov_gamma}
        \partial_t \,\gamma+\{\gamma, \Psi(t,x+tv) \}=0,\qquad \partial_t \,\gamma^2+\{\gamma^2, \Psi(t,x+tv) \}=0.
    \end{equation}
    Integration by parts in phase space gives, for every smooth function $K=K(v)$, and $H=\Psi(t,x+tv)$
    \begin{align*}
         \frac{d}{dt}\iint K(v)\gamma^2(t,y,v) dydv&=\iint \gamma^2(t,y,v)\{K,H\} dydv\\
         &\qquad =\iint \gamma^2(t,y,v) F(t,y+tv)\cdot \nabla_vK(v)dydv.
    \end{align*}
    Taking $K(v)=\partial_j \chi\left(\frac{a-v}{r}\right)$, we obtain
    \begin{align*}
        |\partial_t E_r^0(t,a)|&\lesssim r^{-1}\|F(t)\|_{L^\infty}\iint \left|\nabla^2 \chi\left(\frac{a-v}{r}\right)\right|\gamma^2(t,y,v)dydv\\
        &\qquad \lesssim \min\{r^{-1}, r^2\}\|F(t)\|_{L^\infty}\ncal.
    \end{align*}
    The proof for $\phi_r^0$ is identical, with $\nabla^2\chi $ replaced by $\nabla \chi$. Finally, the bootstrap assumptions imply
    \[
    \int_0^T \|F(t)\|_{L^\infty}\ncal dt\lesssim \varepsilon^4 \int_0^\infty \la t\ra ^{-2+\delta}\log^{50}(2+t)dt \lesssim \varepsilon^4.
    \]
   Integrating the derivative estimate in time proves the difference bound \eqref{eq:diff_bound}. At $t=0$, Lemma \ref{lem:effective_field} and the initial smallness assumption \eqref{eq:initial_smallness} give
   \[
   |\phi_r^0(0,a)|+|E_r^0(0,a)|\lesssim \varepsilon_0^2 \min\{1, r^3\},
   \]
   which proves the final estimate.
\end{proof}

As a consequence, we obtain the optimal decay for the potential and an almost-optimal bound for its gradient. 
\begin{corollary}\label{cor:optimal_decay_E}
    Under the bootstrap assumption \eqref{boot} on $[0,T]$, for every $t\in [0,T]$, we have
    \begin{equation}\label{eq:optimal_phi}
        \|\phi(t)\|_{L^\infty}\lesssim \varepsilon^2 \la t\ra^{-1},
    \end{equation}
    and 
    \begin{equation}\label{eq:optimal_E}
        \|E(t)\|_{L^\infty}\lesssim \varepsilon^2 \la t\ra ^{-2}\log(2+t).
    \end{equation}
\end{corollary}
\begin{proof}
    For $0\leq t\leq 1$, Lemma \ref{lem:effective_field} and the shell representations give
    \begin{align*}
        |\phi(t,x)|+|E(t,x)|\lesssim \ncal \int_0^\infty \left(\frac{\min\{1,R^3\}}{R^2}+\frac{\min\{1,R^3\}}{R^3} \right) dR\,\lesssim\, \varepsilon^2.
    \end{align*}
    For $t\geq 1$, Proposition \ref{prop:uniform_effective} gives
    \[
    t|\phi^0(t,a)|\lesssim \int_0^\infty |\phi_r^0(t,a)|\frac{dr}{r^2}\lesssim \varepsilon_0^2 +\varepsilon^4\lesssim \varepsilon^2. 
    \]
    To estimate $E^0$, we split the $r$-integral up to $\rho:=\log^{-50}(2+t)$. On $0<r<\rho$, we use the direct bound from Lemma \ref{lem:effective_field}. On the remaining scales, we use Proposition \ref{prop:uniform_effective}. This gives
    \begin{align*}
        t^2|E^0(t,a)|&\lesssim \ncal\, \rho +\int_\rho^1 \left(\varepsilon^2 r^3 + \varepsilon^4 r^2\right)\frac{dr}{r^3}+\int_1^\infty \left(\varepsilon_0^2 +\varepsilon^4 r^{-1}\right) \frac{dr}{r^3}\\
        &\qquad \lesssim \ncal\, \rho+\varepsilon_0^2 +\varepsilon^4 \left(1+\log\frac{1}{\rho}\right)\,\lesssim\, \varepsilon^2 \log(2+t).
    \end{align*}
    It remains to return from the effective fields to the physical ones. Lemma \ref{lem:effective_field} and \eqref{boot} give \eqref{eq:optimal_phi} and \eqref{eq:optimal_E}. 
\end{proof}

\medskip

\paragraph{\bf Galilean vector field and propagation of moments} We next derive Galilean vector field estimates for the Hartree component and propagate the moments of the Vlasov profile.
\begin{lemma}\label{lem:improved_Gu}
    Under the bootstrap assumption \eqref{boot} on $[0,T]$, for every $t\in [0,T]$, we have
    \begin{equation}\label{eq:improved_Gu}
        \|Gu(t)\|_{L^2}\lesssim \varepsilon_0+ \varepsilon_0\varepsilon^2 \log^2(2+t). 
    \end{equation}
    Consequently, 
    \begin{equation}
        \|F(t)\|_{L^\infty}\lesssim \varepsilon_0^2 \la t\ra ^{-2}\log^4(2+t).
    \end{equation}
\end{lemma}
\begin{proof}
    Commuting the Hartree equation \eqref{eq:hartree} with $G$, we obtain
    \[
    \left(i\partial_t-\Delta\right)Gu=\phi Gu+2t (\nabla\phi)u.
    \]
    The standard energy estimate and conservation of the bosonic mass \eqref{eq:mass_boson} give
    \[
    \partial_t \|Gu(t)\|_{L^2}\lesssim t\|E(t)\|_{L^\infty}\|u(t)\|_{L^2}\lesssim \varepsilon_0\,t\|E(t)\|_{L^\infty} .
    \]
    Using Corollary \ref{cor:optimal_decay_E}, we obtain
    \begin{align*}
        \|Gu(t)\|_{L^2}\leq \|xu_0\|_{L^2}+C\varepsilon_0\int_0^t s\|E(s)\|_{L^\infty}ds\lesssim \varepsilon_0+\varepsilon_0 \varepsilon^2 \log^2(2+t).
    \end{align*}
    It remains to estimate $F$. Lemma \ref{lem:lorentz_estimates} and \eqref{eq:L6_Gu} bound
    \begin{align*}
        \|F(t)\|_{L^\infty}\lesssim \|u(t)\|_{L^{6,2}}^2 \lesssim t^{-2}\|Gu(t)\|_{L^2}^2.
    \end{align*}
    For $0\leq t\leq 1$, the standard $H^1$-energy estimate and conservation of mass bound $\|u(t)\|_{H^1}$. Hence
    \[
    \|F(t)\|_{L^\infty}\lesssim\|u(t)\|_{L^{6,2}}^2 \lesssim \|u(t)\|_{H^1}^2 \lesssim\varepsilon_0^2.
    \]
    Combining the two time regimes completes the proof.
\end{proof}

We next propagate the moments of the Vlasov profile. 
\begin{lemma}\label{lem:improved_gamma}
    Under the bootstrap assumption \eqref{boot} on $[0,T]$, for every $t\in [0,T]$ and $0\leq b\leq 4$,
    \begin{equation}
        \|\la v\ra ^b\gamma(t)\|_{L^\infty_{x,v}}\lesssim \varepsilon_0.
    \end{equation}
    Moreover, for $1\leq a\leq 4$ and $0\leq b\leq 4$, 
    \begin{equation}
        \big\|\la x\ra ^a \la v\ra ^b \gamma(t)\big\|_{L^\infty_{x,v}}\lesssim\varepsilon_0+\varepsilon_0\varepsilon^2 \log^{5a}(2+t).
    \end{equation}
    In particular, 
    \begin{equation}\label{eq:improved_gamma}
        \big\|\la x,v\ra ^4\gamma(t)\big\|_{L^\infty_{x,v}}\lesssim \varepsilon_0+\varepsilon_0\varepsilon^2 \log^{20}(2+t).
    \end{equation}
\end{lemma}
\begin{proof}
    For any time-independent smooth weight $w=w(x,v)$, \eqref{eq:vlasov_gamma} gives
    \begin{equation}
        \partial_t\left(w\gamma\right)+\{w\gamma, \Psi(t,x+tv)\}=\gamma\,\{w,\Psi(t,x+tv)\}.
    \end{equation}
    Therefore, for integers $0\leq a,b\leq 4$, 
    \begin{equation*}
        \begin{split}
            \partial_t\left( \la x\ra^a \la v\ra^b\gamma\right)+&\left\{\la x\ra^a \la v\ra^b\gamma, \Psi(t,x+tv) \right\}\\
            &\qquad =-at\la x\ra^{a-2}\la v\ra^bx\cdot F(t,x+tv)\gamma+b\la x\ra^a \la v\ra^{b-2}v\cdot F(t,x+tv)\gamma,
        \end{split}
    \end{equation*}
    where the corresponding term is omitted when $a=0$ or $b=0$. Therefore, with initial data $\big\|\la x\ra ^a \la v\ra^b \gamma(0)\big\|_{L^\infty_{x,v}}\leq \varepsilon_0$ for $0\leq a,b\leq 4$, we have
    \begin{align*}
        \big\|\la x\ra ^a \la v\ra^b \gamma(t)\big\|_{L^\infty_{x,v}}&\leq \varepsilon_0+C_a\int_0^t s\|F(s)\|_{L^\infty} \big\|\la x\ra ^{a-1} \la v\ra^b \gamma(s)\big\|_{L^\infty_{x,v}}ds\\
        &\quad \qquad \quad \quad +C_b\int_0^t \|F(s)\|_{L^\infty}\big\|\la x\ra ^a \la v\ra^{b-1} \gamma(s)\big\|_{L^\infty_{x,v}}ds.
    \end{align*}
    Lemma \ref{lem:improved_Gu} and an induction on $a+b$ then give
    \begin{equation*}
        \big\|\la x\ra ^a \la v\ra^b \gamma(t)\big\|_{L^\infty_{x,v}}\lesssim \varepsilon_0+\varepsilon_0\varepsilon ^2 \log^{5a}(2+t).
    \end{equation*}
    Taking $a=0$ proves the velocity-moment bound with no growth, while $a=4=b$ gives \eqref{eq:improved_gamma}. 
\end{proof}

We can now control the second-order Galilean vector fields of $u$ and the spatial derivative of the field $F$. 
\begin{lemma}\label{lem:G2u}
    Under the bootstrap assumption \eqref{boot} on $[0,T]$, for every $t\in [0,T]$, we have
    \begin{equation}\label{eq:G2u}
        \|G^2u(t)\|_{L^2}\lesssim \varepsilon_0+\varepsilon^3 \log^\frac{41}{3}(2+t).
    \end{equation}
    Consequently, for every $t\in [0,T]$, 
    \begin{equation}\label{eq:nabla_F}
        \|\nabla_x F(t)\|_{L^\infty}\lesssim \varepsilon^2 t^{-3}\log^\frac{47}{3}(2+t).
    \end{equation}
\end{lemma}
\begin{proof}
    We first estimate the Vlasov density. For $t\geq 1$, taking $a=\frac{8}{5}$ in Lemma \ref{lem:improved_gamma} gives
    \begin{align*}
        \rho(t,x)=\int \gamma^2(t,x-tv,v) dv\leq \|\la x\ra^{a}\gamma(t)\|_{L^\infty_{x,v}}^2 \int\la x-tv\ra ^{-\frac{16}{5}}dv\lesssim \varepsilon^2 t^{-3}\log^{16}(2+t).
    \end{align*}
    For $0\leq t\leq 1$, the uniform $v$-moment bound gives
    \[
    \|\rho(t)\|_{L^\infty}\lesssim \|\la v\ra ^2 \gamma(t)\|_{L^\infty_{x,v}}^2 \int \la v\ra^{-4}dv\lesssim \varepsilon^2.
    \]
    Thus, 
    \[
    \|\rho(t)\|_{L^\infty}\lesssim \varepsilon^2 \la t\ra^{-3}\log^{16}(2+t).
    \]
    On the other hand, conservation of femionic mass \eqref{eq:mass_fermion} gives
    \[
    \|\rho(t)\|_{L^1}=\|f(t)\|_{L^2}^2 \lesssim\varepsilon_0^2. 
    \]
    Interpolating between $L^1$ and $L^\infty$ gives 
    \[
    \|\rho(t)\|_{L^3}\lesssim \varepsilon^2 \la t\ra ^{-2}\log^\frac{32}{3}(2+t).
    \]
    Commuting \eqref{eq:hartree} twice with the Galilean vector fields yields
    \[
    \left(i\partial_t -\Delta\right) G_j G_ku=\phi G_j G_k u +2t\left(\partial_j\phi\right)G_k u+2t\left(\partial_k\phi\right) G_ju+4t^2 \left(\partial_{jk}\phi\right) u.
    \]
    The standard energy estimate therefore gives
    \begin{align*}
        \partial_t \|G^2u(t)\|_{L^2}&\lesssim t \|E(t)\|_{L^\infty}\|Gu(t)\|_{L^2}+t^2\|\nabla^2 \phi(t)\|_{L^3}\|u(t)\|_{L^6}\\
        &\qquad \lesssim \varepsilon^3 \la t\ra ^{-1}\log^\frac{38}{3}(2+t).
    \end{align*}
    The contribution of $0\leq t\leq 1$ is bounded by $C\varepsilon_0$. With $G_{j}G_ku(0)=-x_jx_k u_0$, integration in time gives
    \begin{align*}
        \|G^2u(t)\|_{L^2}\lesssim\varepsilon_0+\varepsilon^3\log^{\frac{41}{3}}(2+t).
    \end{align*}
    Finally, using Lemma \ref{lem:Riedz_est}, \eqref{eq:phase_G} and \eqref{eq:L6_Gu}, we bound
    \begin{align*}
        \|\nabla F(t)\|_{L^\infty}\lesssim \big\| \nabla(|u(t)|^2)\big\|_{L^{3,1}}\lesssim \|\nabla(e^{i\varsigma}u(t))\|_{L^{6,2}}\|u(t)\|_{L^{6,2}}\lesssim t^{-3}\|Gu(t)\|_{L^2}\|G^2u(t)\|_{L^2},
    \end{align*}
    which proves \eqref{eq:nabla_F}. 
\end{proof}


\subsection{Derivative estimates for the Vlasov profile and field} To establish modified scattering for the Vlasov component, we require bounds for the first derivatives of the free-transport profile and for the spatial derivative of the field generated by the Vlasov density. 
\begin{lemma}\label{lem:weighted_partial_gamma}
    Under the bootstrap assumption \eqref{boot} on $[0,T]$, for every $t\in [0,T]$, we have 
    \begin{equation}
        \big\|\nabla_x \gamma(t)\big\|_{L^\infty_{x,v}}\lesssim \varepsilon_0,\qquad \big\|\nabla_v\gamma(t)\big\|_{L^\infty_{x,v}}\lesssim \varepsilon_0\left(1+\varepsilon^2 \log^{\frac{50}{3}}(2+t)\right).
    \end{equation}
    Moreover, 
    \begin{equation}\label{eq:weighted_derivative_gamma}
        \big\||x| \nabla_x \gamma(t)\big\|_{L^\infty_{x,v}}\lesssim \varepsilon_0\left(1+\varepsilon^2 \log^5(2+t)\right),\,\,\,\big\||x| \nabla_v\gamma(t)\big\|_{L^\infty_{x,v}}\lesssim \varepsilon_0\left(1+\varepsilon^2 \log^\frac{65}{3}(2+t)\right).
    \end{equation}
\end{lemma}
\begin{proof}
    Differentiating \eqref{eq:vlasov_gamma}, we obtain, 
    \begin{equation}
        \partial_t \left(\nabla_x\gamma\right)+\left\{ \nabla_x\gamma, \Psi(t,x+tv)\right\} = t\nabla_x\gamma\cdot (\nabla_x F)(t,x+tv)-\nabla_v\gamma\,\cdot \nabla_xF(t,x+tv),
    \end{equation}
    and 
    \begin{equation}
        \partial_t\left(\nabla_v\gamma\right)+\left\{ \nabla_v\gamma, \Psi(t,x+tv)\right\} = t^2\nabla_x\gamma\cdot (\nabla_x F)(t,x+tv)-t\,\nabla_v\gamma\,\cdot \nabla_xF(t,x+tv).
    \end{equation}
    Lemma \ref{lem:hamiltonian_transport} gives
    \begin{align*}
        \|\nabla_x \gamma(t)\|_{L^\infty_{x,v}}\leq \|\nabla_x \gamma(0)\|_{L^\infty_{x,v}}+\int_0^t \left(s \|\nabla_x \gamma(s)\|_{L^\infty_{x,v}}\|\nabla F(s)\|_{L^\infty} + \|\nabla_v \gamma(s)\|_{L^\infty_{x,v}}\|\nabla F(s)\|_{L^\infty}\right) ds,
    \end{align*}
     and 
     \begin{align}\label{eq:integral_form_nabla_v_gamma}
         \|\nabla_v \gamma(t)\|_{L^\infty_{x,v}}\leq \|\nabla_v \gamma(0)\|_{L^\infty_{x,v}}+\int_0^t \left(s^2 \|\nabla_x \gamma(s)\|_{L^\infty_{x,v}}\|\nabla F(s)\|_{L^\infty} +s \|\nabla_v \gamma(s)\|_{L^\infty_{x,v}}\|\nabla F(s)\|_{L^\infty}\right) ds.
     \end{align}
     We consider 
     \[
     Z(t):= \|\nabla_x \gamma(t)\|_{L^\infty_{x,v}}+(1+t)^{-1} \|\nabla_v \gamma(t)\|_{L^\infty_{x,v}}.
     \]
     The local theory controls the contribution of $0\leq t\leq 1$ and Lemma \ref{lem:G2u} handles $t\geq 1$; hence
     \[
     \int_0^\infty (1+s)\|\nabla F(s)\|_{L^\infty} ds\lesssim \varepsilon^2. 
     \]
     Gr\"onwall inequality then yields
     \[
     Z(t)\lesssim Z(0)\lesssim\varepsilon_0,
     \]
     and thus
     \begin{align*}
         \|\nabla_x \gamma(t)\|_{L^\infty_{x,v}}\lesssim \varepsilon_0.
     \end{align*}
     Returning to \eqref{eq:integral_form_nabla_v_gamma}, we obtain
     \begin{align*}
          \|\nabla_v \gamma(t)\|_{L^\infty_{x,v}}\lesssim  \varepsilon_0+\varepsilon_0 \varepsilon^2 \log^\frac{50}{3}(2+t).
     \end{align*}
     We next propagate the weighted derivatives. Direct computation gives 
     \begin{align*}
         \partial_t \left(x^i \partial_{x_j}\gamma\right)+\left\{x^i \partial_{x_j}\gamma, \Psi \right\}=-tx^i \partial_jF_k\partial_{x_k}\gamma+x^i\partial_jF_k \partial_{v_k} \gamma+t F_i\partial_{x_j}\gamma,
     \end{align*}
     and 
     \begin{align}\label{eq:integral_weighted_derivative}
         \partial_t \left(x^i \partial_{v_j}\gamma\right)+\left\{x^i \partial_{v_j}\gamma, \Psi \right\}=-t^2 x^i \partial_j F_k\partial_{x_k}\gamma+tx^i\partial_j F_k\partial_{v_k}\gamma+t F_i\partial_{v_j}\gamma, 
     \end{align}
     where $\Psi$ and $F$ and $\nabla F$ are evaluated at $(t,x+tv)$. Define
     \[
     \tilde Z(t):=\|x\nabla_x \gamma(t)\|_{L^\infty_{x,v}}+(1+t)^{-1}\|x\nabla_v\gamma(t)\|_{L^\infty_{x,v}}.
     \]
     An argument similar to that used for the unweighted derivatives yields
     \begin{align*}
         &\tilde Z(t) \\
         &\lesssim \varepsilon_0+\int_0^t\left[ (1+s) \|\nabla F(s)\|_{L^\infty}\tilde Z(s)+s \|F(s)\|_{L^\infty}\|\nabla_x \gamma(s)\|_{L^\infty}+\|F(s)\|_{L^\infty}\|\nabla_v \gamma(s)\|_{L^\infty}\right] ds\\
         &\lesssim \varepsilon_0\left(1+\varepsilon^2\log^5(2+t) \right).
     \end{align*}
     In particular, 
     \[
     \|x\nabla_x \gamma(t)\|_{L^\infty_{x,v}}\lesssim \varepsilon_0\left(1+\varepsilon^2\log^5(2+t) \right).
     \]
     Plugging back and returning to \eqref{eq:integral_weighted_derivative} proves \eqref{eq:weighted_derivative_gamma}.
\end{proof}
We now estimate the spatial derivative of the field generated by the Vlasov density. 
\begin{lemma}\label{lem:partial_E}
    Under the bootstrap assumption \eqref{boot} on $[0,T]$, for every $t\in [0,T]$, we have 
    \begin{equation}
        \|\nabla E(t)\|_{L^\infty}\lesssim \varepsilon_0^2 \la t\ra^{-3}\log^\frac{110}{3}(2+t). 
    \end{equation}
\end{lemma}
\begin{proof}
    We first consider $t\geq 1$. Chain rule gives
    \[
     \nabla_x\gamma(t,x-tv,v)=t^{-1}\nabla_v\gamma(t,x-tv,v)-t^{-1}\nabla_v\left[\gamma(t,x-tv,v) \right].
    \]
    Lemma \ref{lem:Riedz_est} and integration by parts then bound
    \begin{align*}
        \|\nabla E(t)\|_{L^\infty}&\lesssim \|\nabla_x \rho(t)\|_{L^{3,1}}\lesssim t^{-1}\|\nabla_v \gamma(t)\|_{L^\infty_{x,v}}\left\|\int_{\R^3}\gamma(t,x-tv,v)dv\right\|_{L^{3,1}_x}\\
        &\qquad \lesssim t^{-1} \|\nabla_v \gamma(t)\|_{L^\infty_{x,v}}\|\la x,v\ra ^{4}\gamma(t)\|_{L^\infty_{x,v}}\left\|\int\la x-tv\ra^{-4}\la v\ra^{-4}dv\right\|_{L^{3,1}_x}
    \end{align*}
    Set 
    \[
    K_t(y):=t^{-3}\left\la \frac{y}{t}\right\ra ^{-4},
    \]
    we then write and estimate the last term by
    \begin{align*}
        \left\|\int\la x-tv\ra^{-4}\la v\ra^{-4}dv\right\|_{L^{3,1}_x}\lesssim \|\la x\ra^{-4}*_yK_t\|_{L^{3,1}_x}\lesssim \|\la x\ra^{-4}\|_{L^1_x}\|K_t \|_{L^{3,1}}\lesssim t^{-2}.
    \end{align*}
    Plugging this back and applying Lemmata \ref{lem:improved_gamma} and \ref{lem:weighted_partial_gamma}, we have
    \begin{align*}
        \|\nabla E(t)\|_{L^\infty}\lesssim \varepsilon_0^2  t^{-3}\log^\frac{110}{3}(2+t).
    \end{align*}
    For $0\leq t\leq 1$, differentiate the density directly:
    \begin{align*}
        \|\nabla E(t)\|_{L^\infty}\lesssim \|\nabla_x \gamma(t)\|_{L^\infty_{x,v}}\|\la x, v\ra^4\gamma(t)\|_{L^\infty}\int_{\R^3}\|\la x-tv\ra ^{-4}\la v\ra^{-4}\|_{L^{3,1}_x}dv\lesssim \varepsilon_0^2. 
    \end{align*}
\end{proof}

We conclude this section by closing the bootstrap and recording the global bounds. 
\begin{proposition}\label{prop:quant_GWP}
    There exists $\varepsilon_*>0$ such that the following holds with $\varepsilon_0\leq \varepsilon_*$, then the Vlasov–Hartree system \eqref{main} admits a unique global strong solution $(u,f)$ in the class of Proposition \ref{prop:LWP}. Moreover, for every $t\geq 0$, $(u,f)$ satisfies
    \begin{equation*}
        \|Gu(t)\|_{L^2}\lesssim\varepsilon_0+\varepsilon_0^3 \log^2(2+t), \qquad \|G^2u(t)\|_{L^2}\lesssim \varepsilon_0+\varepsilon_0^3 \log^{15}(2+t),
    \end{equation*}
    and for $0\leq a,b\leq 4$, 
    \begin{equation*}
        \big\|\la x\ra^a \la v\ra^b \gamma(t)\big\|_{L^\infty_{x,v}}\lesssim \varepsilon_0+\varepsilon_0^3\log^{5a}(2+t),
    \end{equation*}
    \begin{equation*}
        \big\|\nabla_x \gamma(t)\big\|_{L^\infty_{x,v}}\lesssim \varepsilon_0,\qquad \big\|\nabla_v\gamma(t)\big\|_{L^\infty_{x,v}}\lesssim \varepsilon_0\left(1+\varepsilon_0^2 \log^{17}(2+t)\right),
    \end{equation*}
    and
    \begin{equation*}
        \big\||x| \nabla_x \gamma(t)\big\|_{L^\infty_{x,v}}\lesssim \varepsilon_0\left(1+\varepsilon_0^2 \log^5(2+t)\right),\,\,\,\big\||x| \nabla_v\gamma(t)\big\|_{L^\infty_{x,v}}\lesssim \varepsilon_0\left(1+\varepsilon_0^2 \log^{22}(2+t)\right).
    \end{equation*}
    In addition, the fields satisfy 
    \begin{equation*}
        \|\phi(t)\|_{L^\infty}\lesssim\varepsilon_0^2 \la t\ra ^{-1}, \quad  \|E(t)\|_{L^\infty}\lesssim \varepsilon_0^2 \la t\ra^{-2}\log(2+t),\quad \|\nabla E(t)\|_{L^\infty}\lesssim \varepsilon_0^2 \la t\ra^{-3}\log^{37}(2+t).
    \end{equation*}
    and
    \begin{equation*}
        \|F(t)\|_{L^\infty}\lesssim \varepsilon_0^2 \la t\ra^{-2}\log^4(2+t),\qquad \|\nabla _xF(t)\|_{L^\infty}\lesssim \varepsilon_0^2  \la t\ra^{-3}\log^{17}(2+t).
    \end{equation*}
\end{proposition}
\begin{proof}
    Fix a sufficiently large constant $A$ and set $\varepsilon=A\varepsilon_0$. By local well-posedness and continuity, the bootstrap assumptions hold on a nontrivial time interval. The estimates established in this section strictly improve all the bootstrap bounds, provided that $A$ is sufficiently big and $\varepsilon_0$ is sufficiently small. The standard continuity argument therefore propagates these estimates throughout the maximal lifespan $[0,T^*)$. 

    We are left to verify that the continuation criterion in Proposition \ref{prop:LWP} is satisfied. The weighted and derivative estimates for $\gamma$ obtained imply that, on every bounded time interval,
    \[
    \|\la x,v\ra^4 \gamma(t)\|_{L^\infty_{x,v}}+\|\nabla_{x,v}\gamma(t)\|_{L^\infty_{x,v}}
    \]
    remains finite. The $L^2$-bound for $\nabla_{x,v}\gamma(t)$ follows from an $L^2$-energy estimate for the differentiated transport equation. A direct computation (Lemma \ref{lem:weighted_partial_gamma}) gives
    \[
    \partial_t \|\nabla_{x,v}\gamma(t)\|_{L^2_{x,v}}\lesssim \la t\ra^2\|\nabla F(t)\|_{L^\infty}\|\nabla_{x,v}\gamma(t)\|_{L^2_{x,v}};
    \]
    hence Gr\"onwall's inequality and Lemma \ref{lem:G2u} show that $\|\nabla_{x,v}\gamma(t)\|_{L^2_{x,v}}$ stays finite on every bounded time interval. The last quantity to check is the $H^1$-norm of $u$. Indeed, conservation of mass and the energy estimate for the differentiated Hartree equation give 
    \[
    \|\nabla u(t)\|_{L^2}\leq \|\nabla u(0)\|_{L^2}+C\|u_0\|_{L^2}\int_0^t \|E(s)\|_{L^\infty}ds.
    \]
    Hence, Corollary \ref{cor:optimal_decay_E} gives that $\|u(t)\|_{H^1}$ is uniformly bounded in time and proves global well-posedness of \eqref{main}.
\end{proof}


\section{Modified scattering in the Coulomb case} \label{sec:asymptotic}
\subsection{Asymptotics of \eqref{eq:vlasov}} To analyze the Vlasov component, we first determine the asymptotic behavior of the force field generated by $u$. Define
\begin{equation}\label{eq:g}
    g(t,y)\ :=\ t^{\frac{3}{2}}\,e^{\,i\frac{|y|^2 t}{4}}\,u(t,ty),
  \qquad (t,y)\in [1,\infty)\times \R^3.
\end{equation}
A direct computation using \eqref{eq:hartree} gives
\begin{equation}\label{eq:pde_g}
  i\partial_t g(t,y)\ =\phi(t,ty)\,g(t,y)\ +\ \frac{1}{2\sqrt{t}}\,\nabla_y\!\cdot\!\Big(e^{\,i\frac{|y|^2t}{4}}\,G u(t,ty)\Big).
\end{equation}
This motivates us to define the phase
\begin{equation}\label{eq:phase_theta}
    \theta(t,y):=\ \int_{1}^{t}\phi(s,sy)\,ds,
\end{equation}
and the renormalized profile
\begin{equation}\label{eq:h}
  h(t,y)\ :=\ g(t,y)\,e^{i\theta(t,y)}.
\end{equation}
Then
\begin{equation}\label{eq:PDE_h}
  i\partial_t h(t,y)\ =\ \frac{1}{2\sqrt{t}}\,e^{\,i\theta(t,y)}\,\nabla_y\!\cdot\!\Big(e^{\,i\frac{|y|^2t}{4}}\,G u(t,ty)\Big).
\end{equation}
We first establish the asymptotics of $h$ as $t$ goes to infinity:
\begin{proposition}\label{prop:asy_h}
    There exists $h_\infty\in L^2(\R^3)$ such that
    \begin{equation}\label{eq:convergence_h_L2}
        \big\|h(t,y)-h_\infty(y)\big\|_{L^2(\R^3)}\lesssim t^{-1}\log^{15}(2+t),\quad \text{for }t\geq 1.
    \end{equation}
\end{proposition}
\begin{proof}
    First observe that we can rewrite
    \begin{equation}
         i\partial_t h(t,y)=\frac{1}{4\sqrt{t}} \,e^{i\theta(t,y)}\,e^{\,i\frac{|y|^2}{4}t}\,G\cdot Gu(t,ty),
    \end{equation}
    and we bound this by
    \begin{align}\label{eq:partial_h_L2}
        \big\|\partial_t h(t)\big\|_{L^2}\lesssim t^{-\frac{1}{2}} \big\|G^2u(t,ty)\big\|_{L^2(\R^3)}\lesssim t^{-2}\big\|G^2u(t)\big\|_{L^2}.
    \end{align}
    The logarithmic growth of $G^2u$ in Proposition \ref{prop:quant_GWP} gives the integrability of $\big\|\partial_t h(t)\big\|_{L^2}$; hence the claimed convergence.
\end{proof}
Change of variables gives conservation of mass for $g$ and $h$:
\begin{equation}\label{eq:L^2_h_g}
    \|h(t)\|_{L^2}=\|g(t)\|_{L^2}
=t^{\frac{3}{2}}\|u(t,ty)\|_{L^2_y}
=\|u_0\|_{L^2}.
\end{equation}
Differentiating $g$ with respect to $y$, we have
\begin{align}\label{eq:derivatives-g}
    \nabla_y\, g(t,y)=\frac{1}{2}\, t^\frac{3}{2}\,e^{i\frac{|y|^2t}{4}}Gu(t,ty),\qquad \nabla^2 _y \,g(t,y) =\frac{1}{4} \, t^\frac{3}{2}\,e^{i\frac{|y|^2t}{4}} G^2u(t,ty).
\end{align}
The slow growth of $Gu$ and $G^2u$ in Proposition \ref{prop:quant_GWP} and change of variable bound 
\begin{align*}
    \big\|\nabla_y \,g(t)\big\|_{L^2}\lesssim \|Gu(t)\|_{L^2}\lesssim 
    \log^2(2+t), \qquad \big\|\nabla^2 _y\,g(t)\big\|_{L^2}\lesssim \|G^2u(t)\|_{L^2}\lesssim \log^{15}(2+t).
\end{align*}
\eqref{eq:L6_Gu} bounds
\begin{align*}
    \big\|g(t)\big\|_{L^6} =t \|u(t)\|_{L^6}\lesssim \|Gu(t)\|_{L^2}\lesssim \log^2(2+t).
\end{align*}
We also compute
\begin{align*}
    \nabla_y h=e^{i\theta} \left(\nabla_y g+ig\nabla_y \theta \right),
\end{align*}
and 
\begin{align*}
    \partial_{y_j y_k}h=e^{i\theta} \left[\partial_{y_jy_k} g+i \left(\partial_{y_j}g\,\partial_{y_k}\theta+\partial_{y_k}g\,\partial_{y_j}\theta +g\partial_{y_j y_k}\theta \right)- g\partial_{y_j}\theta \,\partial_{y_k} \theta \right].
\end{align*}
Using Proposition \ref{prop:quant_GWP}, we bound
\begin{equation}
    \big\|\nabla \theta(t)\big\|_{L^\infty}\lesssim \int_1^t s \big\|E(s,sy)\big\|_{L^\infty}ds\lesssim \int_1^t s^{-1}\log(2+s)ds\lesssim \log^2(2+t).
\end{equation}
Boundedness of Riesz transform implies
\begin{align*}
    \big\|\nabla^2 \theta(t)\big\|_{L^3}\lesssim \int_1^t s^2 \big\|\nabla^2 (-\Delta)^{-1}\rho(s,sy)\big\|_{L^3}ds\lesssim \int_1^t s^{-1}\log^{12}(2+s)ds\lesssim \log^{13}(2+t).
\end{align*}
Combining these estimates, we bound
\begin{align*}
    \big\|\nabla^2 h(t)\big\|_{L^2}&\lesssim  \big\|\nabla^2 g(t)\big\|_{L^2} +  \big\|\nabla g(t)\big\|_{L^2} \big\|\nabla \theta(t)\big\|_{L^\infty}+\|g(t)\|_{L^6} \big\|\nabla^2 \theta (t)\big\|_{L^3}+\|g(t)\|_{L^2} \big\|\nabla \theta(t)\big\|_{L^\infty}^2 \nonumber \\
    &\lesssim \log^{15}(2+t).
\end{align*}
Combining this with \eqref{eq:L^2_h_g}, we have
\begin{equation}\label{eq:H2_h}
    \|h(t)\|_{H^2}\lesssim \log^{15}(2+t),\quad \text{for }t\geq 1.
\end{equation}
Define
\[
d_n(\cdot):= h(2^{n+1},\cdot)-h(2^n,\cdot), \quad \text{for integers }n\geq 0.
\]
\eqref{eq:partial_h_L2} bounds
\begin{align*}
    \|d_n\|_{L^2}\lesssim \int_{2^n}^{2^{n+1}} \big\|\partial_t h(t)\big\|_{L^2}dt\lesssim 2^{-n}\log^{15} (2+2^n).
\end{align*}
\eqref{eq:H2_h} gives $\|d_n\|_{H^2}\lesssim \log^{15}(2+2^n)$. Interpolating these two estimates, we establish
\begin{equation*}
    \|d_n\|_{H^s}\lesssim 2^{-\frac{2-s}{2}n} \log^{15}(2+2^n).
\end{equation*}
Summing up, we have
\begin{align*}
    \sum_{n\geq 0}\|d_n\|_{H^s}<\infty, \quad \text{for }0\leq s<2.
\end{align*}
As a result, we have proved
\begin{lemma}\label{lem:existence_Hs}
    For all $0\leq s<2$, $h_\infty\in H^s$, with
    \begin{equation}
        \|h_\infty\|_{H^s}\lesssim \|h(1)\|_{H^s}+\sum_{n\geq 0} 2^{-\frac{2-s}{2}n} \log^{15}(2+2^n)\lesssim \log^{15}(3)+(2-s)^{-16}.
    \end{equation}
\end{lemma}
\begin{proof}
    It suffices to prove
    \[
    \sum_{n\geq 0} 2^{-\frac{2-s}{2}n} \log^{15}(2+2^n)\lesssim (2-s)^{-16}.
    \]
    Since $\log(2+2^n)\lesssim 1+n$, the series is bounded by an exponentially damped polynomial series. Writing $\alpha=\frac{\log 2}{2}(2-s)$ for simplicity, one bounds the series by
    \[
    \sum_{n\geq 0} e^{-\alpha n} (1+n)^{15}\lesssim \int _0^\infty e^{-\alpha x}(1+x)^{15}dx\lesssim \alpha^{-16}.
    \]
\end{proof}
With the existence of limit in $H^s$ established, we next discuss the convergence. We have
\begin{proposition}\label{prop:conv_h_H1}
    For all time $t\geq 1$ and any $\kappa\in (0,\frac{1}{2})$, we have
    \begin{equation}\label{eq:H1_convergence_h}
        \big\|h(t)-h_\infty\big\|_{H^1}\lesssim_\kappa t^{-\frac{1}{2}+\kappa}.
    \end{equation}
    More generally, for any $m\in (0,2)$ and any $\kappa>0$, we have
    \begin{equation}\label{eq:general_convergence_h}
        \big\|h(t)-h_\infty\big\|_{H^m}\lesssim_{m,\kappa }t^{-\frac{2-m}{2}+\kappa}.
    \end{equation}
\end{proposition}
\begin{proof}
    For $1< s<2$ and $t\geq 1$, \eqref{eq:H2_h} and Lemma \ref{lem:existence_Hs} give
    \begin{align*}
        \big\|h(t)-h_\infty\big\|_{H^s}\lesssim \|h(t)\|_{H^2}+\|h_\infty\|_{H^s}\lesssim \log^{15}(2+t) +(2-s)^{-16}.
    \end{align*}
    Interpolating this with \eqref{eq:convergence_h_L2}, we have
    \begin{equation}\label{eq:h_Hs_convergence}
    \begin{split}
        \big\|h(t)-h_\infty\big\|_{H^1}&\lesssim \big\|h(t)-h_\infty\big\|_{L^2}^\frac{s-1}{s}\big\|h(t)-h_\infty\big\|_{H^s}^\frac{1}{s}\\
        & \qquad \lesssim t^{-\frac{s-1}{s}}\log^\frac{15(s-1)}{s}(2+t)\left(\log^{15}(2+t)+(2-s)^{-16} \right)^\frac{1}{s}.
    \end{split}
    \end{equation}
    For any $\kappa\in (0,\frac{1}{2})$, set 
    \[
    s_\kappa:=\frac{2}{1+\kappa}\in (1,2).
    \]
    Plugging in \eqref{eq:h_Hs_convergence} using $s=s_\kappa$, we have
    \begin{align*}
        \big\|h(t)-h_\infty\big\|_{H^1}\,\lesssim_\kappa \,t^{-\frac{1}{2}+\frac{\kappa}{2}}\log^{15}(2+t)\,\lesssim_\kappa\, t^{-\frac{1}{2}+\kappa}, 
    \end{align*}
    which proves \eqref{eq:H1_convergence_h}. \eqref{eq:general_convergence_h} can be proved analogously, with $s=1$ replaced by general $m$.
\end{proof}
We now use the convergence of $h(t)$ to study the asymptotics of $F$. With \eqref{eq:h}, write
\begin{align*}
    t^2F(t,ty)=\lambda t^2\int \frac{ty-x}{|ty-x|^3 } |u(t,x)|^2 dx \,=\, \lambda \int \frac{y-x}{|y-x|^3}|h(t,x)|^2 dx.
\end{align*}
Inspired by this, we define
\begin{equation}\label{eq:F_infty}
    F_\infty(y):=\lambda \int \frac{y-x}{|y-x|^3} |h_\infty(x)|^2 dx,
\end{equation}
which is well-defined as Lemma \ref{lem:lorentz_estimates} and Lemma \ref{lem:existence_Hs} imply
\begin{equation*}
    \|F_\infty\|_{L^\infty}\lesssim \|h_\infty\|_{L^{6,2}}^2 \lesssim \|h_\infty\|_{H^1}^2 <\infty.
\end{equation*}
Moreover, we have that $t^2F(t,ty)$ converges to $F_\infty$.
\begin{corollary}\label{cor:convergence_F}
    For all time $t\geq 1$, we have
    \begin{equation}\label{eq:F_convergence_rate}
        \|t^2F(t,ty)-F_\infty(y)\|_{L^\infty}\lesssim t^{-\frac{1}{2}+}.
    \end{equation}
\end{corollary}
\begin{proof}
    Applying Lemma \ref{lem:lorentz_estimates}, we bound
    \begin{align*}
         \|t^2F(t,ty)-F_\infty(y)\|_{L^\infty}\lesssim \big\| |h(t)|^2 -|h_\infty|^2 \big\|_{L^{3,1}}\lesssim \|h(t)-h_\infty\|_{H^1}\left(\|h(t)\|_{H^1}+\|h_\infty\|_{H^1}\right).
    \end{align*}
    Using \eqref{eq:H2_h} and Lemma \ref{lem:existence_Hs} and Proposition \ref{prop:conv_h_H1}, we prove \eqref{eq:F_convergence_rate}.
\end{proof}

\medskip

We are ready to establish the asymptotics of \eqref{eq:vlasov} in Theorem \ref{thm:main}. Expanding \eqref{eq:vlasov_gamma} gives
\begin{align}\label{eq:vlasov_gamma_expand}
    \partial_t \gamma=t\,\nabla_x\gamma \cdot F(t,x+tv)-\nabla_v \gamma \cdot F(t,x+tv).
\end{align}
We consider the translated-$\,\gamma$:
\[
\tilde \gamma(t,x,v):=\gamma(t,x-\log(t)F_\infty(v),v);
\]
and \eqref{eq:vlasov_gamma_expand} gives
\begin{equation}
    \begin{split}
        \partial_t \tilde \gamma(t,x,v)&=t^{-1}\nabla_x\gamma(t,x-\log(t)F_\infty(v),v )\cdot \left[t^2 F(t, x+tv-
    \log(t)F_\infty (v))-F_\infty(v) \right]\\
    &\quad -\nabla_v\gamma(t,x-\log(t)F_\infty(v),v)\cdot F(t,x+tv-\log(t)F_\infty (v)).
    \end{split}
\end{equation}
The second term follows directly from Proposition \ref{prop:quant_GWP} and is bounded by
\[
\|\nabla_v\gamma\|_{L^\infty_{x,v}}\,\|F(t)\|_{L^\infty_x}\lesssim t^{-2}\log^{21}(2+t).
\]
For the first term, we first split
\begin{align*}
    F_\infty(v)-t^2 F(t, x&+tv-\log(t)F_\infty (v)) \\
    &= \left(F_\infty(v)-t^2F(t,tv)\right)+ t^2\left( F(t,tv)-F(t, x+tv-\log(t)F_\infty (v)) \right).
\end{align*}
Fundamental theorem of calculus and triangle inequality then bound the first term by:
\begin{align*}
    &t^{-1} \|\nabla_x\gamma\|_{L^\infty_{x,v}} \left\|F_\infty(v)-t^2F(t,tv)\right\|_{L^\infty}+t \|\la x\ra \nabla_x\gamma\|_{L^\infty_{x,v}}\|\nabla_x F(t)\|_{L^\infty}.
\end{align*}
Applying Proposition \ref{prop:quant_GWP}, Corollary \ref{cor:convergence_F} and boundedness of $F_\infty$ then gives
\begin{equation}\label{eq:convergence_gamma}
  \left\|\partial_t \tilde \gamma(t,x,v)\right\|_{L^\infty_{x,v}}\lesssim t^{-\frac{3}{2}+} +t^{-2}\log^{21}(2+t),  
\end{equation}
which is integrable for large time. Writing back in $f$, we prove \eqref{eq:asy_f}. 


\subsection{Asymptotics of \eqref{eq:hartree}} We start by obtaining a limit for the potential.
\begin{lemma}\label{lem:convergence_phi}
   There exists an asymptotic profile $\phi_{\oo}\in L^{\oo}(\R^3)$ such that for $t\gg 1$,
\begin{equation}\label{eq:phi}
    \begin{split}
     |t\phi^0(t,a)-\phi_{\oo}(a)|&\les \varepsilon^2 t^{-\frac{3}{4}}, \\
     |t\phi(t,tx)-\phi_{\oo}(x)|&\les \varepsilon^2 t^{-\frac{1}{40}}.
    \end{split}
\end{equation}
    
\end{lemma}
\begin{proof}
    It follows from Proposition \ref{prop:uniform_effective} that
\begin{equation}
    \begin{split}
        |\pr_t \phi^0_r(t,a)|\les \varepsilon^3\min\{1,r^2\}\la t\ra^{-2+}.
    \end{split}
\end{equation}
From this, there exists $\phi_{\oo}\in L^\oo$ such that
\begin{equation*}
    \begin{split}
        t\phi^0(t,a)=C_0\lambda\int_{s=1}^{t} \int_{r=0}^{\oo} \pr_s \phi^0_r(s,a)\frac{dr}{r^2}ds+\phi^0(t=1,a)\rightarrow \phi_{\oo}(a).
    \end{split}
\end{equation*}
This gives the first inequality. The second inequality follows from \eqref{eq:difference_effective_potential}.
\end{proof}

With the convergence of the potential obtained, we define the phase corrector as
\begin{align*}
    \ell(y):=\int _1^\infty \left(\phi(s,sy)-\frac{1}{s}\phi_\infty(y)\right) ds.
\end{align*}
Lemma \ref{lem:convergence_phi} ensures that  $\ell$ is well-defined in $L^\infty$. Moreover, the asymptotics of the phase $\theta$ follow immediately:
\begin{corollary}\label{cor:asy_theta}
    For all time $t\geq 1$, we have
    \begin{equation}\label{eq:asy_theta}
        \theta(t,y)=\ell(y)+\log(t) \phi_\infty(y)+\mathcal{O}_{L^\infty} (t^{-\frac{1}{40}}).
    \end{equation}
\end{corollary}
\begin{proof}
    We decompose
    \begin{align*}
        \theta(t,y)=\int_1^t \frac{1}{s}\phi_\infty(y)ds+\int_1^\infty \left(\phi(s,sy)-\frac{1}{s}\phi_\infty(y)\right)ds-\int_t^\infty \left(\phi(s,sy)-\frac{1}{s}\phi_\infty(y)\right)ds.
    \end{align*}
    Since Lemma \ref{lem:convergence_phi} directly bounds
    \[
    \left\|\phi(s,sy)-\frac{1}{s}\phi_\infty(y)\right\|_{L^\infty_y}\lesssim s^{-1-\frac{1}{40}},
    \]
    which is time integrable, we prove \eqref{eq:asy_theta}.
\end{proof}

\medskip

We now prove the asymptotics of $u$  stated in Theorem \ref{thm:main}. With \eqref{eq:h}, write
\begin{align*}
    &g(t,y)-h_\infty(y)e^{-i(\log(t)\phi_\infty(y)+\ell(y))} \,\\
    &\qquad = \, h(t,y)\left(e^{-i\theta(t,y)}-e^{-i(\log(t)\phi_\infty(y)+\ell(y))} \right)+\left(h(t,y)-h_\infty(y)\right)e^{-i(\log(t)\phi_\infty(y)+\ell(y))}.
\end{align*}
Sobolev embedding, Proposition \ref{prop:conv_h_H1} with $m>\frac{3}{2}$ and Corollary \ref{cor:asy_theta} then bound the difference:
\begin{align*}
    \|h(t)\|_{L^\infty}\left\| \theta(t,y)-\ell(y)-\log(t) \phi_\infty(y) \right\|_{L^\infty}+\|h(t)-h_\infty\|_{L^\infty}\lesssim t^{-\frac{1}{80}}.
\end{align*}
Hence, we prove
\[
g(t,y)=h_\infty(y)e^{-i(\log(t)\phi_\infty(y)+\ell(y))}+\mathcal{O}_{L^\infty}(t^{-\frac{1}{80}}).
\]
Writing back to $u$ with \eqref{eq:g}, we obtain \eqref{eq:asy_u} with $h_{\oo}(y)$ replaced by $h_{\oo}(y)e^{-i\ell(y)}$.

Finally, we  show that the asymptotic profile $\phi_{\oo}$ and $f_{\oo}$ are related.
\begin{lemma}\label{lem:relation-f-phi}
    We have that, for almost every $a\in \R^3$, 
    \begin{equation}\label{eq:phi_oo}
        \begin{split}
            \phi_{\oo}(a)=\lambda\int \frac{1}{|a-v|}f^2_{\oo}(x,v)dxdv.
        \end{split}
    \end{equation}
\end{lemma}
\begin{proof}
    Conservation of the fermionic mass \eqref{eq:mass_fermion} gives
    \[
    \|\tilde \gamma(t)\|_{L^2_{x,v}}=\|\gamma(t)\|_{L^2_{x,v}}=\|f_0\|_{L^2_{x,v}}.
    \]
Since $\tilde \gamma(t)$ converges pointwise to $f_\infty$, Fatou's lemma shows that $f_\infty\in L^2_{x,v}$. We next upgrade the convergence to $L^2$. Set
\[
R(t):=t^\frac{1}{100}, \qquad B_t:=\left\{(x,v)\,|\, |x|\leq R(t), |v|\leq R(t)\right\}.
\]
Since $F_\infty \in L^\infty$, for all sufficiently large $t$, 
\[
\log(t)\|F_\infty\|_{L^\infty}\leq \frac{1}{2}R(t).
\]
It follows from Proposition \ref{prop:quant_GWP} that 
\[
\iint _{B_t^c}|\tilde \gamma(t,x,v)|^2 dxdv\lesssim R(t)^{-5}\|\la x,v\ra^4\gamma(t)\|_{L^\infty_{x,v}}^2 \lesssim \varepsilon_0^2 t^{-\frac{1}{20}}\log^{40}(2+t),
\]
which converges to 0. On the other hand, by the $L^\infty$ convergence established, we obtain
\[
\iint_{B_t} |\tilde \gamma(t,x,v)-f_\infty(x,v)|^2 dxdv\lesssim R(t)^6\|\tilde \gamma(t)-f_\infty\|_{L^\infty_{x,v}}^2,
\]
which also converges to 0. Since $f_\infty \in L^2_{x,v}$ and $R(t)\to \infty$, its $L^2$ mass outside $B_t$ also converges to 0. Therefore, 
\[
\|\tilde \gamma(t)-f_\infty\|_{L^2_{x,v}}\to 0.
\]
And the conservation of fermionic mass yields
\[
\|f_\infty\|_{L^2_{x,v}}=\|f_0\|_{L^2_{x,v}}.
\]
Finally, define
\[
\rho_t(v):=\int \tilde \gamma^2(t,x,v)dx.
\]
The $L^2$-convergence gives
\[
\|\rho_t-\rho_\infty\|_{L^1_v}\lesssim \|\tilde \gamma(t)-f_\infty\|_{L^2_{x,v}}\left(\|\tilde \gamma(t)\|_{L^2_{x,v}}+\|f_\infty\|_{L^2_{x,v}}\right)\to 0.
\]
By the definition of the effective potential and translation invariance in $x$, we have
\[
t\phi^0(t,a)=\left(V* \rho_t\right)(a).
\]
Lemma \ref{lem:lorentz_estimates} therefore implies
\[
\left\|V*(\rho_t-\rho_\infty)\right\|_{L^{3,\infty}}\lesssim \left\|\rho_t-\rho_\infty\right\|_{L^{1}}\to 0.
\]
Thus $t\phi^0$ converges to $V*\rho_\infty$ in distributions. By Lemma \ref{lem:convergence_phi}, it converges to $\phi_\infty$. Uniqueness of limit gives the claimed identity
\[
\phi_\infty=V*\rho_\infty.
\]
\end{proof}


\section{Global well-posedness and decay estimates in the long-range regime}\label{sec:GWP_small} Throughout this section and the next, we assume that
\[
    V_\alpha(x)=\lambda |x|^{-\alpha}, \qquad \frac{1}{2}<\alpha<1,
\]
and all implicit constants are allowed to depend on $\alpha$ and $\lambda$. The local construction on $[0,1]$ follows from Remark \ref{rmk:long_range}. We thus work on an interval $[1,T]$ contained in the maximal lifespan. 

The argument uses two modifications. On the kinetic side, a time-dependent symplectic transformation absorbs the leading long-range correction to the free Vlasov characteristics induced by the bosonic force. On the Schr\"odinger side, in contrast to \eqref{eq:phase_theta} and \eqref{eq:h}, the phase correction is constructed from a frequency-regularized part of the fermionic potential rather than from the full potential. The regularization compensates for the derivative loss when the phase is differentiated, while the complementary remainder gives a time-integrable error. 

On $[1,T]$, we impose
\begin{subequations}\label{boot:long_range}
\begin{align}
    &\|\nabla^j F(t)\|_{L^\infty}\leq \varepsilon^2t^{-1-\alpha-j},\qquad 0\leq j\leq2,\label{boot:long_range_Linf}\\
    &\|\nabla^3 F(t)\|_{L^2}\leq \varepsilon^2t^{-\frac52-\alpha}.
    \label{boot:long_range_L2}
\end{align}
\end{subequations}
We will show that under assumption \eqref{eq:initial_smallness_small} for initial data with $\varepsilon_0\ll\varepsilon\ll 1$, these bounds improve. Note that the continuity provided by Proposition \ref{prop:LWP} and Remark \ref{rmk:long_range} implies that \eqref{boot:long_range} hold at $T=1$. Moreover, we have
\begin{equation}\label{eq:initial_long_1}
    \begin{split}
    \|\langle x,v\rangle^4\sigma(1)\|_{L^\infty_{x,v}} + \|\la x\ra^3 \nabla_{x,v}\sigma(1)\|_{L^2_{x,v}\cap L^\infty_{x,v}}+
    &\|\la x\ra\nabla_{x,v}^2 \sigma(1)\|_{L^2_{x,v}}+\|u(1)\|_{H^1}+\|G^3u(1)\|_{L^2}\lesssim \varepsilon_0,
    \end{split}
\end{equation}
where $\sigma$ is defined in \eqref{eq:def_sigma_long}. By definition of \eqref{eq:g} and similar calculations as in \eqref{eq:derivatives-g} , we get
\[
\|g(1)\|_{H^3}\les \sum_{j=1}^3\|G^ju(1)\|_{L^2}\les \varepsilon_0.
\]


\subsection{The symplectic correction} The force in the Vlasov equation is long-range at the level of the free-transport profile. We therefore introduce a time-dependent symplectic
transformation which removes the leading part of the potential along the ray $x=tv$. Set 
\begin{equation}\label{eq:long_Theta_Gamma}
    \Theta(t,v):=-\int_1^t\Psi(s,sv)\,ds, \qquad
    \Gamma(t,v):=\nabla_v\Theta(t,v)=\int_1^t sF(s,sv)\,ds.
\end{equation}
Consider the generating function
\[
    S(t,x,v)=x\cdot v+\frac{t}{2}|v|^2-\Theta(t,v)
\]
and define
\begin{equation}\label{eq:def_sigma_long}
    \sigma(t,x,v):=f\big(t,\nabla_vS(t,x,v),v\big)
    =f\big(t,x+tv-\Gamma(t,v),v\big).
\end{equation}
Since $S$ is a generating function, \eqref{eq:def_sigma_long} is a symplectic transformation. A direct computation gives
\begin{equation}\label{eq:sigma_long_range}
    \partial_t\sigma+\{\sigma,\mathcal K\}=0, \qquad
    \mathcal K(t,x,v)=\Psi\big(t,x+tv-\Gamma(t,v)\big)-\Psi(t,tv).
\end{equation}
The subtraction in $\mathcal K$ is the essential point. Indeed, after a $v$-derivative, the potentially nonintegrable term $tF(t,tv)$ appears only through the difference
\[
    t\big(F(t,x+tv-\Gamma(t,v))-F(t,tv)\big),
\]
which gains one derivative of $F$ and the displacement $x-\Gamma(t,v)$.
\begin{lemma}\label{lem:Gamma_long}
Assume the bootstrap bounds \eqref{boot:long_range}. Then, for every $t\in[1,T]$,
\begin{align}
    \|\nabla_v^j\Gamma(t)\|_{L^\infty_v}&\lesssim \varepsilon^2t^{1-\alpha}, \qquad 0\leq j\leq2,
    \label{eq:Gamma_Linf_long_range}\\
    \|\nabla_v^3\Gamma(t)\|_{L^2_v}&\lesssim \varepsilon^2t^{1-\alpha}.\label{eq:Gamma_L2_long_range}
\end{align}
If $\varepsilon$ is sufficiently small, for fixed $x\in\R^3$, $v\mapsto x+tv-\Gamma(t,v)$ is a global $C^1$-diffeomorphism of $\mathbb R^3$, and
\begin{equation}\label{eq:q_t_jacobian_long_range}
    \det\big(tI-\nabla_v\Gamma(t,v)\big)\simeq t^3
\end{equation}
uniformly for $t\in[1,T]$ and $v\in\mathbb R^3$.
\end{lemma}
\begin{proof}
    For $0\leq j\leq2$, differentiation of \eqref{eq:long_Theta_Gamma} and \eqref{boot:long_range} give
    \begin{align*}
        \|\nabla_v^j\Gamma(t)\|_{L^\infty_v}\leq
    \int_1^t s^{j+1}\|\nabla^jF(s)\|_{L^\infty_x}\,ds
    \lesssim\varepsilon^2\int_1^t s^{-\alpha}\,ds
    \lesssim \varepsilon^2t^{1-\alpha}.
    \end{align*}
    For the third derivative, \eqref{boot:long_range} and scaling give
    \begin{align*}
      \|\nabla_v^3\Gamma(t)\|_{L^2_v}\leq \int_1^t s^4 \|\nabla^3F(s,sv)\|_{L^2_v}\,ds = \int_1^t s^{\frac52} \|\nabla^3F(s)\|_{L^2_x}\,ds \lesssim \varepsilon^2\int_1^t s^{-\alpha}\,ds \lesssim \varepsilon^2t^{1-\alpha}.
     \end{align*}
     In particular,
     \[
         t^{-1}\|\nabla_v\Gamma(t)\|_{L^\infty_v}\lesssim \varepsilon^2t^{-\alpha}\ll1, \qquad t\geq 1.
     \] 
     Hence, for all $v,w\in\mathbb R^3$,
     \[
        \big(q_t(v)-q_t(w)\big)\cdot(v-w)\geq \big(t-\|\nabla_v\Gamma(t)\|_{L^\infty}\big)|v-w|^2\geq \frac{t}{2}|v-w|^2.
    \]
Thus $q_t$ is injective and proper; hence it is a global
diffeomorphism. \eqref{eq:q_t_jacobian_long_range} follows from
\[
    tI-\nabla_v\Gamma(t,v)
    =t\big(I-t^{-1}\nabla_v\Gamma(t,v)\big).
\]
\end{proof}

For later use, we record the relevant bounds for the transformed Hamiltonian. 
\begin{lemma}\label{lem:K_long_range} Under the bootstrap assumption \eqref{boot:long_range} on $[1,T]$, for every $t\in [1,T]$, we have
\begin{equation}\label{eq:K_long_range_first_second}
    \begin{split}
         &\qquad \qquad \quad |\nabla_x\mathcal K(t,x,v)|\lesssim \varepsilon^2t^{-1-\alpha},  \qquad |\nabla_v\mathcal K(t,x,v)| \lesssim \varepsilon^2t^{-2\alpha}\la x\ra,\\  &|\nabla_x\nabla_v\mathcal K(t,x,v)|\lesssim \varepsilon^2t^{-1-\alpha},\quad |\nabla_x^2\mathcal K(t,x,v)| \lesssim \varepsilon^2t^{-2-\alpha},\quad  |\nabla_v^2 \mathcal{K}(t,x,v)|\lesssim \varepsilon^2 t^{-2\alpha}\la x\ra.
    \end{split}
\end{equation}
Moreover, for the pure $v$-derivatives, we have for any $k=1,2,3$
\begin{equation}\label{eq:dvK}
   |\nabla_v^k\mathcal{K}|\les \varepsilon^2 t^{-\alpha}. 
\end{equation}
In addition, all third derivatives containing at least one $x$-derivative satisfy
\begin{equation}\label{eq:K_mixed_third_long_range}
    \sum_{\substack{|\mu|+|\nu|=3\\|\mu|\geq1}}
    |\partial_x^\mu\partial_v^\nu\mathcal K|\lesssim
    \varepsilon^2t^{-1-\alpha}.
\end{equation}
\end{lemma}
\begin{proof}
Write, for simplicity,
\[
X:=x+tv-\Gamma(t,v).
\]
Direct differentiation gives
\[
\nabla_x\mathcal K=-F(t,X), \qquad \nabla_v\mathcal K =-t\big(F(t,X)-F(t,tv)\big) +(\nabla_v\Gamma)^TF(t,X),
\]
and
\[
|\nabla_x^2\mathcal K| \lesssim |\nabla F(t,X)|, \qquad |\nabla_x\nabla_v\mathcal K| \lesssim\big(t+|\nabla_v \Gamma|\big) |\nabla F(t,X)|.
\]
Moreover,
\begin{align*}
|\nabla_v^2\mathcal K| \lesssim \, t^2|\nabla F(t,X)-\nabla F(t,tv)|+\big(t|\nabla_v\Gamma|+|\nabla_v\Gamma|^2\big)
|\nabla F(t,X)| +|\nabla_v^2\Gamma||F(t,X)|.
\end{align*}
Applying the mean-value theorem to the differences of $F$ and of $\nabla F$ above and using \eqref{boot:long_range} together with Lemma \ref{lem:Gamma_long} gives \eqref{eq:K_long_range_first_second}.

For the third-order derivatives containing at least one $x$-derivative, the term $\Psi(t,tv)$ vanishes after differentiation in $x$, and hence
\[
|\nabla_x^3\mathcal K| \lesssim |\nabla^2F(t,X)|,
\]
\[
|\nabla_x^2\nabla_v\mathcal K|\lesssim \big(t+|\nabla_v\Gamma|\big) |\nabla^2F(t,X)|,
\]
and
\[
|\nabla_x\nabla_v^2\mathcal K| \lesssim
\big(t+|\nabla_v\Gamma|\big)^2|\nabla^2F(t,X)|+|\nabla_v^2 \Gamma||\nabla F(t,X)|.
\]
The bootstrap assumptions and Lemma \ref{lem:Gamma_long} therefore imply
\[
|\nabla_x^3\mathcal K| \lesssim \varepsilon^2t^{-3-\alpha},
\qquad
|\nabla_x^2\nabla_v\mathcal K| \lesssim \varepsilon^2t^{-2-\alpha},
\qquad
|\nabla_x\nabla_v^2\mathcal K| \lesssim \varepsilon^2t^{-1-\alpha},
\]
which proves \eqref{eq:K_mixed_third_long_range}.
\end{proof}
\begin{remark}\label{rmk:K_vvv}
It follows from \eqref{eq:K_mixed_third_long_range} that all third-order derivatives of $\mathcal K$ containing at least one $x$-derivative are integrable in $L^\infty_{x,v}$ in time. The only remaining term is $\nabla_v^3\mathcal K$. Direct differentiation shows that, apart from terms bounded by $\varepsilon^4t^{-2\alpha}$ in $L^\infty_{x,v}$, the only contributions to $\nabla_v^3\mathcal K$, which require the top-order estimate on $F$ are, schematically,
\[
t^3\big(\nabla^2F(t,x+tv-\Gamma(t,v))-\nabla^2F(t,tv)\big) \qquad\text{and}\qquad F(t,x+tv-\Gamma(t,v))\nabla_v^3 \Gamma(t,v).
\] 
By the fundamental theorem of calculus and the change of variables
\[
v\mapsto tv+\theta(x-\Gamma(t,v)),\qquad \theta\in[0,1],
\]
whose Jacobian is comparable to $t^3$, we have
\[
t^3 \big\| \nabla^2F(t,x+tv-\Gamma(t,v))-\nabla^2F(t,tv)
\big\|_{L^2_v} \lesssim \varepsilon^2t^{-1-\alpha}\la x\ra +\varepsilon^4t^{-2\alpha}.
\]
Moreover,
\[
\|F(t,x+tv-\Gamma(t,v))\nabla_v^3\Gamma(t,v)\|_{L^2_v} \lesssim
\varepsilon^4t^{-2\alpha}.
\]
Hence the remaining third-order terms admit the required weighted $L^2_v$ control with integrable time decay, since $\alpha>1/2$. Summing up, we have
\begin{equation}\label{eq:K_vvv}
    \nabla_v^3\mathcal K = O_{L^\infty_{x,v}} \left(\varepsilon^4t^{-2\alpha}\right)+O_{L^2_v}\left(\varepsilon^2t^{-1-\alpha}\langle x\rangle+\varepsilon^4t^{-2\alpha}
    \right),
\end{equation}
where the \(L^2_v\) bound in the second decomposition is uniform in \(x\)
after retaining the displayed factor \(\langle x\rangle\).

\end{remark}

\medskip

We next propagate the moments and derivatives of $\sigma$.
\begin{lemma}\label{lem:sigma_bounds_long_range}
If the bootstrap bounds \eqref{boot:long_range} and \eqref{eq:initial_long_1} hold, then for every $t\in [1,T]$, 
\begin{equation}\label{eq:sigma_zero_long_range}
    \|\langle x,v\rangle^4 \sigma(t)\|_{L^\infty_{x,v}}
    \lesssim \varepsilon_0,
\end{equation}
\begin{equation} \label{eq:sigma_first_long_range}
    \|\langle x\rangle^3\nabla_x\sigma(t)\|_{L^2_{x,v}\cap L^\infty_{x,v}} +\|\langle x\rangle^2\nabla_v\sigma(t)\| _{L^2_{x,v}\cap L^\infty_{x,v}} \lesssim \varepsilon_0,
\end{equation}
and 
\begin{equation}\label{eq:sigma_second_long_range}
\|\nabla_x^2\sigma(t)\|_{L^2_{x,v}}+\|\nabla_x\nabla_v\sigma(t)\|_{L^2_{x,v}}+\|\nabla_v^2\sigma(t)\|_{L^2_{x,v}}
    \lesssim \varepsilon_0.
\end{equation}
\end{lemma}
\begin{proof}
    Applying Lemma \ref{lem:hamiltonian_transport} to $\la x,v\ra^4\sigma$ and using Lemma \ref{lem:K_long_range}, we obtain 
    \[
     \frac{d}{dt} \|\langle x,v\rangle^4\sigma(t)\|_{L^\infty} \lesssim \left( \varepsilon^2t^{-1-\alpha} +\varepsilon^4t^{-2\alpha} \right)\|\langle x,v\rangle^4\sigma(t)\|_{L^\infty}.
    \] 
    Since $\alpha>\frac{1}{2}$, the coefficients are integrable. Hence \eqref{eq:sigma_zero_long_range} follows from \eqref{eq:initial_long_1} and Gr\"onwall's inequality. 

    We next differentiate \eqref{eq:sigma_long_range} once. For $p=2$ or $p=\infty$, Lemma \ref{lem:hamiltonian_transport} and Lemma \ref{lem:K_long_range} give 
    \begin{align*}
     \frac{d}{dt}\|\langle x\rangle^3\nabla_x\sigma\|_{L^p}
     &\lesssim \left( \varepsilon^2t^{-1-\alpha} +\varepsilon^4t^{-2\alpha} \right) \|\langle x\rangle^3\nabla_x\sigma\|_{L^p} +\varepsilon^2t^{-2-\alpha}
     \|\langle x\rangle^3\nabla_v\sigma\|_{L^p},\\
     \frac{d}{dt}\|\langle x\rangle^3\nabla_v\sigma\|_{L^p}
     &\lesssim \left( \varepsilon^2t^{-1-\alpha} +\varepsilon^4t^{-2\alpha} \right)
     \|\langle x\rangle^3\nabla_v\sigma\|_{L^p} +\varepsilon^2t^{-\alpha} \|\langle x\rangle^3\nabla_x\sigma\|_{L^p},\\
     \frac{d}{dt}\|\langle x\rangle^2\nabla_v\sigma\|_{L^p}
     &\lesssim \left( \varepsilon^2t^{-1-\alpha}+\varepsilon^4t^{-2\alpha}\right)\left(\|\langle x\rangle^2\nabla_v\sigma\|_{L^p}+\|\langle x\rangle^3\nabla_x\sigma\|_{L^p}\right).
\end{align*}
The first two inequalities, \eqref{eq:initial_long_1} and Gr\"onwall on $\|\langle x\rangle^3\nabla_x\sigma\|_{L^p}+t^{-1}\|\langle x\rangle^3\nabla_v\sigma\|_{L^p}$ imply
\[
    \|\langle x\rangle^3\nabla_x\sigma(t)\|_{L^p} \lesssim\varepsilon_0, \qquad \|\langle x\rangle^3\nabla_v\sigma(t)\|_{L^p} \lesssim \varepsilon_0\big(1+\varepsilon^2t^{1-\alpha}\big).
\]
The third differential inequality and Gr\"onwall's inequality then give
\[
   \|\langle x\rangle^2\nabla_v\sigma(t)\|_{L^p}  \lesssim\varepsilon_0.
\]
It remains to estimate the second derivatives. 
Direct differentiation of \eqref{eq:sigma_long_range} gives
\begin{equation}\label{eq:A}
    \begin{split}
     \pr_t (\nabla_v^2\sigma)+\{\nabla_v^2\sigma,\mathcal{K} \}&=   -2\{\nabla_v\sigma,\nabla_v\mathcal{{K}} \}-\{\sigma,\nabla_v^2\mathcal{K} \}. \\
     \pr_t (\nabla_x\nabla_v\sigma)+\{\nabla_x\nabla_v\sigma,\mathcal{K} \} & =-\{\nabla_x\sigma,\nabla_v\mathcal{K} \}-\{\sigma,\nabla_x\nabla_v\mathcal{K}\}-\{\nabla_v\sigma,\nabla_x\mathcal{K} \}. \\
      \pr_t (\nabla_x^2\sigma)+\{\nabla_x^2\sigma,\mathcal{K} \}&=   -2\{\nabla_x\sigma,\nabla_x\mathcal{{K}} \}-\{\sigma,\nabla_x^2\mathcal{K} \}.
    \end{split}
\end{equation}
and with one moment $\la x\ra$, we have
\begin{equation*}
    \begin{split}
      \pr_t (\la x\ra\nabla_v^2\sigma)+\{\la x\ra\nabla_v^2\sigma,\mathcal{K} \}&=   -2\la x\ra \{\nabla_v\sigma,\nabla_v\mathcal{{K}} \}-\la x\ra\{\sigma,\nabla_v^2\mathcal{K} \}+\nabla_v^2\sigma\{\la x\ra,\mathcal{K}\} ,  \\
     \pr_t (\la x\ra\nabla_x\nabla_v\sigma)+\{\la x\ra\nabla_x\nabla_v\sigma,\mathcal{K} \} & =-\la x\ra\{\nabla_x\sigma,\nabla_v\mathcal{K} \}-\la x\ra\{\sigma,\nabla_x\nabla_v\mathcal{K}\}-\la x\ra\{\nabla_v\sigma,\nabla_x\mathcal{K} \}  \\
     &\qquad \qquad +\nabla_{x}\nabla_v \sigma \{\la x\ra,\mathcal{K}\},  \\
      \pr_t (\la x\ra\nabla_x^2\sigma)+\{\la x\ra\nabla_x^2\sigma,\mathcal{K} \}&=   -2\la x\ra\{\nabla_x\sigma,\nabla_x\mathcal{{K}} \}-\la x\ra\{\sigma,\nabla_x^2\mathcal{K} \}+\nabla_x^2\sigma \{\la x\ra,\mathcal{K}\} .     
    \end{split}
\end{equation*}
Using \eqref{eq:dvK}, \eqref{eq:K_long_range_first_second} and \eqref{eq:K_vvv}, one can argue  similarly  as in the first order derivative case that
\begin{equation*}
    \begin{split}
        \|\la x\ra\nabla_x^2\sigma\|_{L^2_{x,v}}\les \varepsilon_0,\qquad \|\la x\ra\nabla_x\nabla_v \sigma\|_{L^2_{x,v}}\les \varepsilon_0 (1+\varepsilon^2t^{1-\alpha}),\\ \|\la x\ra\nabla_v^2\sigma\|_{L^2_{x,v}}\les \varepsilon_0(1+\varepsilon^2t^{1-\alpha}+\varepsilon^4t^{2-2\alpha}).
    \end{split}
\end{equation*}
This together with \eqref{eq:A} gives  \eqref{eq:sigma_second_long_range}.
\end{proof}

As a consequence, one then has the optimal decay for the density. Before proving this, we need several estimates on the Jacobian of the change of variable
\[
y=x-tv+\Gamma(t,v),\quad 
v:=\bar{v}(x,y).
\]
\begin{lemma}
    Under the bootstrap assumption \eqref{boot:long_range}, for every $t\in [1,T]$, we have
\begin{equation}\label{eq:Jacobian-esimate}
    \begin{split}
          &  |\nabla_x\bar v|+|\nabla_y\bar v|  \lesssim \langle t\rangle^{-1},\qquad \bigl|\det[\nabla_y\bar v]\bigr|\lesssim \langle t\rangle^{-3}. \\
           & |\nabla_x^2 \bar{v}|\les \varepsilon^2\la t\ra^{-2-\alpha}, \qquad |\nabla_x\det [\nabla_y \bar{v}]|\les \varepsilon^2\la t\ra^{-4-\alpha}. \\
           &|\nabla_x^2 \det [\nabla_y \bar{v} ]|\les \la t\ra^{-6}|\nabla_v^3\Gamma(t,\bar{v})|+\varepsilon^4\la t\ra^{-5-2\alpha}.
    \end{split}
\end{equation}
    
\end{lemma}
\begin{proof}
To justify the derivatives of the inverse map more explicitly, recall that
\[
y=x-t\bar v(x,y)+\Gamma\bigl(t,\bar v(x,y)\bigr).
\]
Differentiating with respect to \(x\) and \(y\), respectively, gives
\[
\bigl(tI-\nabla_v\Gamma(t,\bar v)\bigr)\nabla_x\bar v=I,\qquad\bigl(tI-\nabla_v\Gamma(t,\bar v)\bigr)\nabla_y\bar v=-I.
\]
Consequently,
\[\nabla_x\bar v=\bigl(tI-\nabla_v\Gamma(t,\bar v)\bigr)^{-1},\qquad\nabla_y\bar v=-\bigl(tI-\nabla_v\Gamma(t,\bar v)\bigr)^{-1}.
\]
By \eqref{eq:q_t_jacobian_long_range}, we therefore have
\begin{equation}\label{eq:inverse_map_first_derivatives}
    |\nabla_x\bar v|+|\nabla_y\bar v|\lesssim \langle t\rangle^{-1},   \qquad\bigl|\det[\nabla_y\bar v]\bigr|\lesssim \langle t\rangle^{-3}.
\end{equation}

Differentiating
\[
\bigl(tI-\nabla_v\Gamma(t,\bar v)\bigr)\partial_{x_k}\bar v=e_k
\]
with respect to \(x_\ell\), we obtain
\[
\bigl(tI-\nabla_v\Gamma(t,\bar v)\bigr)\partial_{x_\ell}\partial_{x_k}\bar v=\Bigl(    \nabla_v^2\Gamma(t,\bar v)[\partial_{x_\ell}\bar v]\Bigr)\partial_{x_k}\bar v.
\]
It follows from \eqref{eq:inverse_map_first_derivatives} that
\begin{equation}\label{eq:inverse_map_second_derivative}
\begin{split}
|\nabla_x^2\bar v|&\lesssim\bigl|\bigl(tI-\nabla_v\Gamma(t,\bar v)\bigr)^{-1}\bigr|\,|\nabla_v^2\Gamma(t,\bar v)|\,|\nabla_x\bar v|^2\lesssim\langle t\rangle^{-3}|\nabla_v^2\Gamma(t,\bar v)|\lesssim\varepsilon^2\langle t\rangle^{-2-\alpha},
\end{split}
\end{equation}
where the last inequality follows from
\eqref{eq:Gamma_Linf_long_range}.

We next differentiate the identity
\[
\bigl(tI-\nabla_v\Gamma(t,\bar v)\bigr)\nabla_y\bar v=-I.
\]
For every \(k\), this gives
\[
\bigl(tI-\nabla_v\Gamma(t,\bar v)\bigr)\partial_{x_k}\nabla_y\bar v
=\Bigl(   \nabla_v^2\Gamma(t,\bar v)[\partial_{x_k}\bar v]\Bigr)\nabla_y\bar v.
\]
Hence,
\begin{equation}\label{eq:mixed_inverse_map_first}
\begin{split}
|\nabla_x\nabla_y\bar v|
&\lesssim
\bigl|   \bigl(tI-\nabla_v\Gamma(t,\bar v)\bigr)^{-1}\bigr||\nabla_v^2\Gamma(t,\bar v)||\nabla_x\bar v||\nabla_y\bar v|
\\
&\lesssim\langle t\rangle^{-3}|\nabla_v^2\Gamma(t,\bar v)|.
\end{split}
\end{equation}
Differentiating this identity once more in \(x\), and using
\eqref{eq:inverse_map_first_derivatives} and
\eqref{eq:inverse_map_second_derivative}, yields
\begin{equation}\label{eq:mixed_inverse_map_second}
\begin{split}
|\nabla_x^2\nabla_y\bar v|&\lesssim\bigl|\bigl(tI-\nabla_v\Gamma(t,\bar v)\bigr)^{-1}\bigr|^2|\nabla_v^3\Gamma(t,\bar v)||\nabla_x\bar v|^2\\
&\quad+\bigl|\bigl(tI-\nabla_v\Gamma(t,\bar v)\bigr)^{-1}\bigr|^2|\nabla_v^2\Gamma(t,\bar v)||\nabla_x^2\bar v|
\\
&\quad+\bigl|\bigl(tI-\nabla_v\Gamma(t,\bar v)\bigr)^{-1}\bigr|^3|\nabla_v^2\Gamma(t,\bar v)|^2|\nabla_x\bar v|^2\\
&\lesssim\langle t\rangle^{-4}|\nabla_v^3\Gamma(t,\bar v)|+\langle t\rangle^{-5}|\nabla_v^2\Gamma(t,\bar v)|^2.
\end{split}
\end{equation}

Since the determinant is cubic in the entries of
\(\nabla_y\bar v\), we have
\[
\bigl|   \nabla_x\det[\nabla_y\bar v]\bigr|\lesssim|\nabla_y\bar v|^2|\nabla_x\nabla_y\bar v|.
\]
Using \eqref{eq:inverse_map_first_derivatives} and
\eqref{eq:mixed_inverse_map_first}, we obtain
\begin{equation}\label{eq:jacobian_first_derivative}
\begin{split}
\bigl| \nabla_x\det[\nabla_y\bar v]
\bigr|&\lesssim\langle t\rangle^{-5}|\nabla_v^2\Gamma(t,\bar v)|\lesssim\varepsilon^2\langle t\rangle^{-4-\alpha}.
\end{split}
\end{equation}
Similarly, differentiating the determinant twice gives
\[
\bigl|  \nabla_x^2\det[\nabla_y\bar v]\bigr|
\lesssim|\nabla_y\bar v|^2|\nabla_x^2\nabla_y\bar v|+|\nabla_y\bar v||\nabla_x\nabla_y\bar v|^2.
\]
Therefore, by \eqref{eq:inverse_map_first_derivatives},
\eqref{eq:mixed_inverse_map_first}, and
\eqref{eq:mixed_inverse_map_second},
\begin{equation}\label{eq:jacobian_second_derivative}
\begin{split}
\bigl|  \nabla_x^2\det[\nabla_y\bar v]\bigr|&\lesssim\langle t\rangle^{-6}|\nabla_v^3\Gamma(t,\bar v)|+\langle t\rangle^{-7}|\nabla_v^2\Gamma(t,\bar v)|^2\\
&\lesssim\langle t\rangle^{-6}|\nabla_v^3\Gamma(t,\bar v)|+\varepsilon^4\langle t\rangle^{-5-2\alpha}.
\end{split}
\end{equation}

\end{proof}

\begin{corollary}\label{cor:rescale_density}
    Under the bootstrap assumption \eqref{boot:long_range} and smallness assumption on initial data \eqref{eq:initial_smallness_small}, for every $t\in [1,T]$, we have
    \begin{equation}\label{eq:decay_density}
        \|\rho(t,x)\|_{L^\infty}\lesssim \varepsilon_0^2 \la t\ra^{-3}, \qquad \|\nabla \rho(t,x)\|_{L^\infty}\lesssim \varepsilon_0^2 \la t\ra^{-4}, \qquad \|\nabla^2 \rho(t,x)\|_{L^2}\lesssim \varepsilon_0^2 \la t\ra^{-\frac{7}{2}}.
    \end{equation}
    In particular, the rescaled density satisfies
    \begin{equation}\label{eq:rescale_density}
        m(t,x):=t^3\rho(t,tx), \qquad \|m(t)\|_{H^2}\lesssim \varepsilon_0^2. 
    \end{equation}
\end{corollary}
\begin{proof}
Using \eqref{eq:sigma_zero_long_range}, \eqref{eq:Jacobian-esimate} and choosing $\varepsilon\ll 1$, we have
\[
\rho(t,x)=\int \sigma^2(t,y,\bar{v}(x,y))\abs{\det\of{\frac{\pr \bar{v}}{\pr y}}}dy \les \varepsilon_0^2 \la t\ra^{-3}.
\]
Similarly, using \eqref{eq:sigma_second_long_range} and \eqref{eq:Jacobian-esimate}, we have
\begin{equation*}
    \begin{split}
 |\nabla_x \rho(t,x)|&\leq \la t\ra^{-3}\left|\int \sigma(t,y,\bar{v}(x,y)) \nabla_v\sigma \cdot \frac{\pr \bar{v}}{\pr x}  dy\right|+\left|\int \sigma^2(t,y,\bar{v}(x,y))\cdot \pr_{x}[\det \nabla_y \bar{v}]dy \right|  \\ &\les \|\la x\ra^4\sigma\|_{L^{\oo}}\|\nabla_v\sigma\|_{L^{\oo}}\la t\ra^{-4}\les \varepsilon_0^2\la t\ra^{-4}.
    \end{split}
\end{equation*}
For the bound of $\nabla_x^2\rho$, we can compute that
\begin{equation}
    \begin{split}
    |\nabla_x^2\rho(t,x)|&\leq \int \sigma^2(t,y,\bar{v})|\nabla_x^2\det[\nabla_y \bar{v}]|dy+2\int |\nabla_v\sigma \sigma |\nabla_x\bar{v}| |\nabla_x\det [\nabla_y \bar{v}]dy \\
    &\quad + \int |\nabla_v^2\sigma||\sigma||\nabla_x \bar{v}|^2 \det [\nabla_y \bar{v}] dy+\int |\nabla_v \sigma|^2 |\nabla_x \bar{v}|^2 \det [\nabla_y \bar{v}]dy \\
    &\quad +\int |\nabla_v\sigma||\sigma||\nabla_x^2 \bar{v}|\det[\nabla_y \bar{v}]dy.
    \end{split}
\end{equation}
Let's  estimate the first term. Using \eqref{eq:Jacobian-esimate}, we have
\begin{equation*}
    \begin{split}
    \int \sigma^2(t,y,\bar{v}) |\nabla_x^2\det[\nabla_y \bar{v}] dy&\les \la t\ra^{-6}    \int \sigma^2(t,y,\bar{v})|\nabla_v^3\Gamma| dy+\la t\ra^{-5-2\alpha} \varepsilon^4\int \sigma^2(t,y,\bar{v})dy.   
    \end{split}
\end{equation*}
Taking $L^2$ norm, we have
\begin{equation*}
\begin{aligned}
\left\|
    \int \sigma^2(t,y,\bar v)\,
    \bigl|\nabla_x^2 \det[\nabla_y \bar v]\bigr|
    \,dy
\right\|_{L_x^2}
&\lesssim
\langle t\rangle^{-6+\frac{3}{2}}
\int
    \left\|
        \sigma^2(t,y,v)\nabla_v^3\Gamma(t,v)
    \right\|_{L_v^2}
\,dy
\\
&\quad
+\varepsilon^4
\langle t\rangle^{-5-2\alpha+\frac{3}{2}}
\int
    \left\|\sigma^2(t,y,v)\right\|_{L_v^2}
\,dy
\\
&\lesssim
\varepsilon^2
\langle t\rangle^{-\frac72-\alpha}
\left\|
    \langle x\rangle^2\sigma
\right\|_{L_{x,v}^\infty}^2.
\end{aligned}
\end{equation*}
Similarly, we have
\begin{equation*}
    \begin{split}
     \norm{\int |\nabla_v^2\sigma\sigma| |\nabla_x\bar{v}|^2\det [\nabla_y \bar{v}]dy }_{L^2_x}&\les \la t\ra^{-5+3/2}\int \|\nabla_v^2\sigma(t,y,v) \cdot \sigma(t,y,v)\|_{L^2_v}    dy\\ &\les \la t\ra^{-7/2} \|\la x\ra^4 \sigma\|_{L^{\oo}_{x,v}}\|\nabla^2_v\sigma\|_{L^2_{x,v}}.
    \end{split}
\end{equation*}
The remaining terms can be handled similarly using \eqref{eq:Jacobian-esimate}. Therefore, using \eqref{eq:sigma_first_long_range}, \eqref{eq:sigma_second_long_range} and \eqref{eq:sigma_zero_long_range}, we have
\begin{equation*}
    \begin{split}
\|\nabla_x^2\rho\|_{L^2_x}\les \la t\ra^{-7/2} (\|\la x\ra^4\sigma\|_{L^{\oo}_{x,v}}\|\nabla_v^2\sigma \|_{L^2_{x,v}}+\|\la x\ra^2\nabla_v\sigma\|_{L^{\oo}_{x,v}}\|\nabla_v\sigma\|_{L^2_{x,v}}+\|\la x\ra^2\sigma\|^2_{L^{\oo}_{x,v}})\les \varepsilon_0^2 \la t\ra^{-7/2} .
    \end{split}
\end{equation*}
Finally, \eqref{eq:rescale_density} follows from interpolation, conservation of mass and rescaling
\begin{equation*}
    \begin{split}
  \|t^{3}\rho(t,ty)\|_{L^2}&=t^{3/2}\|\rho\|_{L^2}\les t^{3/2}\|\rho\|_{L^1}^{1/2}\|\rho\|_{L^{\oo}}^{1/2}\les \varepsilon_0^2.   \\ 
t^3\||\nabla_y(\rho(t,ty)\|_{L^2}&=t^{3+1-3/2}\|\nabla \rho\|_{L^2}\lesssim t^{5/2}
\|\rho\|_{L^2}^{1/2}
\|\nabla^2 \rho\|_{L^2}^{1/2}\les \varepsilon_0^2\\
t^{3}\|\nabla_y^2(\rho(t,ty)\|_{L^2}&=t^{3+2-3/2}\|\nabla^2\rho\|_{L^2}\les \varepsilon_0^2.
    \end{split}
\end{equation*}
\end{proof}


\subsection{The regularized phase correction} 
By homogeneity, the rescaled potential satisfies
\begin{equation}\label{eq:rescale_phi}
    \varphi(t):= V_\alpha \,*m(t), \qquad \phi(t,ty)=t^{-\alpha}\varphi(t,y).
\end{equation}
Define 
\begin{equation}\label{eq:J}
    J(t):=\left(1-t^{-1}\Delta\right)^{-\frac{\alpha}{2}}=\beta_t (D), \qquad \beta_t (\xi):=\left(1+t^{-1}|\xi|^2 \right)^{-\frac{\alpha}{2}}, 
\end{equation}
and decompose
\begin{equation}\label{eq:varphi_decomp}
    \varphi_\lo(t):=J(t)\varphi(t), \qquad \varphi_\hi(t):=\varphi(t)-\varphi_\lo(t).
\end{equation}
We first observe that the reality of $\varphi$ is inherited by $\varphi_\lo$ and $\varphi_\hi$, since $\beta_t(\xi)$ is both even and real-valued. This multiplier $J(t)$ is roughly the identity at frequencies $|\xi|\ll t^{1/2}$ and gains $\alpha$ derivatives at frequencies $|\xi|\gg t^{1/2}$. This balances the two requirements of the phase: its low part must restore the derivative loss and the high part can be treated as an error term. This construction is the driving force to close the bootstrap for the Hartree part. 
\begin{lemma}\label{lem:est_varphi}
    For every $t\geq 1$, 
    \begin{equation}\label{eq:low_phi_H3}
        \big\|\nabla \varphi_\lo(t)\big\|_{H^3}\lesssim \|m(t)\|_{L^1}+\|m(t)\|_{H^2},
    \end{equation}
    \begin{equation}\label{eq:low_phi_H4}
        \big\|\nabla \varphi_\lo(t)\big\|_{H^4}\lesssim t^\frac{\alpha}{2}\big(\|m(t)\|_{L^1}+\|m(t)\|_{H^2}\big),
    \end{equation}
    and 
    \begin{equation}\label{eq:hi_phi_H3}
        \big\|\varphi_\hi(t)\big\|_{H^3}\lesssim t^{-1+\frac{\alpha}{2}}\big(\|m(t)\|_{L^1}+\|m(t)\|_{H^2}\big).
    \end{equation}
\end{lemma}
\begin{proof}
    Since $\widehat{V_\alpha}(\xi)=c(\lambda,\alpha)|\xi|^{-3+\alpha}$, Plancherel and $|\beta_t(\xi)|\leq 1$ give 
    \begin{align*}
        \big\|\nabla \varphi_\lo(t)\big\|_{H^3}^2 &\lesssim \int_{|\xi|\leq 1} |\xi|^{-4+2\alpha} |\hat m(\xi)|^2 d\xi+\int_{|\xi|>1} |\xi|^{2+2\alpha} |\hat m(\xi)|^2 d\xi\\
        &\lesssim \|m(t)\|_{L^1}^2 \int_0^1 r^{-2+2\alpha}\,dr +\|m(t)\|_{H^2}^2
    \end{align*}
The integral is finite precisely when $\alpha>1/2$, proving
\eqref{eq:low_phi_H3}. At high frequencies, we have $|\beta_t(\xi)|\leq t^{\frac\alpha2}|\xi|^{-\alpha}$. Consequently, the high-frequency contribution to $\|\nabla \varphi_\lo(t)\|_{H^4}^2$ is bounded by
\begin{align*}
    t^{\alpha}\int_{|\xi|>1} |\xi|^{10-2\alpha-6+2\alpha} |\hat m(\xi)|^2 d\xi \,\leq \, t^{\alpha}\|m(t)\|_{H^2}^2,
\end{align*}
and the low-frequency contribution is treated as for \eqref{eq:low_phi_H3}. This proves \eqref{eq:low_phi_H4}. 

Finally, 
\begin{equation}\label{eq:beta_bound}
    0\leq1-\beta_t(\xi) \lesssim_\alpha\min\{1,t^{-1}|\xi|^2\},
\end{equation}
because
\[
    1-\beta_t(\xi)=\frac\alpha2\int_0^{t^{-1}|\xi|^2}
    (1+s)^{-1-\frac\alpha2}\,ds\,\leq \, \frac{\alpha}{2}\int_0^{t^{-1}|\xi|^2}1\,ds\lesssim_\alpha\,t^{-1}|\xi|^2.
\]
Interpolation in \eqref{eq:beta_bound} yields
\[
1-\beta_t(\xi)\lesssim_\alpha t^{-\kappa} |\xi|^{2\kappa}, \qquad 0\leq \kappa \leq 1.
\]
Taking $\kappa=1-\frac{\alpha}{2}$, we find $\|\varphi_\hi\|_{H^3}^2$ is bounded by
\[
t^{-2+\alpha} \left(\int_{|\xi|\leq 1} |\xi|^{-2}|\hat m(\xi)|^2 d\xi \,+\int_{|\xi|>1} |\xi|^4|\hat m(\xi)|^2 d\xi \right)\lesssim t^{-2+\alpha}\big(\|m(t)\|_{L^1}+\|m(t)\|_{H^2}\big)^2.
\]
\end{proof}
We now define the regularized phase
\begin{equation}\label{eq:vartheta}
    \vartheta(t,y)=\int_1^t s^{-\alpha} \varphi_\lo(s,y)\,ds,
\end{equation}
and the renormalized profile
\begin{equation}\label{eq:reg_h}
    \h(t,y)=e^{i\vartheta(t,y)}g(t,y).
\end{equation}
A direct computation from \eqref{eq:hartree} and \eqref{eq:g} gives
\begin{equation*}
    i\partial_tg-t^{-2}\Delta_yg
    =t^{-\alpha}\varphi(t,y)g,
\end{equation*}
and 
\[
i\partial_t \mathfrak{h}-t^{-2}\Delta \h =-2it^{-2}\nabla_y\vartheta\cdot \nabla \h-it^{-2}\Delta_y \vartheta \,\h-t^{-2}|\nabla_y \vartheta|^2\,\h+t^{-\alpha} \varphi_\hi\,\h.
\]
Set
\begin{equation}\label{eq:T_vartheta}
     \mathcal{T}_{\vartheta}:=2\nabla_y \vartheta\cdot \nabla +\Delta_y \vartheta.
\end{equation}
This first-order operator $\mathcal T_\vartheta$ is skew-adjoint on $L^2$ from integration by parts:
\[
  \Re\la\,\mathcal T_\vartheta\, w,\,w\ra_{L^2}=0.
\]
Then we write
\begin{equation}\label{eq:h_pde}
    i\partial_t \h -t^{-2}\Delta \h +it^{-2}\,\mathcal{T}_\vartheta\h=\tail\,\h, \qquad \tail:= -t^{-2}|\nabla_y \vartheta|^2+t^{-\alpha}\varphi_\hi.
\end{equation}

\begin{lemma}\label{lem:bdd_h_H3}
    Under the bootstrap assumption \eqref{boot:long_range} on $[1,T]$ and the initial-data assumption \eqref{eq:initial_smallness_small} for initial data, for every $t\in [1,T]$, we have:
    \begin{equation}\label{eq:bdd_h_H3}
        \|\h(t)\|_{H^3} \lesssim \varepsilon_0. 
    \end{equation}
\end{lemma}
\begin{proof}
Since $\varphi_\lo$ and thus $\vartheta$ are real-valued, 
\[
   \|\h(t)\|_{L^2}=\|g(t)\|_{L^2}=\|u(t)\|_{L^2}=\|u_0\|_{L^2}. 
\]
Let $\nu$ denote a multi-index with $1\leq |\nu|\leq 3$. Commuting $\partial^\nu$ with \eqref{eq:h_pde} and using the skew-adjointness of $\mathcal{T}_\vartheta$, self-adjointness of $\Delta$ and the reality of $\tail$, we obtain
\begin{align*}
    \partial_t \big\|\partial^\nu\h(t)\big\|_{L^2}^2 &=-2t^{-2}\Re \left \la \left[\partial^\nu, \mathcal{T}_\vartheta\right]\h, \partial^\nu \h \right\ra+2\Im \left \la \left[\partial^\nu, \tail\right]\h, \partial^\nu \h \right\ra.
\end{align*}
The Leibniz rule gives
\[
  \left[\partial^\nu, \mathcal{T}_\vartheta\right]\h=2\sum_{0<\mu\leq\nu }\binom{\nu}{\mu} (\partial^\mu\nabla_y \vartheta)\cdot \partial^{\nu-\mu}\h+\sum_{0<\mu\leq \nu}\binom{\nu}{\mu} \partial^\mu(\nabla\cdot \nabla_y \vartheta) \partial^{\nu-\mu}\h.
\]
Consequently, Sobolev embedding and the standard product estimates in three dimensions imply
\begin{align*}
    \big\|\left[\partial^\nu, \mathcal{T}_\vartheta\right]\h\big\|_{L^2}\lesssim \big\|\nabla\vartheta(t)\big\| _{H^{ \max\{3,|\nu|+1\}}}\|\h(t)\|_{H^{|\nu|}}, \qquad 1\leq|\nu|\leq 3.
\end{align*}
The standard Moser estimate and Sobolev embedding give
\begin{align*}
    \sum_{1\leq |\nu|\leq 3}\big\|\left[\partial^\nu, \tail\right]\h\big\|_{L^2}\lesssim \|\tail(t)\|_{H^3}\|\h(t)\|_{H^3}.
\end{align*}
After summing over $1\leq |\nu|\leq 3$, we have
\begin{align*}
    \partial_t \big\|\h(t)\big\|_{H^3}^2\lesssim \left(t^{-2}\big\|\nabla_y \vartheta(t)\big\|_{H^4}+\| \tail(t)\|_{H^3}\right)\|\h(t)\|_{H^3}^2.
\end{align*}
By Corollary \ref{cor:rescale_density} and conservation of the fermionic mass,
\[
\sup_{1\leq s\leq T} \|m(s)\|_{L^1}+\|m(s)\|_{H^2}\lesssim \varepsilon_0^2 +\varepsilon^2\lesssim \varepsilon^2. 
\]
Indeed, scaling gives
\[
\|m(t)\|_{L^1}=t^3\|\rho(t,ty)\|_{L^1_y}=\iint _{\R^3\times \R^3} f^2(t,x,v)\,dxdv =\|f_0\|_{L^2_{x,v}}^2.
\]
Lemma \ref{lem:est_varphi} and the definition of $\vartheta$ then yield
\begin{equation}\label{eq:est_vartheta_long}
    \big\|\nabla \vartheta(t)\big\|_{H^3}\lesssim t^{1-\alpha}\varepsilon^2, \qquad \big\|\nabla \vartheta(t)\big\|_{H^4}\lesssim t^{1-\frac{\alpha}{2}}\varepsilon^2.
\end{equation}
The definition of $\tail$ and Lemma \ref{lem:est_varphi} then give
\begin{equation}
    \|\tail(t)\|_{H^3}\leq t^{-2\alpha}\varepsilon^4+t^{-1-\frac{\alpha}{2}}\varepsilon^2.
\end{equation}
It follows that
\begin{align*}
    \partial_t \big\|\h(t)\big\|_{H^3}^2\lesssim \left(t^{-1-\frac{\alpha}{2}} \varepsilon^2+t^{-2\alpha} \varepsilon^4 \right)\|\h(t)\|_{H^3}^2.
\end{align*}
Both time weights are integrable because $\alpha>\frac{1}{2}$. Gr\"onwall inequality then implies
\begin{align*}
    \big\|\h(t)\big\|_{H^3}^2\leq \big\|g(1)\big\|_{H^3}^2\exp\left(C_\alpha\varepsilon^2+C_\alpha\varepsilon^4 \right).
\end{align*}
The desired estimate \eqref{eq:bdd_h_H3} follows from \eqref{eq:initial_long_1} and smallness of $\varepsilon$.
\end{proof}

\medskip

Scaling and the reality of $\vartheta$ give
\[
|u(t,ty)|^2= t^{-3}\left|\h\left(t,y\right) \right|^2.
\]
Using the homogeneity of $V_\alpha$, we therefore have, for $0\leq j\leq 2$, 
\[
\nabla^jF(t,ty)=-t^{-1-\alpha-j}\nabla^{j+1}\left(V_\alpha \,*|\h(t)|^2\right)(y).
\]
Since $\widehat{V_\alpha}(\xi)=c(\lambda,\alpha) |\xi|^{-3+\alpha}$, Fourier inversion yields
\begin{align*}
    \big\|\nabla^{j+1} \left(V_\alpha \,* |\h(t)|^2\right)\big\|_{L^\infty_y}\lesssim \int |\xi|^{\alpha+j-2} \Big|\,\widehat{|\h|^2}(t, \xi)\Big| d\xi.
\end{align*}
At low frequencies, 
\[
\int_{|\xi|\leq 1} |\xi|^{\alpha+j-2} \Big|\,\widehat{|\h|^2}(t, \xi)\Big| d\xi\,\lesssim\, \big\| |\h(t)|^2 \big\|_{L^1} \int_0^1 r^{\alpha+j}dr\, \lesssim \,\|\h(t)\|_{L^2}^2. 
\]
At high frequencies, the Cauchy–Schwarz inequality gives
\[
\int_{|\xi|> 1} |\xi|^{\alpha+j-2} \Big|\,\widehat{|\h|^2}(t, \xi)\Big| d\xi\,\lesssim\, \left(\int_{|\xi|>1} |\xi|^{2(\alpha+j-5)}d\xi\right)^\frac{1}{2} \big\| |\h(t)|^2\big\|_{H^3}\,\lesssim \|\h(t)\|_{H^3}^2.
\]
The frequency integral is finite for $0\leq j\leq 2$ because $\alpha<1$. Therefore, we have
\begin{equation}\label{eq:improved_boot_order2}
    \big\|\nabla^j F(t)\big\|_{L^\infty}\lesssim t^{-1-\alpha-j}\varepsilon_0^2, \qquad 0\leq j\leq 2.
\end{equation}
For the top-order estimate, the same scaling gives 
\[
\big\|\nabla^3 F(t)\big\|_{L^2}= t^{-\frac{5}{2}-\alpha} \big\|\nabla^4 \left(V_\alpha *|\h(t)|^2\right) \big\|_{L^2}\,\lesssim t^{-\frac{5}{2}-\alpha}  \big\| |\nabla|^{1+\alpha}\left(|\h(t)|^2\right)\|_{L^2}\,\lesssim t^{-\frac{5}{2}-\alpha} \|\h(t)\|_{H^3}^2.
\]
Consequently, we proved
\begin{equation}\label{eq:improved_boot_top}
    \big\|\nabla^3 F(t)\big\|_{L^2}\lesssim t^{-\frac{5}{2}-\alpha} \varepsilon^2_0.
\end{equation}
Choosing $\varepsilon_0\ll\varepsilon$,  \eqref{eq:improved_boot_order2} and \eqref{eq:improved_boot_top} strictly improve both bootstrap bounds and close the bootstrap. 


\section{Modified scattering in the long-range regime}\label{sec:asymptotic_small}
\subsection{Asymptotics of \eqref{eq:vlasov}}
We begin by identifying the limiting bosonic profile that determines the leading long-range correction to the Vlasov characteristics.
\begin{lemma}\label{lem:H1_h_convergence_long}
    There exists $\h_\infty \in H^1(\R^3)$ such that
    \begin{equation}\label{eq:H1_h_conv_long}
        \big\|\h(t)-\h_\infty\big\|_{H^1}\lesssim t^{1-2\alpha},
    \end{equation}
    and 
    \begin{equation}\label{eq:H1_h2_convergence_long}
        \big\| |\h(t)|^2 -|\h_\infty|^2\big\|_{H^1}\lesssim t^{-\alpha}.
    \end{equation}
\end{lemma}
\begin{proof}
    By the symbol bound \eqref{eq:beta_bound} for $1-\beta_t$, similar low-high frequency analysis as in Lemma \ref{lem:est_varphi} gives
    \begin{align*}
        \big\|\varphi_\hi(t)\big\|_{H^2}^2 \lesssim t^{-2} \int \la \xi\ra^4|\xi|^{-2+2\alpha} |\hat m(\xi)|^2d\xi\,\lesssim t^{-2} \left( \|m(t)\|_{L^1}^2+\|m(t)\|_{H^2}^2\right).
    \end{align*}
    Consequently, 
    \begin{equation}\label{eq:hi_phi_H2}
        \big\|\varphi_\hi(t)\big\|_{H^2}\lesssim  t^{-1}\varepsilon_0^2.
    \end{equation}
    Using the standard product estimates $\|fg\|_{H^1}\lesssim \|f\|_{H^2}\|g\|_{H^1}$, uniform boundedness of $\h$ in $H^3$ from Lemma \ref{lem:bdd_h_H3} and \eqref{eq:est_vartheta_long}, we obtain
    \begin{equation}\label{eq:T_theta_est}
        \big\|\mathcal{T}_\vartheta\h(t)\big\|_{H^1}\lesssim \|\nabla\vartheta(t)\|_{H^2}\|\h(t)\|_{H^2}\lesssim t^{1-\alpha}.
    \end{equation}
    Similarly, \eqref{eq:hi_phi_H2} bounds
    \[
    \big\|\tail(t)\h(t)\big\|_{H^1}\lesssim \left(t^{-2}\|\nabla \vartheta(t)\|_{H^2}^2+t^{-\alpha}\|\varphi_\hi(t)\|_{H^2}\right) \|\h(t)\|_{H^1}\lesssim t^{-2\alpha}+t^{-1-\alpha}.
    \]
    It follows from \eqref{eq:h_pde} that
    \begin{align*}
        \big\|\partial_t \h(t)\big\|_{H^1}\lesssim t^{-2}\|h(t)\|_{H^3} +t^{-2}\big\|\mathcal{T}_\vartheta\h(t)\big\|_{H^1}+\big\|\tail(t)\h(t)\big\|_{H^1}\lesssim t^{-2\alpha}.
    \end{align*}
    Since $\alpha>\frac{1}{2}$, this bound is integrable in time. Hence $\h(t)$ converges in $H^1$ to some $\h_\infty\in H^1$, and 
    \[
    \big\|\h(t)-\h_\infty\big\|_{H^1}\,\lesssim \int_t^\infty \big\|\partial_t \h(s)\big\|_{H^1}ds\,\lesssim\, t^{1-2\alpha}.
    \]

    \smallskip

    Since $\tail$ is real-valued, it does not contribute to the evolution of $|\h|^2$. Taking the real part of $\bar{\h} \partial_t \h$ using \eqref{eq:h_pde}, we obtain
    \[
    \partial_t\big(|\h|^2 \big)=2t^{-2}\Im \left(\bar \h \,\Delta \h\right)-2t^{-2}\Re \left(\bar \h\,\mathcal{T}_\vartheta\h\right).
    \]
    With the standard product estimate and \eqref{eq:T_theta_est}, we bound
    \begin{align}\label{eq:H1_h2_differential}
        \big\| \partial_t\big(|\h(t)|^2 \big)\big\|_{H^1}\lesssim t^{-2}\left(\|\Delta\h(t)\|_{H^1}\|\h(t)\|_{H^2}+ \big\|\mathcal{T}_\vartheta\h(t)\big\|_{H^1}\|\h(t)\|_{H^2}\right)\lesssim t^{-2}+t^{-1-\alpha}.
    \end{align}
    Thus, $\partial_t\big(|\h|^2 \big)$ is integrable in $H^1$, hence $|\h(t)|^2$ converges in $H^1$. Since $\h(t)\to \h_\infty$ in $H^1$, 
    \[
    \Big\| \big|\h(t)\big|^2-\big|\h_\infty\big|^2 \Big\|_{L^1}\lesssim \big\|\h(t)-\h_\infty\big\|_{L^2}\left(\|\h(t)\|_{L^2}+\|\h_\infty\|_{L^2} \right)\, \longrightarrow\, 0, \quad t\to \infty.
    \]
    Hence, the uniqueness of limit in the distributional sense yields that the $H^1$-limit of $|\h(t)|^2$ is $|\h_\infty|^2$. Integrating the preceding differential estimate \eqref{eq:H1_h2_differential} gives \eqref{eq:H1_h2_convergence_long}.
    
\end{proof}

We now use the convergence of $\h(t)$ to study the asymptotics of $F$. With \eqref{eq:reg_h}, we write 
\[
t^{1+\alpha}F(t,ty)=\lambda \alpha \int \frac{y-x}{|y-x|^{2+\alpha}}|\h(t,x)|^2dx
\]
and define
\[
F_{\oo}(y)=\lambda \alpha \int \frac{y-x}{|y-x|^{2+\alpha}}|\h_{\oo}(x)|^2dx.
\]
\begin{corollary}
    For all time $t\geq 1$, we have
    \begin{equation}\label{eq:F_convergence_long_range}
        \begin{split}
            \|t^{1+\alpha}F(t,ty)-F_{\oo}(y)\|_{L^{\oo}}\les t^{-\alpha}.
        \end{split}
    \end{equation}
\end{corollary}

\begin{proof}
 By Sobolev embedding and Lemma \ref{lem:H1_h_convergence_long}, we have
\begin{equation*}
    \begin{split}
 \|t^{1+\alpha}F(t,ty)-F_{\oo}(y)\|_{L^{\oo}}\les \|\nabla |\nabla|^{\alpha-3} (|\h(t,x)|^2-|\h_{\oo}(x)|^2 ) \|_{L^{\oo}}   \les \| |\h|^2-|\h_{\oo}|^2\|_{H^1}\les t^{-\alpha}.
    \end{split}
\end{equation*}
 
\end{proof}

We are ready to establish the asymptotics of \eqref{eq:vlasov} in Theorem \ref{thm:long_range_alpha}. We first show that $\sigma$ converges by computing
\begin{equation*}
    \begin{split}
        \pr_t\sigma=\{\mathcal{K},\sigma\}=\nabla_x\mathcal{K}\cdot \nabla_v\sigma-\nabla_v\mathcal{K}\cdot \nabla_x\sigma.
    \end{split}
\end{equation*}
Using \eqref{eq:sigma_first_long_range} and \eqref{eq:K_long_range_first_second}, we have
\[
\|\pr_t\sigma\|_{L^{\oo}_{x,v}}\les \la t\ra^{-2\alpha}.
\]
Therefore, there exists a function $f_{\oo}\in L^{\oo}_{x,v}$ such that
\begin{equation*}
    \begin{split}
        \|\sigma(t,x,v)-f_{\oo}(x,v)\|_{L^{\oo}}\les \la t\ra^{1-2\alpha}.
    \end{split}
\end{equation*}
It remains to show that
\begin{equation*}
    \begin{split}
    \|\widetilde{\sigma}(t,x,v)-\sigma(t,x,v)\|_{L^{\oo}}\les \la t\ra^{1-2\alpha}. \\
        \widetilde{\sigma}(t,x,v):= f\big(t,x+tv- \frac{t^{1-\alpha}-1}{1-\alpha}F_\infty(v),v\big).
    \end{split}
\end{equation*}
This follows from $\|\nabla_x f\|_{L^{\oo}_{x,v}}=\|\nabla_x\sigma\|_{L^{\oo}_{x,v}}$ and \eqref{eq:F_convergence_long_range} since
\begin{equation*}
    \begin{split}
|\sigma(t,x,v)-\widetilde{\sigma}(t,x,v)|&\les \|\nabla_xf\|_{L^{\oo}}\left|\Gamma(t,v)-\frac{t^{1-\alpha}-1}{1-\alpha}F_{\oo}(v)\right| \\
&\les \varepsilon_0\abs{\int_1^t sF(s,sv)ds-\frac{t^{1-\alpha}-1}{1-\alpha}F_{\oo}(v)}\\
&\les \varepsilon_0 \int_1^t s^{-\alpha}|s^{1+\alpha}F(s,sv)-F_{\oo}(v)|ds\les \varepsilon_0 t^{1-2\alpha}+\mathcal{O}_{L^{\oo}_v}(1).
    \end{split}
\end{equation*}
The proof follows by redefining $f_{\oo}(x,v)$ to absorb the term $\mathcal{O}_{L^{\oo}_v}(1)$ and up to a translation in $x$.


\subsection{Asymptotics of \eqref{eq:hartree}} The leading phase correction for \eqref{eq:hartree} is determined by the asymptotics of the rescaled fermionic potential. We begin by identifying its limiting profile.
\begin{lemma}\label{lem:convergence_phi_long_range}
   There exists an asymptotic profile $\phi_{\oo}\in L^{\oo}(\R^3)$ such that for $t\geq 1$,
\begin{equation}\label{eq:phi_long_range}
    \begin{split}
     |t^{\alpha}\phi^0(t,a)-\phi_{\oo}(a)|&\les \varepsilon^2 t^{-\alpha}, \\
     |t^{\alpha}\phi(t,tx)-\phi_{\oo}(x)|&\les \varepsilon^2 t^{-\alpha}.
    \end{split}
\end{equation}
In particular, the rescaled potential $\varphi$ in \eqref{eq:rescale_phi} converges to $\phi_\infty$:
\begin{equation}\label{eq:varphi_conv}
    |\varphi(t,x)-\phi_{\oo}(x)|\lesssim \varepsilon^2 t^{-\alpha}.
\end{equation}
\end{lemma}
\begin{proof}
Choose a smooth radial function $\chi$ supported in an annulus and a constant $c_{\alpha}$ such that
\[
|z|^{-\alpha}=c_{\alpha}\int_0^{\oo}\chi\of{\frac{z}{R}}\frac{dR}{R^{1+\alpha}}.
\]
Define $\phi^0_r$ and $\phi_R$ as in Section \ref{sec:GWP}. Moreover, we have
\begin{equation*}
    \begin{split}
        \phi^0(t,a)=c_{\alpha,\lambda}t^{-\alpha}\int_0^{\oo} \phi^0_r(t,a)\frac{dr}{r^{1+\alpha}}.
    \end{split}
\end{equation*}
The argument of Lemma \ref{lem:effective_field} gives, for $t\geq 1$,
\begin{equation}
    \begin{split}
       \mathfrak{d}_r:=\phi_{tr}(t,x)-\phi^0_r(x)&=\iint \of{\chi\of{\frac{a-v-\frac{y}{t}}{r}}-\chi\of{\frac{a-v}{r}} }\ga^2(t,y,v)dydv \\
       &=\iint \of{\chi\of{\frac{a-v-\frac{y-\Gamma(t,v)}{t}}{r}}-\chi\of{\frac{a-v}{r}} }\sigma^2(t,y,v)dydv
    \end{split}
\end{equation}
satisfies
\begin{equation}
    \begin{split}
    |\mathfrak{d}_r|\les \min \left\{\frac{1}{rt^{\alpha}} ,\frac{r^2}{t^{\alpha}}\right\}\|\la x,v\ra^4\sigma\|^2_{L^{\oo}_{x,v}}\les \varepsilon_0^2 t^{-\alpha}\min\offf{\frac{1}{r},r^2 }.
    \end{split}
\end{equation}
Therefore,
\begin{equation}\label{eq:one-time-use}
    \begin{split}
        t^{\alpha}|\phi(t,tx)-\phi^0(t,x)|\les \int_0^{\oo} |\mathfrak{d}_r|\frac{dr}{r^{1+\alpha}}\les \varepsilon_0^2t^{-\alpha}.
    \end{split}
\end{equation}
It follows similarly as in  Proposition \ref{prop:uniform_effective} that
\begin{equation}
    \begin{split}
        |\pr_t \phi^0_r(t,a)|\les \varepsilon^2\min\{r^{-1},r^2\}\la t\ra^{-1-\alpha}\|\la x,v\ra^4\sigma\|_{L^{\oo}_{x,v}}^2.
    \end{split}
\end{equation}
It follows that there exists $\phi_{\oo}\in L^\oo$ such that
\begin{equation*}
    \begin{split}
        t^{\alpha}\phi^0(t,a)=C_0\lambda\int_{s=1}^{t} \int_{r=0}^{\oo} \pr_s \phi^0_r(s,a)\frac{dr}{r^{1+\alpha}}ds+\phi^0(t=1,a)\rightarrow \phi_{\oo}(a).
    \end{split}
\end{equation*}
This gives the first inequality. The second inequality follows from \eqref{eq:one-time-use}. 
\end{proof}

We next record a quantitative approximation estimate for the modification operator $J(t)$. 
\begin{lemma}\label{lem:Jt}
    For every $t\geq 1$ and $f\in \dot H^{\frac{3}{2}+2\alpha}$, we have
    \begin{equation}
        \big\|\left(J(t)-\operatorname{Id}\right)f\big\|_{L^\infty}\lesssim t^{-\alpha} \big\|f\big\|_{\dot H^{\frac{3}{2}+2\alpha}}.
    \end{equation}
\end{lemma}
\begin{proof}
    By Fourier inversion and the Cauchy–Schwarz inequality, 
    \begin{align*}
        \big\|(J(t)-\operatorname{Id})f\big\|_{L^\infty} &\lesssim\int |1-\beta_t(\xi)|\,\big|\widehat f(\xi)\big|d\xi \lesssim \left(\int|1-\beta_t(\xi)|^2|\xi|^{-3-4\alpha}d\xi\right)^{1/2}\|f\|_{\dot H^{\frac{3}{2}+2\alpha}}.
    \end{align*}
    Splitting the integral at $|\xi|=t^{\frac{1}{2}}$ and applying \eqref{eq:beta_bound}, we obtain
    \begin{align*}
        \int|1-\beta_t(\xi)|^2|\xi|^{-3-4\alpha}d\xi\,&\lesssim \, t^{-2}\int_{|\xi|\leq t^\frac{1}{2}}|\xi|^4|\xi|^{-3-4\alpha}d\xi\,+\,\int_{|\xi|>t^{\frac{1}{2}}}|\xi|^{-3-4\alpha}d\xi\\
        &\lesssim \, t^{-2}\,+\,t^{-2}\int_1^{t^\frac{1}{2}}r^{3-4\alpha}dr+\int_{t^\frac{1}{2}}^\infty r^{-1-4\alpha}dr\,\lesssim_\alpha \, t^{-2\alpha},
    \end{align*}
    since $\frac{1}{2}<\alpha<1$. 
\end{proof}
We now determine the asymptotic expansion of the phase correction $\vartheta$. To account for the contribution from the initial time interval $[0,1]$, define
\[
\ell_\alpha(y):=\int_1^\infty s^{-\alpha}\big(\varphi_\lo(s,y)-\phi_\infty(y)\big) ds.
\]
This is well defined in $L^\infty$. Indeed, since $\varphi_{\mathrm{low}}(s)=J(s)\varphi(s)$, Lemma \ref{lem:convergence_phi_long_range} and Lemma \ref{lem:Jt} give
\begin{equation}\label{eq:one-time-2}
    \|\varphi_{\mathrm{low}}(s)-\phi_\infty\|_{L^\infty}\leq\|\varphi(s)-\phi_\infty\|_{L^\infty}
+ \|(J(s)-\operatorname{Id})\varphi(s)\|_{L^\infty}
\lesssim s^{-\alpha} + s^{-\alpha}\|\varphi(s)\|_{\dot H^{\frac{3}{2}+2\alpha}} \lesssim s^{-\alpha},
\end{equation}
following from the uniform boundedness of $\|\varphi(s)\|_{\dot H^{\frac{3}{2}+2\alpha}}$
\[
\|\varphi(s)\|_{\dot H^{\frac{3}{2}+2\alpha}}=\big\| \nabla^{\frac{3}{2}+2\alpha} \left(V_\alpha *m(s)\right)\big\|_{L^2}\lesssim \big\||\nabla|^{3\alpha-\frac{3}{2}}m(s)\big\|_{L^2}\lesssim \|m(s)\|_{H^2}\lesssim \varepsilon_0^2.
 \]
\begin{lemma}\label{lem:asy_vartheta}
    For every $t\geq 1$, we have
    \begin{equation}
        \vartheta(t,x)=\frac{t^{1-\alpha}-1}{1-\alpha}\phi_\infty(y)+\ell_\alpha(y)+\mathcal{O}_{L^\infty}\left(t^{1-2\alpha}\right).
    \end{equation}
\end{lemma}
\begin{proof}
    We first decompose $\vartheta$ using the scaling \eqref{eq:rescale_phi}:
    \begin{align*}
        \vartheta(t,y)=\phi_\infty(y)\int_1^t s^{-\alpha} ds\,+\,\ell_\alpha(y)-\int_t^\infty s^{-\alpha}\big(\varphi_\lo(s,y)-\phi_\infty(y)\big)ds.
    \end{align*}
    The last term follows directly from \eqref{eq:one-time-2}.
\end{proof}

\medskip

We now prove the asymptotics of $u$ stated in Theorem \ref{thm:long_range_alpha}. We first upgrade the convergence of $\h$ from \eqref{eq:H1_h_conv_long} to a pointwise estimate. The uniform boundedness for $\h$ in Lemma \ref{eq:bdd_h_H3}, together with its convergence to $\h_\infty$ in $H^1$, implies that $\h_\infty\in H^3$ and 
\[
\big\|\h(t)-\h_\infty\big\|_{H^3}\lesssim 1.
\]
Indeed, this follows from weak compactness in $H^3$ and the uniqueness of limit in $H^1$. The Gagliardo–Nirenberg inequality in three dimensions gives 
\[
\big\|\h(t)-\h_\infty\big\|_{L^\infty}\lesssim t^{\frac{3}{4}(1-2\alpha)}.
\]
Since $g=e^{-i\vartheta}\h$, Lemma \ref{lem:asy_vartheta} yields
\begin{align*}
    \Big\|g(t,y)- e^{-i\left(\frac{t^{1-\alpha}-1}{1-\alpha}\phi_\infty(y)+\ell_\alpha(y)\right)}\h_\infty(y)\Big\|_{L^\infty}&\lesssim\,\|\h_\infty\|_{L^\infty}\Big\|\vartheta(t,y)-\left(\frac{t^{1-\alpha}-1}{1-\alpha}\phi_\infty(y)+\ell_\alpha(y)\right)\Big\|_{L^\infty}\\
    &\qquad +\|\h(t)-\h_\infty\|_{L^\infty}\,\lesssim\, t^{\frac{3}{4}\left(1-2\alpha\right)}.
\end{align*}
Writing back to $u$ with \eqref{eq:g}, we obtain \eqref{eq:asy_u-long} with $h_{\oo}(y)$ replaced by $h_{\oo}(y)e^{-i\ell_\alpha(y)+\frac{i}{1-\alpha}\phi_\infty(y)}$. The relation between the asymptotic profile $\phi_{\oo} $ and $f_{\oo}$ follows similarly as in Lemma \ref{lem:relation-f-phi}.


\section{The short-range regime}\label{sec:short_range}
In this section, we prove Theorem \ref{thm:short_range}. Throughout the section, we  set
\[
    V_\alpha(x)=\lambda |x|^{-\alpha},\quad 1<\alpha<2,\qquad s:=\alpha-1\in(0,1).
\]
We first record the fractional identities used below. Componentwise, one has
\begin{equation}\label{eq:short_range_fractional_identity}
    \partial_j\partial_k\big(V_\alpha*q\big)=c_{\alpha,\lambda} \mathcal R_j\mathcal R_k|\nabla|^s q,\qquad 1\leq j,k\leq 3.
\end{equation}
In particular, if $F=-\nabla(V_\alpha*|w|^2)$, then
\begin{equation}\label{eq:short_range_fractional_F}
    \partial_kF_j=-c_{\alpha,\lambda}\mathcal R_j\mathcal R_k
    |\nabla|^s\big(|w|^2\big).
\end{equation}


\subsection{Local theory}
For $\frac{3}{2}\leq\alpha<2$, the $H^1$ local argument used below the endpoint $\alpha=\frac{3}{2}$ no longer gives the Lipschitz norm of the bosonic force. The same simultaneous Schr\"odinger--characteristic iteration closes after propagating one additional derivative of $u$. Importantly, the contraction is still performed in the lower $H^1$ norm, and therefore no second derivative of $f_0$ is needed.

\begin{lemma}\label{lem:short_range_local}
Let $1<\alpha<2$. Assume \eqref{eq:initial-short-range} holds
\begin{enumerate}
    \item If $1<\alpha<\frac{3}{2}$, the local construction and continuation criterion of Proposition \ref{prop:LWP} remain valid. More precisely, for some admissible pair $(q,r)$ with $q>2$,
    \[
        u\in C([0,T];H^1)\cap L^q([0,T];W^{1,r}),\qquad F\in L^1([0,T];W^{1,\infty}).
    \]
    \item If $\frac{3}{2}\leq\alpha<2$, there exists $T>0$ such that
    \[
        u\in C([0,T];H^2)\cap L^{\frac83}([0,T];W^{2,4}), \qquad F\in L^1([0,T];W^{1,\infty}),
    \]
    and the weighted and first-derivative norms of $\gamma$ appearing in Proposition \ref{prop:LWP} remain finite. The solution can be continued as long as
    \[
    \|u(t)\|_{H^2} +\|\la x,v\ra^4\gamma(t)\|_{L^\infty_{x,v}} +\|\nabla_{x,v}\gamma(t)\|_{L^2_{x,v}\cap L^\infty_{x,v}}
    \]
    stays finite.
\end{enumerate}
If the assumptions of Theorem \ref{thm:short_range} hold and $\varepsilon_0$ is sufficiently small, then the solution exists on $[0,1]$ and
\begin{align}\label{eq:short_range_initial_interval}
    \sup_{0\leq t\leq1}\Big(&\|u(t)\|_{H^{n_\alpha}}+\|Gu(t)\|_{L^2}+\|G^2u(t)\|_{L^2} +\|\la x,v\ra^4\gamma(t)\|_{L^\infty_{x,v}}+\|\nabla_{x,v}\gamma(t) \|_{L^2_{x,v}\cap L^\infty_{x,v}}\Big) \lesssim_{\alpha,\lambda}\varepsilon_0.
\end{align}
\end{lemma}
\begin{proof}
We first recall the argument for $1<\alpha<\frac{3}{2}$. Choose an admissible pair $(q,r)$ such that
\begin{equation}\label{eq:short_range_local_admissible_pair}
    \frac{3}{1-s}<r<6,\qquad \frac{2}{q}+\frac{3}{r}=\frac{3}{2}.
\end{equation}
Such a choice is possible precisely because $s<\frac{1}{2}$. In particular, $q>2$, $W^{1-s,r}(\R^3)\hookrightarrow L^\infty(\R^3)$, and $W^{1,r}(\R^3)$ is an algebra. By \eqref{eq:short_range_fractional_F}, Sobolev embedding, and boundedness of the Riesz transforms,
\begin{align}\label{eq:short_range_local_nabla_F}
    \|\nabla F_w(t)\|_{L^\infty} &\lesssim_{\alpha,\lambda} \big\|\mathcal R_j\mathcal R_k|\nabla|^s(|w(t)|^2) \big\|_{W^{1-s,r}} \lesssim_{\alpha,\lambda} \big\||w(t)|^2\big\|_{W^{1,r}}\lesssim_{\alpha,\lambda}\|w(t)\|_{W^{1,r}}^2.
\end{align}
Since $|\nabla V_\alpha(z)|\lesssim |z|^{-1-\alpha}$ and $\alpha<2$, a decomposition at $|z|=1$ also gives
\begin{equation}\label{eq:short_range_local_F}
    \|F_w(t)\|_{L^\infty} \lesssim_{\alpha,\lambda}\|w(t)\|_{L^2}^2+ \|w(t)\|_{W^{1,r}}^2.
\end{equation}
Consequently,
\begin{equation}\label{eq:short_range_local_force_integral}
    \|F_w\|_{L^1_tW^{1,\infty}_x([0,T])} \lesssim_{\alpha,\lambda}T\|w\|_{L^\infty_tL^2_x}^2 +T^{1-\frac2q}\|w\|_{L^q_tW^{1,r}_x}^2.
\end{equation}
The analogous difference estimate follows by writing $|w|^2-|\widetilde w|^2$ as a sum of two products. Since both $V_\alpha$ and $\nabla V_\alpha$ belong to $L^1_{\rm loc}(\R^3)$,
\begin{equation}\label{eq:short_range_local_phi}
    \|V_\alpha*\rho_g\|_{W^{1,\infty}} \lesssim_{\alpha,\lambda}
    \|\rho_g\|_{L^1}+\|\rho_g\|_{L^\infty},
\end{equation}
and the same estimate holds for differences. These bounds close exactly the simultaneous Picard iteration used in Proposition \ref{prop:LWP}.

\smallskip

We now consider $\frac{3}{2}\leq\alpha<2$. We use the same iteration, but in
\[
    u\in L^\infty_tH^2_x\cap L^{\frac83}_tW^{2,4}_x,
\]
while keeping the same characteristic norm for $\gamma$. Fix $\beta\in(\frac34,1)$. Since $W^{\beta,4}(\R^3)\hookrightarrow L^\infty$, \eqref{eq:short_range_fractional_F} and the fractional Leibniz rule give
\begin{align}\label{eq:short_range_local_high_nabla_F}
    \|\nabla F_w(t)\|_{L^\infty} &\lesssim_{\alpha,\lambda} \big\|\mathcal R_j\mathcal R_k|\nabla|^s(|w(t)|^2)\big\|_{W^ {\beta,4}}\lesssim_{\alpha,\lambda}\big\||w(t)|^2\big\|_{W^{s+\beta,4}}
    \lesssim_{\alpha,\lambda}\|w(t)\|_{W^{2,4}}^2.
\end{align}
Here $s+\beta<2$ because $s<1$ and $\beta<1$. Together with the direct estimate for $F_w$, this yields
\begin{equation}\label{eq:short_range_local_high_force_integral}
    \|F_w\|_{L^1_tW^{1,\infty}_x([0,T])}\lesssim_{\alpha,\lambda}
    T\|w\|_{L^\infty_tL^2_x}^2+T^{\frac14}\|w\|_{L^{\frac83}_tW^{2,4}_x}^2.
\end{equation}
It remains to verify that the fermionic potential acts on $H^2$. For $0\leq t\leq1$, the weighted and first-derivative bounds for $\gamma$ give
\begin{equation}\label{eq:short_range_local_rho_H1}
    \|\rho(t)\|_{L^1\cap L^\infty}+\|\rho(t)\|_{W^{1,3}}
    \lesssim\|\la x,v\ra^4\gamma(t)\|_{L^\infty}^2+\|\nabla_{x,v}\gamma(t) \|_{L^2\cap L^\infty}^2.
\end{equation}
Indeed, the $L^1$ and $L^\infty$ bounds are immediate, while
\[
    \nabla_x\rho(t,x)=2\int_{\R^3}\gamma(t,x-tv,v)\nabla_x\gamma(t,x-tv,v)\,dv
\]
is bounded in $L^1$ by Cauchy--Schwarz and in $L^\infty$ by the weighted $L^\infty$ norm. Interpolation gives the $L^3$ estimate. Therefore,
\eqref{eq:short_range_fractional_identity} gives
\begin{equation}\label{eq:short_range_local_high_phi}
    \|\phi(t)\|_{W^{1,\infty}}+\|\nabla^2\phi(t)\|_{L^3}
    \lesssim_{\alpha,\lambda}\|\rho(t)\|_{L^1\cap L^\infty}+\|\rho(t)\|_{W^{1,3}}.
\end{equation}
Consequently,
\begin{align}\label{eq:short_range_local_H2_product}
    \|\phi(t)w(t)\|_{H^2}\lesssim\Big(\|\phi(t)\|_{W^{1,\infty}}
    +\|\nabla^2\phi(t)\|_{L^3}\Big)\|w(t)\|_{H^2}.
\end{align}
We spell out the iteration estimates. Let
\[
    \|w\|_{\mathcal S_T^m}:=\|w\|_{L^\infty_tH^m_x([0,T])}+\|w\|_{L^{8/3}_t W^{m,4}_x([0,T])}, \qquad m=1,2,
\]
and let $\mathcal X^1$ denote the weighted first-order characteristic norm in Proposition \ref{prop:LWP}:
\begin{equation*}
    \begin{split}
        \|g\|_{\mathcal{X}^1}:=\|\la x,v\ra^4g\|_{L^{\oo}_{x,v}}+\|\nabla_{x,v}g\|_{L^2_{x,v}\cap L^{\oo}_{x,v}}
    \end{split}
\end{equation*}
Suppose that
\[
    \|u_n\|_{\mathcal S_T^2}+\|\gamma_n\|_{L^\infty_t\mathcal X^1([0,T])}\leq M.
\]
Using \eqref{eq:short_range_local_H2_product} in the Duhamel formula gives
\begin{equation}\label{eq:short_range_local_high_uniform_u}
    \|u_{n+1}\|_{\mathcal S_T^2}\lesssim \|u_0\|_{H^2}+C_{\alpha,\lambda}TM^3.
\end{equation}
The differentiated Hamiltonian transport equation, together with
\eqref{eq:short_range_local_high_force_integral}, gives
\begin{equation}\label{eq:short_range_local_high_uniform_gamma}
    \|\gamma_{n+1}\|_{L^\infty_t\mathcal X^1([0,T])}\lesssim \|f_0\|_{\mathcal X^1} \exp\big(C_{\alpha,\lambda}(T+T^{1/4})M^2\big).
\end{equation}
Thus the iteration preserves a fixed ball after first choosing $M$ in terms of the initial norm and then taking $T$ sufficiently small.

The contraction is taken in the lower metric
\[
    d_T((w,g),(\widetilde w,\widetilde g)) :=\|w-\widetilde w\|_{\mathcal S_T^1} +\|g-\widetilde g\|_{L^\infty_t(L^2_{x,v}\cap L^\infty_{x,v})}.
\]
The near--far decomposition for $\nabla V_\alpha$ gives
\begin{align}\label{eq:short_range_local_high_force_difference}
    \|F_w-F_{\widetilde w}\|_{L^1_tL^\infty_x([0,T])} \lesssim_{\alpha,\lambda}(T+T^{1/4})\big(\|w\|_{\mathcal S_T^1}+\|\widetilde w\|_{\mathcal S_T^1}\big)\|w-\widetilde w\|_{\mathcal S_T^1}.
\end{align}
Indeed, the far part is estimated in $L^1_x$, while the near part is estimated in $L^\infty_x$ using $W^{1,4}\hookrightarrow L^\infty$. Likewise, \eqref{eq:short_range_local_phi} and the density difference estimate from the proof of Proposition \ref{prop:LWP} give
\begin{equation}\label{eq:short_range_local_high_phi_difference}
    \|\phi_g-\phi_{\widetilde g}\|_{W^{1,\infty}}\lesssim_{M,\alpha,\lambda}
    \|g-\widetilde g\|_{L^2_{x,v}\cap L^\infty_{x,v}}.
\end{equation}
Subtracting the Schr\"odinger iterates and using Strichartz therefore yields
\begin{align}\label{eq:short_range_local_high_contraction_u}
    \|u_{n+1}-u_n\|_{\mathcal S_T^1} \lesssim_{M,\alpha,\lambda}T
    \Big(&\|u_n-u_{n-1}\|_{\mathcal S_T^1}+\|\gamma_n-\gamma_{n-1}\|_{L^\infty_t(L^2\cap L^\infty)}\Big).
\end{align}
Subtracting the Hamiltonian transport equations and using \eqref{eq:short_range_local_high_force_difference} gives
\begin{equation}\label{eq:short_range_local_high_contraction_gamma}
    \|\gamma_{n+1}-\gamma_n\|_{L^\infty_t(L^2\cap L^\infty)}
    \lesssim_{M,\alpha,\lambda}(T+T^{1/4})\|u_n-u_{n-1}\|_{\mathcal S_T^1}.
\end{equation}
After decreasing $T$, the iteration is contractive. The uniform $\mathcal S_T^2\times L^\infty_t\mathcal X^1$ bounds pass to the limit, and standard persistence of regularity gives the stated strong solution and continuation criterion.

For sufficiently small data, the same estimates close on $[0,1]$. On this interval, \eqref{eq:short_range_local_high_phi} also controls the commuted
Schr\"odinger equations. Commuting once and twice with $G$, and using $G_ju(0)=ix_ju_0$ and $G_jG_ku(0)=-x_jx_ku_0$, gives
\eqref{eq:short_range_initial_interval}. In the high range, the same inhomogeneous Strichartz estimate also gives, for every admissible pair,
\begin{equation}\label{eq:short_range_initial_G2_strichartz}
    \|G^2u\|_{L^q_tL^r_x([0,1])}
    \lesssim_{\alpha,\lambda}\varepsilon_0.
\end{equation}
\end{proof}


\subsection{Decay of the fields}
For $T\geq1$, define
\begin{align}\label{eq:short_range_bootstrap_norm}
    \mathcal N_T:=\sup_{1\leq t\leq T}\Big(\|u(t)\|_{H^{n_\alpha}}+\|Gu(t)\|_{L^2}+\|G^2u(t)\|_{L^2}+\|\la x,v\ra^4\gamma(t)\|_{L^\infty_{x,v}} +\|\nabla_{x,v} \gamma(t)\|_{L^2_{x,v}\cap L^\infty_{x,v}}
    \Big).
\end{align}
\begin{lemma}\label{lem:short_range_Vlasov_fields}
Assume that $\mathcal N_T\leq A$, where $A\geq\varepsilon_0$. Then, for every $1\leq t\leq T$ and every $1<\alpha<2$,
\begin{align}
    \|\phi(t)\|_{L^\infty} &\lesssim_{\alpha,\lambda}A^2t^{-\alpha},
    \label{eq:short_range_phi_decay}\\
    \|E(t)\|_{L^\infty} &\lesssim_{\alpha,\lambda}A^2t^{-\alpha-1},
    \label{eq:short_range_E_decay}\\
    \|\nabla^2\phi(t)\|_{L^3} &\lesssim_{\alpha,\lambda}A^2t^{-\alpha-1}.
    \label{eq:short_range_nabla2_phi_decay}
\end{align}
\end{lemma}
\begin{proof}
Choose a smooth radial function $\chi$ supported in an annulus and a constant $c_\alpha$ such that
\[
    |z|^{-\alpha} =c_\alpha\int_0^\infty \chi\left(\frac zR\right)\frac{dR}{R^{1+\alpha}}.
\]
Define $\phi_R$ and $E_R$ as in Section \ref{sec:GWP}. The argument of Lemma \ref{lem:effective_field} gives, for $t\geq1$,
\begin{equation}\label{eq:short_range_shell_bound}
    |\phi_R(t,x)|+|E_R(t,x)| \lesssim A^2\min\left\{1,\frac{R^3}{t^3}\right\}.
\end{equation}
The shell representations of $V_\alpha$ and $\nabla V_\alpha$ therefore give
\begin{align*}
    |\phi(t,x)| &\lesssim_{\alpha,\lambda}A^2 \int_0^\infty\min\left\{1,\frac{R^3}{t^3}\right\}
    \frac{dR}{R^{1+\alpha}} \lesssim_{\alpha,\lambda}A^2t^{-\alpha},\\
    |E(t,x)| &\lesssim_{\alpha,\lambda}A^2 \int_0^\infty\min\left\{1,\frac{R^3}{t^3}\right\}
    \frac{dR}{R^{2+\alpha}} \lesssim_{\alpha,\lambda}A^2t^{-\alpha-1}.
\end{align*}
The second integral is finite at the origin precisely when $\alpha<2$. It remains to estimate the differentiated potential. The weighted bound for $\gamma$ and conservation of its $L^2$ norm imply
\begin{equation}\label{eq:short_range_rho_bounds}
    \|\rho(t)\|_{L^1}\lesssim A^2, \qquad
    \|\rho(t)\|_{L^\infty}\lesssim A^2t^{-3}, \qquad
    \|\rho(t)\|_{L^3}\lesssim A^2t^{-2}.
\end{equation}
Moreover,
\[
    \partial_{v_j}\big[\gamma(t,x-tv,v)\big] =-t(\partial_{x_j}\gamma)(t,x-tv,v)+(\partial_{v_j}\gamma)(t,x-tv,v).
\]
Multiplying by $2\gamma(t,x-tv,v)$ and integrating in $v$, the total $v$-derivative vanishes and gives
\begin{equation}\label{eq:short_range_derivative_rho_identity}
    \partial_{x_j}\rho(t,x) =\frac{2}{t}\int_{\R^3} \gamma(t,x-tv,v)(\partial_{v_j}\gamma)(t,x-tv,v)\,dv.
\end{equation}
Consequently,
\[
    \|\nabla\rho(t)\|_{L^1}\lesssim A^2t^{-1}, \qquad \|\nabla\rho(t)\|_{L^\infty}\lesssim A^2t^{-4},
\]
and hence
\begin{equation}\label{eq:short_range_derivative_rho}
    \|\nabla\rho(t)\|_{L^3}\lesssim A^2t^{-3}.
\end{equation}
Using \eqref{eq:short_range_fractional_identity}, boundedness of the Riesz
transforms on $L^3$, and fractional interpolation, we conclude that
\begin{align*}
    \|\nabla^2\phi(t)\|_{L^3} &\lesssim_{\alpha,\lambda}
    \big\||\nabla|^s\rho(t)\big\|_{L^3}\lesssim_{\alpha,\lambda} \|\rho(t)\|_{L^3}^{1-s}
    \|\nabla\rho(t)\|_{L^3}^{s} \lesssim_{\alpha,\lambda}A^2t^{-2-s} =A^2t^{-\alpha-1}.
\end{align*}
\end{proof}

For short-range $\alpha$, the identities following \eqref{eq:g} give
\begin{equation}\label{eq:short_range_g_H2}
    \|g(t)\|_{H^2} \lesssim \|u(t)\|_{L^2} +\|Gu(t)\|_{L^2}+\|G^2u(t)\|_{L^2} \lesssim A.
\end{equation}
\begin{lemma}\label{lem:short_range_bosonic_force}
Assume that $\mathcal N_T\leq A$. Then, for every $1<\alpha<2$ and $1\leq t\leq T$,
\begin{equation}\label{eq:short_range_F_decay}
    \|F(t)\|_{L^\infty} \lesssim_{\alpha,\lambda}A^2t^{-\alpha-1}.
\end{equation}
If $1<\alpha<\frac{3}{2}$, then also
\begin{equation}\label{eq:short_range_nabla_F_pointwise}
    \|\nabla F(t)\|_{L^\infty} \lesssim_{\alpha,\lambda}A^2t^{-\alpha-2}.
\end{equation}
\end{lemma}
\begin{proof}
By homogeneity,
\begin{align}
    F(t,ty) &=-\lambda t^{-\alpha-1} \left(\nabla|\cdot|^{-\alpha}*|g(t)|^2\right)(y),
    \label{eq:short_range_F_scaling}\\
    \partial_kF_j(t,ty) &=c_{\alpha,\lambda}t^{-\alpha-2} \mathcal R_j\mathcal R_k|\nabla|^s\big(|g(t)|^2\big)(y).
    \label{eq:short_range_nabla_F_scaling}
\end{align}
Since $\alpha<2$, the kernel $|z|^{-1-\alpha}$ is locally integrable.
Splitting the first convolution into $|z|\leq1$ and $|z|>1$ gives
\[
    \big\|\nabla|\cdot|^{-\alpha}*|g|^2\big\|_{L^\infty} \lesssim_\alpha \|g\|_{L^\infty}^2+\|g\|_{L^2}^2\lesssim_\alpha\|g\|_{H^2}^2.
\]
Together with \eqref{eq:short_range_g_H2}, this proves
\eqref{eq:short_range_F_decay}. Suppose now that $1<\alpha<\frac{3}{2}$. Then $2-s>\frac{3}{2}$, and $H^2(\R^3)$ is an algebra. Hence
\begin{align*}
    \big\|\mathcal R_j\mathcal R_k|\nabla|^s(|g|^2)\big\|_{L^\infty} &\lesssim_\alpha
    \big\|\mathcal R_j\mathcal R_k|\nabla|^s(|g|^2) \big\|_{H^{2-s}}\lesssim_\alpha \big\||g|^2\big\|_{H^2}\lesssim \|g\|_{H^2}^2.
\end{align*}
Using \eqref{eq:short_range_nabla_F_scaling} proves \eqref{eq:short_range_nabla_F_pointwise}.
\end{proof}

For $\frac{3}{2}\leq\alpha<2$, the endpoint Sobolev embedding used in \eqref{eq:short_range_nabla_F_pointwise} is unavailable. The transport estimates do not require a pointwise decay statement; they only require the time integral of $t^2\|\nabla F(t)\|_{L^\infty}$. We obtain exactly this quantity from a Strichartz norm of $G^2u$.
\begin{lemma}\label{lem:short_range_integrated_lipschitz}
Let $\frac{3}{2}\leq\alpha<2$. Choose $r=r_\alpha$ such that
\begin{equation}\label{eq:short_range_high_r_choice}
    \frac{3}{3-\alpha}<r<3,
\end{equation}
and let $q=q_\alpha$ be determined by
\begin{equation}\label{eq:short_range_high_admissible}
    \frac2q+\frac3r=\frac{3}{2}.
\end{equation}
Then $2<r<3$ and $q>4$. If $\mathcal N_T\leq A$, then
\begin{align}\label{eq:short_range_integrated_lipschitz_bound}
    \int_1^T t^2\|\nabla F(t)\|_{L^\infty}\,dt \lesssim_{\alpha,\lambda}A^2+A\|G^2u\|_{L^q_tL^r_x([1,T])}.
\end{align}
\end{lemma}
\begin{proof}
By \eqref{eq:short_range_nabla_F_scaling}, Sobolev embedding, and
\eqref{eq:short_range_high_r_choice},
\begin{align}\label{eq:short_range_high_multiplier}
    \big\|\mathcal R_j\mathcal R_k|\nabla|^s(|g|^2)\big\|_{L^\infty}&\lesssim_{\alpha,\lambda}
    \big\|\mathcal R_j\mathcal R_k|\nabla|^s(|g|^2)\big\|_{W^{2-s,r}}\lesssim_{\alpha,\lambda}\big\||g|^2\big\|_{W^{2,r}}.
\end{align}
Since $2<r<3$, one has $2r<6$. The Sobolev embedding
$H^2(\R^3)\hookrightarrow L^\infty\cap W^{1,2r}$ and the Leibniz rule give
\begin{align}\label{eq:short_range_product_W2r}
    \big\||g|^2\big\|_{W^{2,r}} &\lesssim \|g\|_{L^\infty}\|\nabla^2g\|_{L^r}
    +\|\nabla g\|_{L^{2r}}^2+\|g\|_{H^2}^2\lesssim A\|\nabla^2g\|_{L^r}+A^2.
\end{align}
The definition of $g$ and the identity for the second Galilean derivative imply
\begin{equation}\label{eq:short_range_g_second_scaling}
    \|\nabla_y^2g(t)\|_{L^r_y} \lesssim t^{\frac{3}{2}-\frac3r}\|G^2u(t)\|_{L^r_x} =t^{\frac2q}\|G^2u(t)\|_{L^r_x}.
\end{equation}
Combining \eqref{eq:short_range_nabla_F_scaling}--\eqref{eq:short_range_g_second_scaling}, we obtain
\begin{equation}\label{eq:short_range_high_pointwise_integrand}
    t^2\|\nabla F(t)\|_{L^\infty}\lesssim_{\alpha,\lambda}
    A^2t^{-\alpha}+A t^{-\alpha+\frac2q}\|G^2u(t)\|_{L^r}.
\end{equation}
Since $q>4$ and $\alpha\geq\frac{3}{2}$,
\[
    \alpha>1+\frac1q.
\]
Equivalently,
$t^{-\alpha+2/q}\in L^{q'}([1,\infty))$. H\"older's inequality in time
therefore gives
\begin{align*}
    \int_1^Tt^2\|\nabla F(t)\|_{L^\infty}\,dt&\lesssim_{\alpha,\lambda}A^2
    +A\big\|t^{-\alpha+2/q}\big\|_{L^{q'}([1,\infty))}\|G^2u\|_{L^q_tL^r_x([1,T])},
\end{align*}
which proves \eqref{eq:short_range_integrated_lipschitz_bound}.
\end{proof}


\subsection{Global bounds}
In this subsection, we prove and record the decay estimates needed to show linear scattering.
\begin{proposition}\label{prop:short_range_global}
Under the assumptions of Theorem \ref{thm:short_range}, the solution is global and satisfies
\begin{align}\label{eq:short_range_global_bound}
    \sup_{t\geq0}\Big(\|u(t)\|_{H^{n_\alpha}} +\|Gu(t)\|_{L^2}+\|G^2u(t)\|_{L^2}+\|\la x,v\ra^4\gamma(t)\|_{L^\infty_{x,v}} +\|\nabla_{x,v}\gamma(t)\|_{L^2_{x,v}\cap L^\infty_{x,v}}\Big) \lesssim_{\alpha,\lambda}\varepsilon_0.
\end{align}
In the range $\frac{3}{2}\leq\alpha<2$, with $(q,r)$ chosen in \eqref{eq:short_range_high_r_choice}--\eqref{eq:short_range_high_admissible}, one also has
\begin{equation}\label{eq:short_range_global_G2_strichartz}
    \|G^2u\|_{L^q_tL^r_x([1,\infty))} \lesssim_{\alpha,\lambda}\varepsilon_0.
\end{equation}
\end{proposition}
\begin{proof}
In view of Lemma \ref{lem:short_range_local}, it suffices to work on $[1,T]$. We close a bootstrap for \eqref{eq:short_range_bootstrap_norm}; in the high range $\frac{3}{2}\leq\alpha <2$, we include additionally
\begin{equation}\label{eq:short_range_high_bootstrap_strichartz}
    \|G^2u\|_{L^q_tL^r_x([1,T])}\leq \varepsilon.
\end{equation}
Assume these bounds hold for a sufficiently small $\varepsilon$. We first control the Schr\"odinger component. Conservation of bosonic mass and the energy estimate for the differentiated equation give
\begin{equation}\label{eq:short_range_H1_u}
    \|\nabla u(t)\|_{L^2} \leq \|\nabla u(1)\|_{L^2} +\|u_0\|_{L^2}\int_1^t\|E(\tau)\|_{L^\infty}\,d\tau.
\end{equation}
Commuting once with $G$ similarly gives
\begin{equation}\label{eq:short_range_G_u}
    \|Gu(t)\|_{L^2} \leq \|Gu(1)\|_{L^2} +C\|u_0\|_{L^2}\int_1^t \tau\|E(\tau)\|_{L^\infty} \,d\tau.
\end{equation}
For the twice-commuted equation,
\begin{align}\label{eq:short_range_G2_equation}
    (i\partial_t-\Delta)G_jG_ku=\phi G_jG_ku+2t(\partial_j\phi)G_ku+2t(\partial_k\phi) G_ju+4t^2(\partial_{jk}\phi)u.
\end{align}
Since $\phi$ is real, its first term does not contribute to the $L^2$ energy. Thus
\begin{align}\label{eq:short_range_G2_energy}
    \partial_t\|G^2u(t)\|_{L^2} \lesssim t\|E(t)\|_{L^\infty}\|Gu(t)\|_{L^2}+t^2\|\nabla^2\phi(t)\|_{L^3}\|u(t)\|_{L^6}.
\end{align}
Lemma \ref{lem:short_range_Vlasov_fields} therefore gives
\begin{equation}\label{eq:short_range_G2_integrand}
    \partial_t\|G^2u(t)\|_{L^2} \lesssim_{\alpha,\lambda}\varepsilon^3t^{-\alpha}.
\end{equation}
Since $\alpha>1$, \eqref{eq:short_range_H1_u}--\eqref{eq:short_range_G2_integrand} imply
\begin{equation}\label{eq:short_range_commuted_improvement}
    \sup_{1\leq t\leq T} \big(\|u(t)\|_{H^1}+\|Gu(t)\|_{L^2}+\|G^2u(t)\|_{L^2}\big)
    \lesssim_{\alpha,\lambda}\varepsilon_0+\varepsilon^3.
\end{equation}
In the high range, we also propagate the $H^2$ norm. By the Leibniz rule,
\begin{equation}\label{eq:short_range_H2_nonlinearity}
    \begin{split}
    &\|\phi(t)u(t)\|_{H^2} \\
    &\qquad \lesssim\|\phi(t)\|_{L^\infty}\|u(t)\|_{H^2}
    +\|E(t)\|_{L^\infty}\|u(t)\|_{H^1}+\|\nabla^2\phi(t)\|_{L^3}\|u(t)\|_{L^6}
    \lesssim_{\alpha,\lambda}\varepsilon^3t^{-\alpha}.
    \end{split}
\end{equation}
Hence
\begin{equation}\label{eq:short_range_H2_improvement}
    \sup_{1\leq t\leq T}\|u(t)\|_{H^2} \lesssim_{\alpha,\lambda}\varepsilon_0+\varepsilon^3.
\end{equation}
Applying the inhomogeneous Strichartz estimate with the pair $(q,r)$ to \eqref{eq:short_range_G2_equation}, we obtain
\begin{align*}
    \|G^2u\|_{L^\infty_tL^2_x\cap L^q_tL^r_x([1,T])}\lesssim{}&\|G^2u(1)\|_{L^2}+\|\phi G^2u\|_{L^1_tL^2_x}+\|tEGu\|_{L^1_tL^2_x}+\|t^2(\nabla^2\phi)u\|_{L^1_tL^2_x}.
\end{align*}
Every term on the right is bounded by $C_{\alpha,\lambda}(\varepsilon_0+\varepsilon^3)$, because \eqref{eq:short_range_phi_decay}--\eqref{eq:short_range_nabla2_phi_decay} give the integrable weight $t^{-\alpha}$. Therefore,
\begin{equation}\label{eq:short_range_G2_strichartz_improvement}
    \|G^2u\|_{L^q_tL^r_x([1,T])}\lesssim_{\alpha,\lambda}\varepsilon_0+\varepsilon^3.
\end{equation}
It remains to propagate the Vlasov norms. The weighted transport equation gives
\begin{align}\label{eq:short_range_weight_transport}
    \|\la x,v\ra^4\gamma(t)\|_{L^\infty} \lesssim\|\la x,v\ra^4\gamma(1)\| _{L^\infty}+\int_1^t(1+\tau)\|F(\tau)\|_{L^\infty}\|\la x,v\ra^4\gamma(\tau)\| _{L^\infty}\,d\tau.
\end{align}
Likewise, differentiating the transport equation and applying Lemma \ref{lem:hamiltonian_transport} gives, for $p=2,\infty$,
\begin{align}\label{eq:short_range_derivative_transport}
    \|\nabla_{x,v}\gamma(t)\|_{L^p}\lesssim\|\nabla_{x,v}\gamma(1)\|_{L^p}+\int_1^t(1+\tau)^2 \|\nabla F(\tau)\|_{L^\infty} \|\nabla_{x,v}\gamma(\tau)\|_{L^p}\,d\tau.
\end{align}
By \eqref{eq:short_range_F_decay},
\begin{equation}\label{eq:short_range_F_integrability}
    \int_1^\infty(1+t)\|F(t)\|_{L^\infty}\,dt \lesssim_{\alpha,\lambda}\varepsilon^2.
\end{equation}
If $1<\alpha<\frac{3}{2}$, then
\eqref{eq:short_range_nabla_F_pointwise} gives
\[
    \int_1^\infty(1+t)^2\|\nabla F(t)\|_{L^\infty}\,dt\lesssim_{\alpha,\lambda}\varepsilon^2.
\]
If $\frac{3}{2}\leq\alpha<2$, Lemma \ref{lem:short_range_integrated_lipschitz} and
\eqref{eq:short_range_G2_strichartz_improvement} give the same conclusion:
\begin{equation}\label{eq:short_range_high_integrability}
    \int_1^T(1+t)^2\|\nabla F(t)\|_{L^\infty}\,dt \lesssim_{\alpha,\lambda}\varepsilon^2+\varepsilon(\varepsilon_0+\varepsilon^3)
    \lesssim_{\alpha,\lambda}\varepsilon^2.
\end{equation}
Gronwall's inequality applied to \eqref{eq:short_range_weight_transport}--
\eqref{eq:short_range_derivative_transport} therefore yields
\begin{align}\label{eq:short_range_transport_improvement}
    \sup_{1\leq t\leq T}\Big(\|\la x,v\ra^4\gamma(t)\|_{L^\infty}
    +\|\nabla_{x,v}\gamma(t)\|_{L^2\cap L^\infty}\Big)\lesssim_{\alpha,\lambda}
    \varepsilon_0e^{C_{\alpha,\lambda}\varepsilon^2}.
\end{align}
Taking $\varepsilon_0\ll\varepsilon\ll 1$, the estimates above improve all bootstrap
bounds. The standard continuity argument, followed by the continuation criterion in Lemma \ref{lem:short_range_local}, proves global existence, \eqref{eq:short_range_global_bound}, and
\eqref{eq:short_range_global_G2_strichartz}.
\end{proof}


\subsection{Linear scattering}
We now establish the linear scattering stated in Theorem \ref{thm:short_range}. 
\begin{proof}[Proof of Theorem \ref{thm:short_range}]
Recall
\begin{equation}\label{eq:short_range_gamma_time}
    \partial_t\gamma(t,x,v) =t\nabla_x\gamma(t,x,v)\cdot F(t,x+tv) -\nabla_v\gamma(t,x,v)\cdot F(t,x+tv).
\end{equation}
Proposition \ref{prop:short_range_global} and
\eqref{eq:short_range_F_decay} imply, for $p=2,\infty$ and $t\geq1$,
\[
    \|\partial_t\gamma(t)\|_{L^p_{x,v}} \lesssim_{\alpha,\lambda}\varepsilon_0^3t^{-\alpha}.
\]
Since $\alpha>1$, there exists $f_+\in L^2_{x,v}\cap L^\infty_{x,v}$ such that
\begin{equation}\label{eq:short_range_scattering_gamma}
    \|\gamma(t)-f_+\|_{L^2_{x,v}\cap L^\infty_{x,v}}\lesssim_{\alpha,\lambda}\varepsilon_0^3t^{1-\alpha}.
\end{equation}
Recalling the definition of $\gamma$ proves
\eqref{eq:short_range_scattering_f}. For the Schr\"odinger component, if $1<\alpha<\frac{3}{2}$, then
\begin{align*}
    \|\phi(t)u(t)\|_{H^1}&\lesssim \|\phi(t)\|_{L^\infty}\|u(t)\|_{H^1}+\|E(t)\|_{L^\infty}\|u(t)\|_{L^2}\lesssim_{\alpha,\lambda}\varepsilon_0^3t^{-\alpha}.
\end{align*}
If $\frac{3}{2}\leq\alpha<2$, then
\eqref{eq:short_range_H2_nonlinearity} gives
\[
    \|\phi(t)u(t)\|_{H^2}\lesssim_{\alpha,\lambda}\varepsilon_0^3t^{-\alpha}.
\]
Thus, in both cases,
\begin{equation}\label{eq:short_range_nonlinearity_scattering}
    \|\phi(t)u(t)\|_{H^{n_\alpha}}\lesssim_{\alpha,\lambda}\varepsilon_0^3t^{-\alpha}.
\end{equation}
It follows that $e^{it\Delta}u(t)$ converges in $H^{n_\alpha}$. Setting
\[
    u_+:=u_0-i\int_0^\infty e^{is\Delta}\big(\phi(s)u(s)\big)\,ds,
\]
we obtain
\begin{align*}
    \|u(t)-e^{-it\Delta}u_+\|_{H^{n_\alpha}} &\leq \int_t^\infty \|\phi(s)u(s)\|_{H^{n_\alpha}}\,ds\,\lesssim_{\alpha,\lambda}\, \varepsilon_0^3t^{1-\alpha}.
\end{align*}
This proves \eqref{eq:short_range_scattering_u} and completes the proof.
\end{proof}

\begin{remark}\label{rem:short_range_threshold}
The condition $\alpha>1$ is exactly what makes the coefficients in the moment,
derivative, and Galilean estimates integrable in time. The threshold $\alpha=\frac{3}{2}$ is not a scattering threshold. It is only the endpoint of the $H^1$ estimate
\[
    H^{2-s}(\R^3)\hookrightarrow L^\infty(\R^3),
    \qquad s=\alpha-1<\frac{1}{2}.
\]
For $\frac{3}{2}\leq\alpha<2$, one additional derivative of $u$ gives the local Lipschitz characteristic flow, while the global Vlasov derivative estimate is closed by the integrated bound
\[
    \int_1^\infty t^2\|\nabla F(t)\|_{L^\infty}\,dt<\infty.
\]
The upper restriction $\alpha<2$ is the natural endpoint of this classical characteristic argument, since $\nabla V_\alpha\in L^1_{\rm loc}(\R^3)$ exactly for $\alpha<2$.
\end{remark}




\begin{thebibliography}{99}
\bibitem{BardosDegond1985}
C.~Bardos and P.~Degond,
\textit{Global existence for the Vlasov--Poisson equation in 3 space variables with small initial data},
Ann. Inst. H. Poincar\'e Anal. Non Lin\'eaire
\textbf{2} (1985), no.~2, 101--118.

\bibitem{BV26}
L. Bigorgne and R. Velozo Ruiz,
\href{https://doi.org/10.1088/1361-6544/ae55f5}
{\it Late-time asymptotics of small data solutions for the Vlasov--Poisson system,}
Nonlinearity {\bf 39} (2026), no. 4, Paper No. 045007.

\bibitem{CMMP25}
E. C\'ardenas, J. ~K. Miller, D. Mitrouskas and N. Pavlovi\'c,
\href{https://arxiv.org/abs/2502.18678}
{\it Emergence of fermion-mediated interactions in Bose--Fermi mixtures,}
arXiv:2502.18678.

\bibitem{CKP25}
E. C\'ardenas, J.~K. Miller and N. Pavlovi\'c,
\href{https://arxiv.org/abs/2309.04638}
{\it On the effective dynamics of Bose--Fermi mixtures,}
Ars Inven. Anal. { 2025}, Paper No. 5, 54 pp.

\bibitem{Cav26}
K. Cavanagh,
\href{https://arxiv.org/abs/2607.04444}
{\it Global existence and time decay for the Vlasov--Hartree system,}
arXiv:2607.04444.

\bibitem{ChoiKwon2016}
S.-H.~Choi and S.~Kwon,
\textit{Modified scattering for the Vlasov--Poisson system},
Nonlinearity
\textbf{29} (2016), 2755--2774.

\bibitem{FOPW23}
P. Flynn, Z. Ouyang, B. Pausader and K. Widmayer,
\href{https://doi.org/10.1007/s42543-021-00041-x}
{\it Scattering map for the Vlasov--Poisson system,}
Peking Math. J. {\bf 6} (2023), no. 2, 365--392.

\bibitem{GO93}
J. Ginibre and T. Ozawa,
\href{https://doi.org/10.1007/BF02097031}
{\it Long range scattering for non-linear Schr\"odinger and Hartree
equations in space dimension $n\geq 2$,}
Comm. Math. Phys. {\bf 151} (1993), no. 3, 619--645.

\bibitem{GV00I}
J. Ginibre and G. Velo,
\href{https://doi.org/10.1142/S0129055X00000137}
{\it Long range scattering and modified wave operators for some Hartree type equations. I,}
Rev. Math. Phys. {\bf 12} (2000), no. 3, 361--429.

\bibitem{GV00II}
J. Ginibre and G. Velo,
\href{https://doi.org/10.1007/PL00001014}
{\it Long range scattering and modified wave operators for some Hartree type equations. II,}
Ann. Henri Poincar\'e {\bf 1} (2000), no. 4, 753--800.

\bibitem{Glassey}
R.~T. Glassey, {\it The Cauchy problem in kinetic theory}, SIAM, Philadelphia, PA, 1996; MR1379589

\bibitem{Grafakos}
L. Grafakos, {\it Classical Fourier analysis}, third edition, 
Graduate Texts in Mathematics, 249, Springer, New York, 2014; MR3243734

\bibitem{HN98}
N. Hayashi and P.~I. Naumkin,
\href{https://doi.org/10.1353/ajm.1998.0011}
{\it Asymptotics for large time of solutions to the nonlinear
Schr\"odinger and Hartree equations,}
Amer. J. Math. {\bf 120} (1998), no. 2, 369--389.

\bibitem{HNO98}
N. Hayashi, P.~I. Naumkin and T. Ozawa,
\href{https://doi.org/10.1137/S0036141096312222}
{\it Scattering theory for the Hartree equation,}
SIAM J. Math. Anal. {\bf 29} (1998), no. 5, 1256--1267.

\bibitem{HN01}
N. Hayashi and P.~I. Naumkin,
{\it Scattering theory and large time asymptotics of solutions to the Hartree
type equations with a long range potential,}
Hokkaido Math. J. {\bf 30} (2001), no. 1, 137--161.

\bibitem{HP26}
Y. Hong and S. Pankavich,
\href{https://arxiv.org/abs/2604.04256}
{\it Modified scattering for the Vlasov--Riesz system with long-range interactions,}
arXiv:2604.04256.

\bibitem{HK24}
W. Huang and H. Kwon,
\href{https://arxiv.org/abs/2407.16919}
{\it Scattering of the Vlasov--Riesz system in the three dimensions,}
Ann. Henri Poincar\'e (2026)

\bibitem{HK26}
W. Huang and H. Kwon,
\href{https://arxiv.org/abs/2602.21344}
{\it Scattering map for the Vlasov--Poisson system with a repulsive
harmonic potential,}
arXiv:2602.21344.

\bibitem{HPS24}
W. Huang, B. Pausader and M. Suzuki,
\href{https://arxiv.org/abs/2412.13434}
{\it The Vlasov--Poisson system with a perfectly conducting wall:
Convex domains,} Comm. Math. Phys., to appear
arXiv:2412.13434.

\bibitem{IacobelliRossiWidmayer2026}
M.~Iacobelli, S.~Rossi, and K.~Widmayer,
\textit{On the stability of vacuum in the screened Vlasov--Poisson equation},
Journal of the London Mathematical Society
\textbf{113} (2026), no.~1, e70426.

\bibitem{IPWW22}
A.~D. Ionescu, B. Pausader, X. Wang and K. Widmayer,
\href{https://doi.org/10.1093/imrn/rnab155}
{\it On the asymptotic behavior of solutions to the Vlasov--Poisson system,}
Int. Math. Res. Not. IMRN {\bf 2022} (2022), no. 12, 8865--8889.

\bibitem{KP11}
J.~Kato and F.~Pusateri,
\textit{A new proof of long-range scattering for critical nonlinear Schr\"odinger equations},
Differential Integral Equations
\textbf{24} (2011), no.~9--10, 923--940.

\bibitem{KT}
M. Keel and T.~C. Tao, Endpoint Strichartz estimates, Amer. J. Math. {\bf 120} (1998), no.~5, 955--980.

\bibitem{KW25}
B. Kepka and K. Widmayer,
\href{https://arxiv.org/abs/2511.04363}
{\it Modified scattering dynamics in the Vlasov--Poisson equation near an attractive point mass,}
arXiv:2511.04363.

\bibitem{LMR05}
M.~Lemou, F.~M\'ehats and P.~Rapha\"el,
\textit{Orbital stability and singularity formation for Vlasov--Poisson systems},
C. R. Acad. Sci. Paris, Ser. I
\textbf{341} (2005), 269--274.

\bibitem{LionsPerthame1991}
P.-L.~Lions and B.~Perthame,
\textit{Propagation of moments and regularity for the 3-dimensional Vlasov--Poisson system},
Invent. Math.
\textbf{105} (1991), 415--430.

\bibitem{NO92}
H. Nawa and T. Ozawa,
\href{https://doi.org/10.1007/BF02102628}
{\it Nonlinear scattering with nonlocal interaction,}
Comm. Math. Phys. {\bf 146} (1992), no. 2, 259--275.

\bibitem{ONeil}
R. O'Neil, Convolution operators and $L(p,\,q)$ spaces, Duke Math. J. {\bf 30} (1963), 129--142; MR0146673

\bibitem{Pau26}
B. Pausader,
\href{https://doi.org/10.1137/25M1805825}
{\it Stability problems in collisionless kinetic systems,}
in {\it Proceedings of the International Congress of Mathematicians 2026},
Vol.~5, Invited Lectures: Sections 9--11, 2026, 516--533.

\bibitem{PW21}
B. Pausader and K. Widmayer,
\href{https://link.springer.com/article/10.1007/s00220-021-04117-8}
{\it Stability of a point charge for the Vlasov--Poisson system:
The radial case,}
Comm. Math. Phys. {\bf 385} (2021), no. 3, 1741--1769.

\bibitem{PWY22}
B. Pausader, K. Widmayer and J. Yang,
\href{https://doi.org/10.4171/JEMS/1518}
{\it Stability of a point charge for the repulsive Vlasov--Poisson system,}
J. Eur. Math. Soc. {\bf 28} (2026), no. 7, 2751--2848.

\bibitem{PX}
B. Pausader and M. Xie,
{\it A Physical-Space Approach to Modified Scattering for the 3D Hartree Equation: Full-Space and Robin Half-Space,} in preparation.

\bibitem{Peetre}
J. Peetre, Espaces d'interpolation et th\'eor\`eme de Soboleff, Ann. Inst. Fourier (Grenoble) {\bf 16} (1966), fasc. 1, 279--317; MR0221282

\bibitem{Pfaffelmoser1992}
K.~Pfaffelmoser,
\textit{Global classical solutions of the Vlasov--Poisson system in three dimensions for general initial data},
J. Differential Equations
\textbf{95} (1992), 281--303.

\bibitem{SchlueTaylor2025}
V.~Schlue and M.~Taylor,
\textit{Inverse modified scattering and polyhomogeneous expansions for the Vlasov--Poisson system},
Nonlinearity
\textbf{38} (2025), no.~9, 095019.

\bibitem{Stein}
E.~M. Stein, {\it Singular integrals and differentiability properties of functions}, Princeton Mathematical Series, No. 30, Princeton Univ. Press, Princeton, NJ, 1970;

\bibitem{Tao09}
T.~Tao,
\textit{A pseudoconformal compactification of the nonlinear Schr\"odinger equation and applications},
New York J. Math. \textbf{15} (2009), 265--282.

\bibitem{VanHoose2024}
T.~Van Hoose,
\textit{Modified scattering for the Hartree nonlinear Schr\"odinger equation}, arXiv:2407.18411 (2024).
\end{thebibliography}
\end{document}